\documentclass[11pt]{article}
\usepackage{amsfonts}
\usepackage{amsthm,amssymb,amsmath,amscd}
\usepackage{wasysym}
\usepackage[all,cmtip]{xy}
\usepackage{mathrsfs}
\usepackage{mathtools}
\usepackage{txfonts}
\usepackage{url}
\usepackage{lscape}
\usepackage{authblk}
\usepackage[british]{babel}

\DeclareMathAlphabet{\mathcal}{OMS}{cmsy}{m}{n}
\DeclareSymbolFont{largesymbols}{OMX}{cmex}{m}{n}

\objectmargin{1pt}

\newtheorem{thm}{Theorem}[section]
\newtheorem{defn}[thm]{Definition}
\newtheorem{lem}[thm]{Lemma}

\newtheorem{rem}[thm]{Remark}
\newtheorem{ex}[thm]{Example}

\newcommand{\defCategory}[2]{
  \newcommand{#1}{#2\defvariable}
  }

\newcommand{\defvariable}[2][]{
\if\relax\detokenize{#1}\relax  %if the first arg is empty
   \if\relax\detokenize{#2}\relax
    \else
    ({#2})
    \fi
\else
  ^{{#1}}({#2})
\fi
 }

\newcommand{\lra}{\longrightarrow}

\newcommand{\xra}{\xrightarrow}
\newcommand{\dra}{\dashrightarrow}
\newcommand{\ra}{\rightarrow}

\newcommand{\gh}{{\rm gh\, }}
\newcommand{\Fgh}{{\rm Fgh\, }}
\newcommand{\cogh}{{\rm cogh\,}}
\newcommand{\Fcogh}{{\rm Fcogh\,}}

\newcommand{\add}{{\rm add\, }}

\newcommand{\Hom}{{\rm Hom \, }}
\newcommand{\End}{{\rm End\, }}

\newcommand{\Img}{{\rm Im}\,\,}

\newcommand{\Ker}{{\rm Ker}\,\,}
\newcommand{\Ext}{{\rm Ext}}

\newcommand{\E}{{\rm E}}
\newcommand{\Ve}{{\rm Vec}}
\newcommand{\dk}{{\rm dim_{_{K}}}}
\newcommand{\cpx}[1]{#1^{\bullet}}

\newcommand{\Dz}[1]{{\rm D}^+(#1)}
\newcommand{\Df}[1]{{\rm D}^-(#1)}
\newcommand{\Db}[1]{{\mathscr D}^b(#1)}
\newcommand{\C}[1]{{\mathscr C}(#1)}

\newcommand{\Cb}[1]{{\mathscr C}^b(#1)}
\newcommand{\K}[1]{{\mathscr K}(#1)}

\newcommand{\Kf}[1]{{\mathscr K}^-(#1)}
\newcommand{\Kb}[1]{{\mathscr K}^b(#1)}
\newcommand{\Modc}{\ensuremath{\mbox{{\rm -Mod}}}}
\newcommand{\modc}{\ensuremath{\mbox{{\rm -mod}}}}
\newcommand{\stmodc}[1]{#1\mbox{{\rm -{\underline{mod}}}}}
\newcommand{\Proj}[1]{#1\mbox{{\rm -Proj}}}
\newcommand{\proj}[1]{#1\mbox{{\rm -proj}}}

\newcommand{\thick}{{\rm thick}}
\newcommand{\otimesL}{\otimes^{\rm\bf L}}

\newcommand{\Oc}[3]{{#1}^{#2,#3}}
\DeclareMathOperator{\gldim}{gl.\!dim}

\newcommand{\Gr}[1]{#1\mbox{{\rm -Gr}}}
\newcommand{\gr}[1]{#1\mbox{{\rm -gr}}}
\newcommand{\perf}[1]{#1\mbox{{\rm -perf}}}
\newcommand{\Grperf}[1]{#1\mbox{{\rm -grperf}}}

\newcommand{\HomP}{{\rm Hom}^{\bullet}}

 \defCategory{\D}{\mathscr{D}}

\newcommand{\lla}{\longleftarrow}
\newcommand{\lraf}[1]{\stackrel{#1}{\lra}}
\newcommand{\llaf}[1]{\stackrel{#1}{\lla}}

\newcommand{\draf}[1]{\stackrel{#1}{\dra}}

\newcommand{\scr}[1]{\mathscr #1}
\newcommand{\al}[1]{\mathcal #1}

\newcommand{\proja}{\mathcal{P}_{\!\!\!\scriptscriptstyle{\mathcal{A}}}}
\newcommand{\projb}{\mathcal{P}_{\!\!\scriptscriptstyle{\mathcal{B}}}}

\newcommand{\sta}{\underline{\mathcal{A}}}
\newcommand{\stb}{\underline{\mathcal{B}}}

\newcommand{\diag}{{\rm diag}}

\begin{document}

\title{Symmetric approximation sequences, Beilinson-Green algebras and derived equivalences}

\author{\medskip Shengyong Pan \\

\small School of Mathematics and Statistics\\ Beijing Jiatong University\\
No. 3 Shangyuancun\\
Beijing, 100044, P. R. China\\
Email: shypan@bjtu.edu.cn}

\date{}

\maketitle

\begin{abstract}
In this paper, we will consider a class of locally $\Phi$-Beilinson-Green algebras, where $\Phi$ is an infinite
admissible set of the integers, and show that symmetric approximation sequences in $n$-exangulated categories give rise to derived equivalences between quotient algebras of locally $\Phi$-Beilinson-Green algebras in the principal diagonals modulo some factorizable ghost and coghost ideals by the locally finite tilting family. Then we get a class of derived equivalent algebras that have not been obtained by using previous techniques. From higher exact sequences, we obtain derived equivalences between subalgebras of endomorphism algebras by constructing tilting complexes, which generalizes Chen and Xi's result for exact sequences. From a given derived equivalence, we get derived equivalences between locally $\Phi$-Beilinson-Green algebras of semi-Gorenstein modules. Finally, from given graded derived equivalences of group graded algebras, we get derived equivalences between associated Beilinson-Green algebras of group graded algebras.
\end{abstract}

\renewcommand{\thefootnote}{\alph{footnote}}

\maketitle

\renewcommand{\thefootnote}{\alph{footnote}}
\setcounter{footnote}{-1} 
\renewcommand{\thefootnote}{\alph{footnote}}
\setcounter{footnote}{-1} 

\maketitle

\section{Introduction}\label{Intro}
In the representation theory of algebras, derived equivalences have been
shown to preserve many algebraic and geometric invariants and provide new connections.
However, in general, it is very difficult to describe the derived equivalence class of a given ring.
One idea is to study locally derived equivalences, that is, to establish some elementary derived equivalences between certain rings, and hope that derived equivalent rings can be related by a sequence of such elementary derived equivalences. Many such kind of derived equivalent rings comes from mutations sequences of objects in categories, where approximations play a central role. This occurs in many aspects in algebra and geometry, such as mutations of tilting modules \cite{Happel1998b}, mutations of cluster tilting objects \cite{Buan2006}, mutations of silting objects \cite{Aihara2012}, mutations of exceptional sequences \cite{Rudakov1990}, and mutations of modifying modules \cite{Iyama2014a} in the study of the NCCR conjecture (\cite[Conjecture 4.6]{Vandenbergh2004}).

It is interesting to know whether the mutation sequences always give derived equivalent endomorphism rings. If a mutation sequence is an Auslander-Reiten sequence over an Artin algebra, the endomorphism algebras are derived equivalent \cite{Hu2011}. And if this sequence is a triangle in algebraic triangulated category, the endomorphism algebras are also derived equivalent \cite{D}. However, this is not always true. For instance, the endomorphism rings of cluster tilting objects related by mutation sequences are not always derived equivalent. Also, in general, Auslander-Reiten triangles do not give rise to derived equivalences. In \cite{Hu2013a}, Hu, Koenig and Xi proved that certain
triangles with symmetric approximations and some conditions give rise to derived equivalences of quotient algebras of endomorphism algebras of objects in the sequences modulo some particularly defined ideals. After that Chen \cite{Chen2013} generalized their result to $n$-angles in $n$-angulated categories. Recently, Chen and Hu \cite{CH} introduced
symmetric approximation sequences in additive categories and weakly $n$-angulated categories which include (higher) Auslander-Reiten sequences (triangles) and showed that such sequences always give rise to derived equivalences between the quotient rings of endomorphism rings of objects in the sequences modulo some ghost and coghost ideals. Therefore, they unified the results in  \cite{Hu2013a} and \cite{Chen2013}.

In this paper, we first focus on higher exact sequences in additive categories, then we obtain derived equivalences between subalgebras of endomorphism algebras in the higher exact sequences. This result can be described as the follow theorem.

\begin{thm}(=Theorem 3.2)\label{pppp}
Let  $\al C$ be an additive category and $M$ be an object in $\al {C}$. Suppose that
$X\xra f M_1\xra {d_1} M_2\ra\cdots\xra {d_{n-1}} M_n\xra g Y$ is a higher $\add(M)$-exact sequence. Let
$$
A=\left\{\begin{pmatrix}
s_1&s_2 \\
fs_3 &s_4
\end{pmatrix}\in \begin{pmatrix}
\End_{\al C}(M) &\Hom_{\al C}(M,X) \\
\Hom_{\al C}(X,M) &\End_{\al C}(X)
\end{pmatrix}  \bigg| \
\begin{aligned}
s_3\in\Hom_{\al C}(M_1,M)& \; \text{ and there exists}\\
s_5 \in \End_{\al C}(M_1) & \; \text{such that}\; s_4f=fs_5.
\end{aligned}
\right\}
$$
and
$$
B=\left\{\begin{pmatrix}
s_1&s_2g \\
s_3 &s_4
\end{pmatrix}\in \begin{pmatrix}
\End_{\al C}(M) &\Hom_{\al C}(M,Y) \\
\Hom_{\al C}(Y,M) &\End_{\al C}(Y)
\end{pmatrix}  \bigg| \
\begin{aligned}
s_2\in\Hom_{\al C}(M,M_n)& \; \text{ and there exists}\\
 s_5 \in \End_{\al C}(M_n) & \; \text{such that}\; gs_4=s_5g.
\end{aligned}
\right\}
$$
be subrings of $\End_{\al C}(M\oplus X)$  and $\End_{\al C}(M\oplus Y)$, respectively. Then $A$ and $B$ are derived equivalent.
\end{thm}

We will investigate symmetric approximation sequences in $n$-exangulated categories which give rise to derived equivalences between quotient algebras of locally $\Phi$-Beilinson-Green algebras in the principal diagonals modulo some factorizable ghost and coghost ideals. For the definition of locally $\Phi$-Beilinson-Green algebras, we refer to Section 5. We construct locally finite tilting sets of complexes and show that the locally endomorphism rings of these complexes are isomorphic to quotient algebras of locally $\Phi$-Beilinson-Green algebras. Then we obtain
derived equivalences which have not been observed by using previous techniques. Consequently, we generalize the main result in \cite{PZ} to an $n$-exangulated category ($n\geq 3$) with an $n$-angle endo-functor $F$ and the locally $\Phi$-Beilinson-Green algebras. Note that a locally $\Phi$-Beilinson-Green algebra has a local unit, but it doesn't have a unit in general.
The result is the following theorem.

\begin{thm}(=Theorem \ref{Theorem-ghost})\label{ppp}
Let $(\al C,\mathbb{E},\mathfrak{s})$ be an $n$-exangulated category ($n\geq 1$) with an $n$-exangle endo-functor $F$,
and let $M$ be an object in $\mathcal {C}$. Suppose that $\Phi$ is an admissible set of $\mathbb{Z}$ and that $\mathcal {C}(M, F^i(X))=0=\mathcal {C}(Y,F^i(M))$ for all $0\neq i\in\Phi$. Let $X\xra f M_1\ra M_2\ra\cdots\ra M_n\xra g Y\draf{\delta}$
be an $n$-$\mathbb{E}$-exangle in $\mathcal {C}$ with $M_j\in\add(M)$ for all $j=1,\cdots,n$, such that $f$ is a left $\add_{\Oc{\al C}{F}{\Phi}}(M)$-approximation of $X$ and that $g$ is a right $\add_{\Oc{\al C}{F}{\Phi}}(M)$-approximation of $Y$ in $\Phi$-orbit category $\Oc{\al C}{F}{\Phi}$.
Then the quotient rings of locally $\Phi$-Beilinson-Green algebras
$$
\frac{\scr G^{\Phi,F}(X\oplus M)}{I}
\quad\text{and} \quad\frac{\scr G^{\Phi,F}(Y\oplus M)}{J}
$$
are derived equivalent, where $I=\diag(\cdots,\Fcogh_M (X\oplus M),\cdots)$ and $J=\diag(\cdots,\Fgh_M (Y\oplus M),\cdots)$ are ideals of $\scr G^{\Phi,F}(X\oplus M)$ and $\scr G^{\Phi,F}(M\oplus Y)$, respectively.
\end{thm}

In particular, if $\Phi=0$, then the quotient algebras
$\frac{\End(X\oplus M)}{\Fcogh_M (X\oplus M)}$ and $\frac{\End(M\oplus Y)}{\Fgh_M (Y\oplus M)}$
are derived equivalent.

We also construct derived equivalences of locally $\Phi$-Beilinson-Green algebras of semi-Gorenstein projective modules from a given derived equivalence.

\begin{thm}(=Theorem \ref{ps1})\label{p1p}
Let $\Phi$ be an admissible subset of $\mathbb{Z}$. Suppose that  $G: \Db{A\Modc}\lra\Db{B\Modc}$ is a derived equivalent between left coherent rings $A$ and $B$.
Denote by $\bar{G}$ the stable functor induced by $G$.
 If $X$ is a finitely presented semi-Gorenstein projective $A$-module, then locally $\Phi$-Beilinson-Green algebras $\scr G^{\Phi}(A\oplus X)$
and $\scr G^{\Phi}(B\oplus \bar{G}(X))$ are derived equivalent.
\end{thm}

Finally, we get derived equivalences between the Beilinson-Green algebras from $G$-graded derived equivalences between $G$-graded algebras $A$ and $B$ for a group $G$.

\begin{thm}(=Theorem \ref{GBG})\label{p2p} Suppose that there is a $G$-graded derived equivalence between $G$-graded algebras $A$ and $B$. Then there is a derived equivalence between the Beilinson-Green algebras $\overline{A}$ and $\overline{B}$.
\end{thm}

This paper is organized as follows.
In Section \ref{Pre}, we review some
basic facts on derived equivalences, approximations, ghost and coghost ideals and $n$-exangulated categories. In Section \ref{higer}, we give the definition of higher $\al D$-exact sequences, and we obtain derived equivalences between certain endomorphism subalgebras in the higher $\al D$-exact sequences, where $\al D$ is a subcategoy of an additive category $\al C$, and we prove Theorem \ref{pppp}.
In Section \ref{locallyfinite}, we introduce locally finite algebras, locally finite tilting families of complexes and locally $\Phi$-Beilinson-Green algebras with $\Phi$
an admissible subset of $\mathbb{Z}$.
In Section \ref{Derivedequiva}, we show that symmetric approximation sequences in $n$-exangulated categories give rise to derived equivalences between quotient algebras of locally $\Phi$-Beilinson-Green algebras in the principal diagonals modulo some ghost and coghost ideals by the locally finite tilting families of complexes, and we prove Theorem \ref{ppp}.
In section \ref{derivedequivalence}, we study derived equivalences of locally finite $\Phi$-Beilinson-Green algebras from a given derived equivalence, and from given graded derived equivalences of group graded algebras, we get derived equivalences between associated Beilinson-Green algebras of group graded algebras, Theorems \ref{p1p} and \ref{p2p} will be proved.
In Section \ref{example}, some examples are given to explicit our main theorems.

\section{Preliminaries}\label{Pre}
In this section, we shall review basic definitions and facts which will be useful in the proofs later on.

\subsection{Derived equivalences}
Throughout this paper, unless specified otherwise, $k$ will be a field.
We begin by briefly recalling some definitions and notations on
derived categories and derived equivalences.

Let $\mathscr{A}$ be an additive category. For two morphisms $\alpha:
X\ra Y$ and $\beta: Y\ra Z$, their composition is denoted by
$\alpha\beta$. A
complex $\cpx{X}=(X^i,d_{X}^i)$ over $\mathscr{A}$ is a sequence of
objects $X^i$ and morphisms $d_{X}^i$ in $\mathscr{A}$ of the form:
$$\cdots \ra X^i\stackrel{d_{X}^i}\ra X^{i+1}\stackrel{d_{X}^{i+1}}\ra
X^{i+1}\ra\cdots,$$ such that $d_{X}^id_{X}^{i+1}=0$ for all
$i\in\mathbb{Z}$. If $\cpx{X}=(X^i,d_{X}^i)$ and
$\cpx{Y}=(Y^i,d_{Y}^i)$ are two complexes, then a morphism $\cpx{f}:
\cpx{X}\ra\cpx{Y}$ is a sequence of morphisms $f^i: X^i\ra Y^i$ of
$\mathscr{A}$ such that $d^i_{X}f^{i+1}=f^id^i_{Y}$ for all
$i\in\mathbb{Z}$. The map $\cpx{f}$ is called a chain map between
$\cpx{X}$ and $\cpx{Y}$. The category of
complexes over $\mathscr{A}$ with chain maps is denoted by
$\C{\mathscr{A}}$. Let $\cpx{f}:\cpx{X}\ra \cpx{Y}$ be a
morphism of complexes. We say that $\cpx{f}$ is null-homotopic if
we get $f^i=d^ir^{i+1}+r^id^{i-1}$, where $r^i:X^i\ra Y^{i-1}$.
The homotopy category of complexes over $\mathscr{A}$ is denoted by
$\K{\mathscr{A}}$. If $\mathscr{A}$ is an abelian category, then a
morphism of complexes $\cpx{f}:\cpx{X}\ra \cpx{Y}$ is a
quasi-isomorphism if $H^i(\cpx{f}): H^i(\cpx{X})\stackrel{\sim}\ra
H^i(\cpx{Y})$ for $i\in \mathbb{Z}$, where $H^i(\cpx{X})$ denotes
the $i$-th cohomology group of the complex $\cpx{X}$. Denote by $\D{\mathscr{A}}$
the derived category of complexes over $\mathscr{A}$. It is
well known that, for an abelian category $\mathscr{A}$, the
categories $\K{\mathscr{A}}$ and $\D{\mathscr{A}}$ are triangulated
categories. For basic results on triangulated categories, we refer
the reader to \cite{Ha1} and \cite{N}.

Let $R$ be a commutative artinian ring with identity and let $A$ be an $R$-algebra.
Denote by $A\Modc$ and $A\modc$ the category of left $A$-modules and finitely presented left
$A$-modules, respectively. Note that if $A$ is a noethrian $R$-algebra, then $A\modc$ is an abelian category.
The full subcategory of $A\Modc$ and $A\modc$ consisting of
projective modules is denoted by $\Proj{A}$ and $\proj{A}$, respectively.
 Let $\Kb{A}$ denote the
homotopy category of bounded complexes of $A$-modules and let
$\Db{A}$ denote the bounded derived category of $A$-mod, respectively.

The following theorem is a key ingredient of Morita theory on derived equivalences for module categories of
rings or algebras which was established by Rickard \cite{Ri1}. For more details
on derived equivalences, we refer to
\cite{Ri1}.

\begin{thm} \label{R} \cite[Theorem \;6.4]{Ri1}
Let $A$ and $B$ be rings with identities. The following conditions are equivalent.

$(i)$ $\Db{A\Modc}$ and $\Db{B\Modc}$ are equivalent
as triangulated categories.

$(ii)$ $\Kf{\Proj{A}}$ and $\Kf{\Proj{B}}$ are equivalent as
triangulated categories.

$(iii)$ $\Kb{\Proj{A}}$ and $\Kb{\Proj{B}}$ are equivalent as
triangulated categories.

$(iv)$ $\Kb{\proj{A}}$ and $\Kb{\proj{B}}$ are equivalent as
triangulated categories.

$(v)$ $B$ is isomorphic to $\End_{\Db{A\Modc}}(\cpx{T})$ for some complex
$\cpx{T}$ in $\Kb{\proj{A}}$ satisfying

\qquad $(a)$ $\Hom_{\Db{A\Modc}}(\cpx{T},\cpx{T}[n])=0$
         for all $n\neq 0$.

\qquad $(b)$ $\add(\cpx{T})$, the category of direct summands of
          finite direct sums of copies of $\cpx{T}$, generates
          $\Kb{\proj{A}}$ as a triangulated category.
\end{thm}
\noindent{\bf Remarks.} (1) The rings $A$ and $B$ are said to be derived equivalent if $A$
and $B$ satisfy the conditions of the above theorem.

(2) The complex $\cpx{T}\in \Kb{\proj{A}}$
in Theorem \ref{R} (v) which satisfies the conditions (a) and (b) is
called a tilting complex for $A$.

Keller also \cite{Keller1994} gave the Morita theory on derived categories by differential graded categories which is useful to study derived equivalences for locally finite algebras via locally finite
tilting sets of complexes in the later.

\begin{thm}\cite[Theorem 9.2]{Keller1994} \label{k}
Let $\al{A}$ and $\al B$ be small $k$-categories. Then the following are equivalent.

$(i)$ There is a $\al A$-$\al B$-bimodue $Y$ such that  the left derived functor $\bf L T_Y: \al A\to\al B$ is a derived equivalence.

$(ii)$ There is an triangulated equivalence between $\D{\al A}$ and $\D{\al B}$.

$(iii)$ $\al B$ is equivalent to a full  subcategory $\al U$ of $\D{\al A}$ whose objets form a set of small generators and  $\al D(\al A )(U, V[i])=0$ for $i\neq0$ and $U,V\in\al U$.
\end{thm}

\subsection{Approximations, cohomological approximation and ghosts}

Let $\al C$ be an additive category, and let $\al D$ be a full additive subcategory of $\al C$ and an object $C$ in $\al C$.  A morphism $f:D\ra C$ in $\al C$ is called a right
$\al D$-approximation of $C$ if $D\in \al D$ and the induced map $\al (D', f): \al C(D', D)\ra \al C(D',C)$ is surjective for all $D'\in \al D$. Dually, we can define a left
$\al D$-approximation of $C$ in $\al C$.

Let  $\mathbb Z$ be the
set of all integers. Recall that a subset of $\mathbb Z$ containing $0$ is called an
admissible subset of $\mathbb Z$ (\cite{Hu2013}) if the following condition is
satisfied:

If $i,j,k\in \Phi$ satisfy that $i+j+k\in \Phi$, then $i+j\in \Phi$ if and
only if $j+k\in \Phi$.

Let $F$ be an endo-functor from $\al C$ to itself. Suppose that $\Phi$ is an admissible subset of $\mathbb{Z}$.
Recall that the left and right cohomological approximations
with respective to $\Phi$ in a triangulated category $\al C$ have been introduced in \cite{Hu2013a}.  Let $\al {D}$ be a full
subcategory of $\al C$ and  $X$ an object of $\al C$.  A
morphism $f:X\ra D$ is called a left ($\al {D}, F, \Phi$)-approximation of $X$ if $D\in\al {D}$, and
for any morphism $f': X\ra F^iD'$ with $D'\in\al {D}$ and $i\in\Phi$, there is a morphism $f'': D\ra F^iD' $
such that $f'=ff''$.  In the case that $\Phi$ is an admissible subset, we have the $\Phi$-orbit category $\al C^{F, \Phi}$,
and that $f$ is a left ($\al {D}, F, \Phi$)-approximation of $X$ is equivalent to saying that $\al C^{F, \Phi}(D,D')\ra \al C^{F, \Phi}(X,D')$ is surjective for all $D'\in \al D$.
Similarly, we have the notion of a right ($\mathcal {D}, F, \Phi$)-approximation of $X$ \cite{CH}.  A
morphism $g:D_X\ra X$ is called a right ($\mathcal {D}, F, \Phi$)-approximation of $X$, if for any morphism $D\ra F^iX$ with $D\in\mathcal {D}$ and $i\in\Phi$, there is a morphism $g'': D\ra F^iD_X $
such that $g=g''F^ig$.  In the case that $\Phi$ is an admissible subset, we have the $\Phi$-orbit category $\al C^{F, \Phi}$,
and that $f$ is a right ($\mathcal {D}, F, \Phi$)-approximation of $X$ is equivalent to saying that $\al C^{F, \Phi}(D,D_X)\ra \al C^{F, \Phi}(D,X)$ is surjective for all $D\in \al D$.

By an ideal $\al I$ of $\al C$ we mean that an additive subgroup $\al I(A, B)\subseteq \al C(A,B)$, for all $A,B\in \al C$, such that the composite $\alpha\beta$ of the morphisms
$\alpha, \beta\in\al C$ belongs to $\al I$ provided that either $\alpha$ or $\beta$ is in $\al I$. The quotient category $\al C/\al I$ of $\al C$ modulo an ideal $\al I$ has the same objects as $\al C$ and has the morphism
$\al C/\al I(A,B):=\al C(A,B)/\al I(A,B)$ for objects $A,B\in\al C$.

Let $\al D$ be a full additive subcategory of $\al C$, a morphism $f$ in $\al C$ is called a $\al D$-ghost provided that $\al (\al D, f)=0$. All $\al D$-ghosts in $\al C$ form an ideal of $\al C$, called the ideal of $\al D$-ghosts and denote by $\gh_{\al D}$.
Dually,  a morphism $g$ in $\al C$ is called a $\al D$-coghost provided that $\al (g,\al D)=0$. All $\al D$-coghosts in $\al C$ form an ideal of $\al C$, called the ideal of $\al D$-coghosts and denote by $\cogh_{\al D}$.
Let $\al F_{\al D}$ be the ideal of morphisms in $\al C$ factorizing through an object in $\al D$. The intersection $\gh_{\al D}\cap\al F_{\al D}$ is called the ideal of factorizable $\al D$-ghosts in $\al C$, denote by $\Fgh_{\al D}$. Similarly,
the intersection $\cogh_{\al D}\cap\al F_{\al D}$ is called the ideal of factorizable $\al D$-coghosts in $\al C$, denote by $\Fcogh_{\al D}$.

The following lemma describes the properties of ghost and coghost ideals.
\begin{lem} \cite[Lemma 2.1]{CH} \label{2.2}
$(1)$ If $X\in\al C$ admits a right $\al D$-approximation $f_X: D_X\ra X$, then
$$
\gh_{\al D}(X,Y)=\{g\in\al C(X,Y)|f_{X}g=0\}.
$$

$(2)$  If $Y\in\al C$ admits a left $\al D$-approximation $f^Y: Y\ra D^Y$, then
$$
\cogh_{\al D}(X,Y)=\{g\in\al C(X,Y)|gf^Y=0\}.
$$

$(3)$ If $X\in \al D$, then $\gh_{\al D}(X,Y)=0$ and $\cogh_{\al D}(X,Y)=\Fcogh_{\al D}(X,Y)$.

$(4)$ If $Y\in \al D$, then $\cogh_{\al D}(X,Y)=0$ and $\gh_{\al D}(X,Y)=\Fgh_{\al D}(X,Y)$.
\end{lem}

\subsection{Symmetric approximation sequences}

In this subsection, we will review the definition  and some properties of symmetric approximation sequences.

Let $\al C$ be an additive category, and let $\al D$ be a full additive subcategory of $\al C$.  A right $\al D$-approximation sequence in $\al C$ is a sequence
$$
D_m\ra D_{m-1}\ra \cdots\ra D_1\ra D_0 \ra Y
$$
with $D_i\in\al D$ for $i=0,1,\cdots m$, such that the following sequence is an exact sequence
$$
\al C(D, D_m)\ra \al C(D, D_{m-1})\ra\cdots\ra \al C(D, D_0)\ra \al C(D, Y)\ra 0
$$
for all $D\in\al D$. We can define a left $\al D$-approximation sequence dually. Recall that a pseudo-kernel of a morphism $g: X\ra Y$ is a morphism $f: Z\ra X$ such that
$$
\al C(C, Z)\ra \al C(C,X)\ra \al C(C, Y)$$
is exact for all $C\in\al C$. The pseudo-cokernal is defined dually.

\begin{defn}\cite[Definition 3.1]{CH}  Let $\al C$ be an additive category, and let $\al D$ be a full additive subcategory of $\al C$. A sequence
$$
X\xra {f_0} D_1\xra {f_1}\cdots \xra{f_{n-1}} D_{n}\xra{f_n} Y\quad \quad (\star)
$$
in $\al C$ is called a symmetric $\al D$-approximation sequence if the following three conditions are satisfied.

$(1)$ The sequence $D_1\xra {f_1}\cdots \xra{f_{n-1}} D_{n}\xra{f_n} Y$ is a right $\al D$-approximation sequence in $\al C$;

$(2)$ The sequence $X\xra {f_0} D_1\xra {f_1}\cdots \xra{f_{n-1}} D_{n}$ is a left $\al D$-approximation sequence in $\al C$;

$(3)$ The morphism $f_0$ is a pseudo kernel of $f_1$ and the morphism $f_n$ is a pseudo cokernel of $f_{n-1}$.
\end{defn}

The above definition, the sequence $(\star)$ is called a higher $\al D$-split sequence, if we replace the condition $(3)$ by the following condition

$(3')$  The morphism $f_0$ is a kernel of $f_1$ and the morphism $f_n$ is a cokernel of $f_{n-1}$.

\begin{lem}\cite[Lemma 3.3]{CH}\label{self-ortho}
Let $\mathcal{C}$ be an additive category, and let  $M$ be an object in $\mathcal{C}$. Suppose that
$\cpx{P}$:
$$0\lra P^{0}\lraf{d^{0}}P^{1}\lra\cdots\lra P^{n-1}\lraf{d^{n-1}} P^n\lra 0$$
is a complex over $\mathcal{C}$ such that $P^i\in\add(M)$ for all $i>0$, and  that the following two conditions are satisfied:

\smallskip
\begin{enumerate}
\item[(1)] $H^i(\HomP_{\mathcal{C}}(M, \cpx{P}))=0$ for all $i\neq 0, n$;

\item[(2)] $H^i(\HomP_{\mathcal{C}}(\cpx{P}, M))=0$ for all $i\neq -n$.
\end{enumerate}

{\parindent=0pt Then} $\cpx{P}$ is self-orthogonal as a complex both in $\Kb{\mathcal{C}/\cogh_M}$ and in $\Kb{\mathcal{C}/\Fcogh_M}$.
\label{Lemma-tilting-complex-left}
\end{lem}

\subsection{$n$-exangulated categories}

In this subsection we follow \cite[Section 2]{HLN} in order to recall the definition of an $n$-exangulated category.
One of the purposes of introducing $n$-exangulated categories is to provide a
common ground for studying the different settings of higher homological algebra.
Note that $n$-angulated and $n$-exact categories are $n$-exangulated categories.

 Throughout this subsection, we assume that $\mathcal{C}$ is an additive category.
We assume $\mathcal{C}$ comes equipped with a biadditive functor $\mathbb{E}: \mathcal{C}^{op}\times \mathcal{C}\rightarrow {\rm Ab}$. Thus,
 for any pair of objects $A, C\in\mathcal{C}$, the functors
 $$
 \mathbb{E}(C, -):\mathcal{C}\rightarrow {\rm Ab}
 $$
 and
  $$
\mathbb{E}(-, A):\mathcal{C}^{op}\rightarrow {\rm Ab}
 $$
  are additive. Furthermore, each morphism $f:X\ra Y$ in $\al C$ gives rise to abelian group homomorphism
  $\mathbb{E}(C, f):\mathbb{E}(C, X)\ra \mathbb{E}(C, Y)$ and $\mathbb{E}(f,A):\mathbb{E}(Y, A)\ra \mathbb{E}(X, A)$.

We are now to recall the definition of an $n$-exangulated category.

\begin{defn} \cite[Definition 2.32]{HLN}
An $n$-exangulated category is a triplet $(\al C,\mathbb{E},\mathfrak{s})$ of additive category $\al C$,
biadditive functor $\mathbb{E}: \al C^{op}\times\al C\ra Ab$, and its exact realization $\mathfrak{s}$, satisfying
the following conditions.

$(EA1)$ The class of $\mathfrak{s}$-inflations is closed under composition. Dually, the class of $\mathfrak{s}$-deflations is closed under composition.

$(EA2)$ For each $\delta\in\mathbb{E}(D,A)$ and $c\in\al C(C,D)$, if $\mathfrak{s}(c^{\mathbb{E}}\delta)=[\cpx{X}]$ and $\mathfrak{s}(\delta)=[\cpx{Y}]$,
then there exists a morphism $\cpx{x}=(1_A,f^1,\cdots,f^n,c):\cpx{X}\ra\cpx{Y}$ realising $(1_A,c):c^{\mathbb{E}}\delta\ra\delta$ such that $\mathfrak{s}((d_X^0)_{\mathbb{E}}\delta)=[\cpx{M_f}]$, where
$ \cpx{M_f}$ is the mapping cone. Such an $\cpx{f}$ is called a good lift of $(1_A,c)$.

$(EA2^{op})$ Dual of $(EA2)$.

\end{defn}

\begin{lem}\cite[Definitions 2.9 and 2.13,Proposition 3.6]{HLN}\label{Lemma-n-exangle}
Let $(\al C,\mathbb{E},\mathfrak{s})$ be an $n$-exangulated category, and let
$$
X^0\xra{d^0_X} X^1\xra{d^1_X}\cdots\xra{d_X^{n}} X^{n+1}\draf{\delta}
$$
  be a distinguished $n$-exangle in ${\al C}$. Then we have the following:

  $(1)$. $d^i_Xd_X^{i+1}=0$ for all $i=0,1,2,\cdots, n-1$, $(d_X^0)_{\mathbb{E}}\delta=0$ and $(d_X^n)^{\mathbb{E}}\delta=0$;

  $(2)$.  For all $Z\in {\al C}$, we have the following exact sequences

$$
\al C(Z,X^0)\xra{\al C(Z,d^0_X)}\al C(Z,X^1)\xra{\al C(Z,d^1_X)}\cdots\xra{\al C(Z,d^{n-1}_X)} \al C(Z,X^n)\xra{\al C(Z,d^n_X)} \al C(Z,X^{n+1})\xra{(\delta_{\sharp}^{\mathbb{E}})_Z}\mathbb{E}(Z,X^0),
$$
and
$$
\al C(X^{n+1},Z)\xra{\al C(d^n_X,Z)}\al C(X^n,Z)\xra{\al C(d^{n-1}_X,Z)}\cdots\xra{\al C(Z,d^1_X)} \al C(X^1,Z)\xra{\al C(d^0_X,Z)} \al C(X^0,Z)\xra{(\delta^{\sharp}_{\mathbb{E}})_Z}\mathbb{E}(X^{n+1},Z);
$$

  $(3)$. Suppose that $2\leq m< n$. Each commutative diagram
 $$\xymatrix@M=1mm{
X^0\ar[r]^{d^0_X}\ar[d]^{h^0} &X^1\ar[r]^{d^1_X}\ar[d]^{h^1} &\cdots\ar[r]^{d^{m-1}_X}& X^m\ar[r]^{d^m_X}\ar[d]^{h^m}&X^{m+1}\ar[r]^{d^{m+1}_X}\ar@{..>}[d]^{h^{m+1}} &\cdots\ar[r]^{d^n_X} &X^{n+1}\ar@{..>}[d]^{h^{n+1}}\ar@{..>}[r]^{\delta} &\\
Y^0\ar[r]^{d^0_Y} &Y^1\ar[r]^{d^1_Y} &\cdots\ar[r]^{d^{m-1}_Y}&Y^m\ar[r]^{d^m_Y} &Y^{m+1}\ar[r]^{d^{m+1}_Y}&\cdots\ar[r]^{d^n_Y} &Y^{n+1}\ar@{..>}[r]^{\rho} &\\
}$$
can be completed in ${\al C}$ to a morphism of $n$-exangles.
\end{lem}

\begin{lem}\label{5.13}
Let
$$
\xymatrix@M=1mm{
X^0\ar[r]^{d^0_X}\ar[d]^{h^0} &X^1\ar[r]^{d^1_X}\ar[d]^{h^1} &\cdots\ar[r]^{d^{m-1}_X}& X^m\ar[r]^{d^m_X}\ar[d]^{h^m}&X^{m+1}\ar[r]^{d^{m+1}_X}\ar[d]^{h^{m+1}} &\cdots\ar[r]^{d^n_X} &X^{n+1}\ar[d]^{h^{n+1}}\ar@{..>}[r]^{\delta} &\\
Y^0\ar[r]^{d^0_Y} &Y^1\ar[r]^{d^1_Y} &\cdots\ar[r]^{d^{m-1}_Y}&Y^m\ar[r]^{d^m_Y} &Y^{m+1}\ar[r]^{d^{m+1}_Y}&\cdots\ar[r]^{d^n_Y} &Y^{n+1}\ar@{..>}[r]^{\rho} &\\
}
$$
be any morphism of $n$-exangles. Then the following are equivalent.

(1) $h^0$ factors through $d^0_X$;

(2) $(h^0)_{\mathbb{E}}\delta=(h^{n+1})^{\mathbb{E}}\rho=0$;

(3) $h^{n+1}$ factors through $d^n_Y$.
\end{lem}

\begin{proof}
By Lemma \ref{Lemma-n-exangle}, we have the following exact sequences
$$
\al C(-,X^0)\xra{\al C(-,d^0_X)}\al C(-,X^1)\xra{\al C(-,d^1_X)}\cdots\xra{\al C(-,d^{n-1}_X)} \al C(-,X^n)\xra{\al C(-,d^n_X)} \al C(-,X^{n+1})\xra{(\delta_{\sharp}^{\mathbb{E}})}\mathbb{E}(-,X^0).
$$
and
$$
\al C(X^{n+1},-)\xra{\al C(d^n_X,-)}\al C(X^n,-)\xra{\al C(d^{n-1}_X,-)}\cdots\xra{\al C(d^1_X,-)} \al C(X^1,-)\xra{\al C(d^0_X,-)} \al C(X^0,-)\xra{(\delta^{\sharp}_{\mathbb{E}})}\mathbb{E}(X^{n+1},-).
$$
in ${\rm Ab}$ are exact. If $h^0$ factors through $d^0_X$, then by the exact sequence, $h^0$ is in $\Ker(\delta^{\sharp}_{\mathbb{E}})$, that is,
$(h^0)_{\mathbb{E}}\delta=0$, consequently, $(h^{n+1})^{\mathbb{E}}\rho=(h^0)_{\mathbb{E}}\delta=0$.

\end{proof}

\medskip

\begin{defn}\cite[Definition 2.31]{BA21} \label{exangulated functor}
Let $(\al C,\mathbb{E},\mathfrak{s})$ and $(\al C,\mathbb{E},\mathfrak{s}')$ be  $n$-exangulated categories.
An additive functor $F$ from ${\al T}$ to ${\al T}'$ is called an {\em $n$-exangulated functor} if  there is a natural transformation
$$
\Upsilon=\{\Upsilon_{(C,A)}\}_{(C,A)\in\al C^{op}\times\al C}: \mathbb{E}\Longrightarrow  \mathbb{E}'(F-,F-)
$$
of functors $\al C^{op}\times\al C\ra AB$, such that $\mathfrak{s}(\delta)=[\cpx{X}]$ implies $\mathfrak{s}'(\Upsilon_{(X^{n+1},X^0)}(\delta))=[F(\cpx{X})]$.
\end{defn}

\medskip
Let  $(\al C,\mathbb{E},\mathfrak{s})$ and $(\al C',\mathbb{E}',\mathfrak{s})$ be $n$-exangulated categories.
By Definition \ref{exangulated functor}, $F$ is an $n$-exangulated functor  from ${\al C}$ to ${\al C}'$, then we have the following:
$$
F(X^0)\xra{F(d^0_X)} F(X^1)\xra{F(d^1_X)}\cdots\xra{F(d_X^{n})} F(X^{n+1})\draf{\Upsilon_{(X^{n+1},X^0)}(\delta)}
$$
is an $\mathfrak{s}'$-distinguished $n$-exangle whenever
$$
X^0\xra{d^0_X} X^1\xra{d^1_X}\cdots\xra{d_X^{n}} X^{n+1}\draf{\delta}
$$
is an $\mathfrak{s}$-distinguished $n$-exangle.
Therefore, if $F$ is an $n$-exangle endo-functor between $\al C$, we conclude that
$$
F^i(X^0)\xra{F^i(d^0_X)} F^i(X^1)\xra{F^i(d^1_X)}\cdots\xra{F^i(d_X^{n})} F^i(X^{n+1})\draf{\Theta(\delta)},
$$
where $\Theta(\delta):=\Upsilon_{(F^i(X^{n+1}),F^i(X^0))}(\Upsilon_{(F^{i-1}(X^{n+1}),F^{i-1}(X^0))}\cdots(\Upsilon_{(X^{n+1},X^0)}(\delta)))\in\mathbb{E}(F^i(X^{n+1}),F^i(X^0))$ is an $\mathfrak{s}$-distinguished $n$-exangle. Thanks to Raphael Bennett-Tennenhaus for explanations of $n$-exangulated functors.

\subsection{Symmetric approximation in n-exangulated categories}

Now let $(\al C,\mathbb{E},\mathfrak{s})$ be an $n$-exangulated category, and let $F$ be an $n$-exangulated functor from ${\al C}$ to itself.  Suppose that $\Phi$ is an admissible subset of $\mathbb{Z}$, and $\Oc{\al C}{F}{\Phi}$ is the $\Phi$-orbit category of ${\al C}$. Then one may ask whether the $\Phi$-orbit category is again naturally $n$-exangulated.

\medskip
The following lemma will be useful in the proof of Theorem \ref{Theorem-ghost}.

\begin{lem}\label{4.2}
Let $(\al C,\mathbb{E},\mathfrak{s})$ be an $n$-exangulated category ($n\geq 1$) with an $n$-exangulated endo-functor $F$,
and let $M$ be an object in $\mathcal {C}$. Suppose that
$$
X\xra f M_1\ra M_2\ra\cdots\ra M_n\xra g Y\draf{\delta}
$$
is an $n$-exangle in $\mathcal {C}$ with $M_j\in\add(M)$ for all $j=1,\cdots,n$, and that $f$ is a left $\add_{\Oc{\al C}{F}{\Phi}}(M)$-approximation
of $X$ and that $g$ is a right $\add_{\Oc{\al C}{F}{\Phi}}(M)$-approximation of $Y$ in $\Phi$-orbit category $\Oc{\al C}{F}{\Phi}$. Then the complex
$$
0\lra X\xra f M_1\ra M_2\ra\cdots\lra M_n\lra 0$$
is self-orthogonal in $\Kb{{\al C}^{F,\Phi}/\cogh_{\al D}}$ and $\Kb{{\al C}^{F,\Phi}/\Fcogh_{\al D}}$, where $\al D=\add_{{\al C}^{F,\Phi}}(M)$.
\end{lem}

\begin{proof} Let $\al D=\add_{\Oc{\al C}{F}{\Phi}}(M)$. Then by our assumptions,
$$
X\xra f M_1\ra M_2\ra\cdots\ra M_n\xra g Y$$
is a symmetric $\al D$-approximation sequence.  Denote the complex
$X\xra f M_1\ra M_2\ra\cdots\ra M_n$ by $\cpx{P}$ and put $X$ in degree zero.
By the condition (1) of symmetric $\al D$-approximation sequence, applying $\Oc{\al C}{F}{\Phi}(M,-)$ to $\cpx{P}$ results in a sequence
$$
 0\lra \Oc{\al C}{F}{\Phi}(M,X)\xra{\Oc{\al C}{F}{\Phi}(M,f)} \Oc{\al C}{F}{\Phi}(M,M_1)\lra \Oc{\al C}{F}{\Phi}(M,M_2)\lra\cdots\lra \Oc{\al C}{F}{\Phi}(M,M_n)\lra 0$$
with $H^i(\Hom_{{\Oc{\al C}{F}{\Phi}}}(M,\cpx{P}))=0$ for all $i\neq 0, n$. By the condition (2) of symmetric $\al D$-approximation sequence applying $\Oc{\al C}{F}{\Phi}(-,M)$ to $\cpx{P}$ results in a sequence
$$
 0\lra \Oc{\al C}{F}{\Phi}(M_n, M)\lra \cdots \Oc{\al C}{F}{\Phi}(M_1,M)\ra \Oc{\al C}{F}{\Phi}(X, M)\lra 0$$
 with $H^i(\Hom_{{\Oc{\al C}{F}{\Phi}}}(\cpx{P}, M))=0$ for all $i\neq -n$.
Then by Lemma \ref{self-ortho}, the complex
 $$
 \cpx{P}: 0\lra X\xra f M_1\ra M_2\ra\cdots\ra M_n$$
with $X$ in degree zero is self-orthogonal in $\Kb{{\al C}^{F,\Phi}/\cogh_{\al D}}$ and $\Kb{{\al C}^{F,\Phi}/\Fcogh_{\al D}}$.
\end{proof}

\section{Higher exact sequences and derived equivalences for subalgebras}\label{higer}

We introduce the following definition of higher $\al D$-exact sequences.

\begin{defn} \label{higher} Let $\al C$ be an additive category, and let $\al D$ be a full subcategory of $\al C$. A sequence
$$
X\xra {f} D_1\xra {f_1}\cdots \xra{f_{n-1}} D_{n}\xra{g} Y\quad \quad (\star)
$$
in $\al C$ is called a higher $\al D$-exact sequence if the following three conditions are satisfied.

$(1)$ $D_i\in \al D$ for $i=1,\cdots, n$

$(2)$ There are two exact sequences

$$
 0\lra \Hom_{\al C}(X\oplus D, X)\xra {f^*} \Hom_{\al C}(X\oplus D,  D_1)\ra\cdots\ra \Hom_{\al C}(X\oplus D, D_n)\xra {g^*} \Hom_{\al C}(X\oplus D, Y)\quad (\dagger);
$$

$$
 0\lra \Hom_{\al C}(Y, Y\oplus D)\xra {g_*} \Hom_{\al C}(D_n, Y\oplus D)\ra\cdots\ra \Hom_{\al C}(D_1,Y\oplus D)\xra {f_*} \Hom_{\al C}(X, Y\oplus D)\quad (\ddagger)
$$
for every object $D\in \al D$.

\end{defn}

Note that if $n=1$, then higher $\al D$-exact sequences are in the sense of exact $\al D$-sequences introduced by Chen and Xi in \cite{CX}.

\begin{thm}\,\label{4.1}
Let  $\al C$ be an additive category, and let $M$ be an object in $\al {C}$. Suppose that
 $$X\xra f M_1\xra {d_1} M_2\ra\cdots\xra {d_{n-1}} M_n\xra g Y$$
 is a higher $\add(M)$-exact sequence. Let
$$
A=\left\{\begin{pmatrix}
s_1& s_2 \\
fs_3 &s_4
\end{pmatrix}\in \begin{pmatrix}
 \End_{\al C}(M)& \Hom_{\al C}(M,X)\\
\Hom_{\al C}(X,M) & \End_{\al C}(X)
\end{pmatrix}  \bigg| \
\begin{aligned}
s_3\in\Hom_{\al C}(M_1,M)& \; \text{ and there exists}\\
s_5 \in \End_{\al C}(M_1) & \; \text{such that}\; s_4f=fs_5.
\end{aligned}
\right\},
$$
and
$$
B=\left\{\begin{pmatrix}
s_1&s_2g \\
s_3 &s_4
\end{pmatrix}\in \begin{pmatrix}
\End_{\al C}(M) &\Hom_{\al C}(M,Y) \\
\Hom_{\al C}(Y,M) &\End_{\al C}(Y)
\end{pmatrix}  \bigg| \
\begin{aligned}
s_2\in\Hom_{\al C}(M,M_n)& \; \text{ and there exists}\\
 s_5 \in \End_{\al C}(M_n) & \; \text{such that}\; gs_4=s_5g.
\end{aligned}
\right\},
$$
be subrings of $\End_{\al C}(M\oplus X)$  and $\End_{\al C}(M\oplus Y)$, respectively. Then $A$ and $B$ are derived equivalent.
\end{thm}

\begin{rem} This theorem is the higher version of Chen and Xi \cite[Proposition 2.4]{CX}. We use a different approach to prove this result in the following.
\end{rem}

Before we prove the Theorem \ref{4.1}, we give some rather preliminaries. This construction comes from Chen's result in an abelian category in \cite{Chen2014}, we modify her idea in an additive category.

Let $W:=M\oplus X \oplus\bigoplus^n_{i=1} M_i\oplus Y$. We thus have an endomorphism algebra $\Gamma:=\End(W)$ of $W$ as follows:
$$\Gamma=
\begin{pmatrix}
 \End(X) &  \Hom(X,M)  &  \Hom(X,M_1) & \cdots & \Hom(X,M_n)& \Hom(X,Y)\\
\Hom(M,X)      &\End(M) & \Hom(M,M_1) & \cdots & \Hom(M,M_n)&\Hom(M,Y)\\
\Hom(M_1,X)& \Hom(M_1,M)    & \Hom(M_1,M_1)  & \cdots  & \Hom(M_1,M_n)& \Hom(M_1,Y)\\
  \vdots & \vdots & \vdots & \vdots& \vdots\\
  \Hom(M_{n-1},X) &\Hom(M_{n-1},M)    & \Hom(M_{n-1},M_1)&\cdots  & \Hom(M_{n-1},M_n)& \Hom(M_{n-1},Y)\\
  \Hom(M_n,X) &\Hom(M_n,M) &\Hom(M_n,M_1) &\cdots &\End(M_n)&\Hom(M_n,Y)\\
\Hom(Y,X)      &\End(Y,M) & \Hom(Y,M_1) & \cdots & \Hom(Y,M_n)&\End(Y)\\
\end{pmatrix}.
$$

We consider the subalgebra $\Lambda$ of $\Gamma$
$$\Lambda
=
\begin{pmatrix}
\widetilde{\End(X)}  &  \widetilde{\Hom(X,M)}     &  \widetilde{\Hom(X,M_1)} & \cdots & \widetilde{\Hom(X,M_n)}& \widetilde{\Hom(X,Y)}\\
 \Hom(M,X)      &\End(M) & \Hom(M,M_1) & \cdots & \Hom(M,M_n)&\widetilde{\Hom(M,Y)}\\
   \Hom(M_1,X)& \Hom(M_1,M)    & \Hom(M_{n-2},M_1)  & \cdots  & \Hom(M_1,M_n)& \widetilde{\Hom(M_1,Y)}\\
  \vdots & \vdots & \vdots & \vdots& \vdots\\
  \Hom(M_{n-1},X) &\Hom(M_{n-1},M)    & \Hom(M_{n-1},M_1)&\cdots  & \Hom(M_{n-1},M_n)& \widetilde{\Hom(M_{n-1},Y)}\\
  \Hom(M_n,X) &\Hom(M_n,M) &\Hom(M_n,M_1) &\cdots &\End(M_n)&\widetilde{\Hom(M_n,Y)}\\
\Hom(Y,X)      &\End(Y,M) & \Hom(Y,M_1) & \cdots & \Hom(Y,M_n)&\widetilde{\End(Y)}\\
\end{pmatrix}.
$$

where
$$
\begin{array}{rl}
\widetilde{\End(X)}=\{t\in\End(X)\mid tf=ft'\;\text{for some}\; t'\in\End_{\al C}(M_1) \},\\
\widetilde{\Hom(X,M)}=\{t\in\Hom(X,M)\mid t \;\text{factors through}\; f \;\text{in}\; \mathcal {C}\},\\
\widetilde{\Hom(X,M_i)}=\{t\in\Hom(X,M_i)\mid t \;\text{factors through}\; f \;\text{in}\; \mathcal {C}, i=1,\cdots,n\},\\
\widetilde{\Hom(X,Y)}=\{t\in\Hom(X,Y)\mid t \;\text{factors through}\; f\; \text{and}\; g\;\text{in}\; \mathcal {C}\},\\
\widetilde{\Hom(M,Y)}=\{t\in \Hom(M,Y)\mid t \;\text{factors through}\; g \;\text{in}\; \mathcal {C} \},\\
\widetilde{\Hom(M_i,Y)}=\{t\in \Hom(M,Y)\mid t \;\text{factors through}\; g \;\text{in}\; \mathcal {C},i=1,\cdots,n \},\\
\widetilde{\End(Y)}=\{t\in \End(Y)\mid gt=t'g\;\text{for some} \; t'\in\End_{\al C}(M_n) \}.
\end{array}
$$

\begin{lem}\cite[Lemma\; 4.2]{Xi}\label{lem3}
Suppose that $\Lambda$ is a subring of $\Gamma$ with the same identity.
\begin{itemize}

\item[{\rm (1)}] The restriction functor $\mathcal {F}: \Gamma\modc\ra \Lambda\modc$ is an exact faithful functor, and has a right adjoint
$\mathcal {G}=\Hom(_\Lambda\Gamma_\Gamma,-): \Lambda\modc\ra \Gamma\modc$ and a left adjoint $\mathcal {H}=\Gamma\otimes_\Lambda-: \Lambda\modc\ra \Gamma\modc$. In particular, $\mathcal {H}$ preserves projective modules and
$\mathcal {G}$ preserves injective modules.
\item[{\rm (2)}] The functor $\mathcal {H}=\Gamma\otimes_\Lambda-: \proj{\Lambda}\ra \proj{\Gamma}$ which sends $\Lambda$ to $\Gamma$ is faithful.
\end{itemize}
\end{lem}

Recall that an additive category $\mathcal{C}$ is idempotent complete (or Karoubi envelope) if for every
idempotent $p: C\ra C$, that is, $p^2 = p$, there is a decomposition $C\simeq K\oplus K'$ such that $p\simeq \begin{pmatrix}
0&0 \\
0&1
\end{pmatrix}$.
Note that the additive category $\mathcal{C}$ is idempotent complete if and only if every idempotent has a kernel.
For more details of Karoubi's construction on the idempotent complete of an additive category, we refer to \cite{Ka}.

Let $\mathcal{C}$ be an additive category. The idempotent completion
of $\mathcal{C}$ is denoted by $\widehat{\mathcal{C}}$ and is defined as follows. The objects of $\widehat{\mathcal{C}}$ are the pairs $(C, p)$, where $C$ is an object of $\mathcal{C}$ and $p:C\ra C$ is an idempotent morphism. A morphism in $\widehat{\mathcal{C}}$ from $(C, p)$
to $(D, q)$ is a morphism $f:C\ra D \in\mathcal{C}$ such that $fp=qf=f$. For any object $(C, p)$ in
$\widehat{\mathcal{C}}$, the identity morphism $1_{(C, p)} = p$.
There is a fully faithful additive functor $i_{\mathcal {C}}:\mathcal {C}\ra \widehat{\mathcal{C}}$ defined as follows. For
an object $C$ in $\mathcal{C}$, we have that $i_{\mathcal {C}}(C)=(C, 1_C)$ and for a morphism $f$ in $\mathcal {C}$, we have that
$i_{\mathcal {C}}(f)=f$. Every additive category $\mathcal{C}$ can be fully faithfully embedded into an idempotent complete additive category $\widehat{\mathcal{C}}$.

\begin{lem}
Let $\al C$ be an additive category.
Then the additive functor $\Hom_{\widehat{\al C}}((W,1_W),-): \widehat{\add W}\lra \Gamma\proj$ is an equivalence of additive categories, where $\widehat{\add W}$ is the idempotent complete of $\add W$.
\end{lem}

\begin{proof} Clearly, the additive functor $\Hom_{\widehat{\al C}}((W,1_W),-)$ is fully faithful. To see that it is
dense, let $P\in\Gamma\proj$. Then we have $P\oplus Q\simeq n\Gamma$. Hence there is an idempotent $e\in\End(n\Gamma)$ such that $P=\Ker(e)$.
Therefore there is an idempotent $f\in\End(nW)$ such that $\Hom_{\widehat{\al C}}((W,1_W),(nW,f))=e$.
It follows that $nW=\Ker(f)\oplus\Img(f)$ and $\Ker(f)$ is in $\add W$. Then $\Hom_{\widehat{\al C}}((W,1_W),(\Ker(f),1))=P$
since $\Hom_{\al C}(W,-)$ is left exact.
\end{proof}

Let $\mathcal{G}$ be the inverse functor of $\Hom_{\widehat{\al C}}((W,1_W),-)$, and denote by $\mathcal{F}$ the composition of the functors  $-_\Lambda\otimes\Gamma$ and $\mathcal{G}$.
We then have the following commutative diagram:

$$\xymatrix{
\proj{\Lambda} \ar[rrr]^{\Gamma\otimes_\Lambda-}\ar[drrr]_{\mathcal{F}}&&& \proj{\Gamma}\ar@<1ex>[d]^{\mathcal{G}}\\
&&&\widehat{\add W}\ar@<1ex>[u]^{\Hom_{\mathcal {C}}(W,-)}
.}$$
By Lemma \ref{lem3}, we know that $\mathcal{F}$ is a faithful functor. Denote the image of $F$ by $\mathcal {S}$. Then $\mathcal {S}$
is a subcategory of $\widehat{\add V}$, but it is not necessarily a full subcategory of $\widehat{\add V}$.

As in \cite{C}, we have the following isomorphisms:
$$\aligned
\widetilde{\Hom(X,M)}\simeq\Hom_{\Lambda}(\Lambda e_{11},\Lambda e_{22})
\simeq\widetilde{\Hom_{\Gamma}(\Gamma e_{11},\Gamma e_{22})}
\simeq\Hom_{\mathcal {S}}(\mathcal{G}(\Gamma e_{11}),\mathcal{G}(\Gamma e_{22}))\\
\simeq\Hom_{\mathcal {S}}(X,M),
\endaligned$$
where $e_{11}, e_{22}$ are idempotents of $\Lambda$, and $\widetilde{\Hom_{\Gamma}(\Gamma e_{11},\Gamma e_{22})}$ is a subring of $\Hom_{\Gamma}(\Gamma e_{11},\Gamma e_{22})$.
Consequently,
$$\aligned
\widetilde{\Hom(X,M)}\simeq\Hom_{\mathcal {S}}(X,M),
\widetilde{\Hom(M,Y)}\simeq\Hom_{\mathcal {S}}(M,Y),
\widetilde{\End(X)}\simeq\End_{\mathcal {S}}(X),\\
\End(M)\simeq\End_{\mathcal {S}}(M), \Hom(M,X)\simeq\Hom_{\mathcal {S}}(M,X), \Hom(Y,M)\simeq\Hom_{\mathcal {S}}(Y,M),\\
\widetilde{\Hom(M,Y)}\simeq\Hom_{\mathcal {S}}(M,Y), \Hom(Y,X)\simeq\Hom_{\mathcal {S}}(Y,X), \widetilde{\End(Y)}\simeq\End_{\mathcal {S}}(Y),\\
\endaligned$$
and
$$\aligned
\widetilde{\Hom(X,M_i)}\simeq\Hom_{\mathcal {S}}(X,M_i),
\widetilde{\Hom(M_i,Y)}\simeq\Hom_{\mathcal {S}}(M_i,Y),
\Hom(M_i,X)\simeq\End_{\mathcal {S}}(M_i,X),\\
\Hom(Y,M_i)\simeq\Hom_{\mathcal {S}}(Y,M_i),
\Hom(M_i,M_j)\simeq\Hom_{\mathcal {S}}(M_i,M_j),
\endaligned$$
for $1\leq i,j\leq n$.

{\bf The proof of Theorem \ref{4.1}}
We then simplify the proof by the above construction.

Since $\widetilde{\Hom(X,M)}\simeq\Hom_{\mathcal {S}}(X,M)$ and $\widetilde{\Hom(M,Y)}\simeq\Hom_{\mathcal {S}}(M,Y)$,
we have the following exact sequences,
$\Hom_{\mathcal {S}}(M_1,M)\xra{\Hom_{\mathcal {S}}(f,M)} \Hom_{\mathcal {S}}(X,M)\ra 0,
\;\text{and}\;
\Hom_{\mathcal {S}}(M,M_n)\xra{\Hom_{\mathcal {S}}(M,g)} \Hom_{\mathcal {S}}(M,Y)\ra 0.$
Hence, the morphism $X\xra f M_1$ is a left $\add M$-approximation in $\mathcal {S}$, and the morphism $M_n\xra g Y$ is a right $\add M$-approximation in $\mathcal {S}$. Verify that there are two exact sequences
$$
0\lra \Hom_{\mathcal {S}}(M, X)\xra {\Hom_{\mathcal {S}}(M, f)} \Hom_{\mathcal {S}}(M,  M_1)\ra\cdots\ra \Hom_{\mathcal {S}}(M, M_n)\xra {\Hom_{\mathcal {S}}(M, g)} \Hom_{\mathcal {S}}(M, Y)\lra 0,
$$
and
$$0\lra \Hom_{\mathcal {S}}(Y, M)\xra {\Hom_{\mathcal {S}}(g, M)} \Hom_{\mathcal {S}}(M_n, M)\ra\cdots\ra \Hom_{\mathcal {S}}(M_1,M)\xra {\Hom_{\mathcal {S}}(f, M)} \Hom_{\mathcal {S}}(X, M)\lra 0.
$$
Set the complex $\cpx{P}:  0\lra X\xra {f}  M_1 \xra {d_1}\cdots\lra  M_{n-1}
\xra {\overline{d_{n-1}}} M_n\oplus M\lra 0$, where $\overline{d_{n-1}}=(0,d_{n-1}):M_{n-1}\ra M_n\oplus M$.
Since $H^i(\Hom_{\al S}(M,\cpx{P}))=0$ for all $i\neq n$ and $H^i(\Hom_{\al S}(\cpx{P},M))=0$ for all $i\neq -n$,
it follows from \cite[Lemma 2.1]{Hoshino2003} that the complex $\cpx{P}$
is self-orthonal in $\Kb{\mathcal{S}}$.

Let $V=M\oplus X$. Hence
$$\End_{\mathcal{S}}(V)\simeq\begin{pmatrix}
\End_{\mathcal {S}}(M)&       \Hom_{\mathcal {S}}(M, X)\\
\Hom_{\mathcal {S}}(X,M)     & \End_{\mathcal {S}}(X)
\end{pmatrix}
=\begin{pmatrix}
\End_{\mathcal {C}}(M)&       \Hom_{\mathcal {C}}(M, X)\\
\widetilde{\Hom_{\mathcal {S}}(X,M})     & \widetilde{\End_{\mathcal {C}}(X)}
\end{pmatrix}=A.
$$
Therefore, we have a complex over $\End_{\mathcal{S}}(V)$ of the form
$$
\cpx{T}: 0\lra \Hom_{\mathcal {S}}(V, X)\xra {\Hom_{\mathcal {S}}(V, f)} \Hom_{\mathcal {S}}(V,  M_1)\ra\cdots\ra \Hom_{\mathcal {S}}(V, M_n\oplus M)\lra 0.
$$
Then we have to show that $\cpx{T}$ is a tilting complex over $\End_{\mathcal{S}}(V)$.
The complex $\cpx{T}$ is self-orthogonal since $\Hom_{\mathcal{S}}(V,-):\add V\ra \proj{\End_{\mathcal{S}}(V)}$ is fully faithful.
It is easy to see that $\add(\cpx{T})$ generates $\Kb{\proj{\End_{\mathcal{S}}(V)}}$ as a triangulated category. It suffices to show that  $\End_{\Kb{\proj{\End_{\mathcal{S}}(V)}}}(\cpx{T})\simeq B$ as rings.

Let $\overline{d_{n-1}}=(0,d_{n-1}):M_{n-1}\ra M_n\oplus M$ and $\overline{g}=\begin{pmatrix}
1 &0 \\
0 &g
\end{pmatrix}: M_n\oplus M\ra Y\oplus M$. Then it follows $\overline{d_{n-1}}\overline{g}=0$ from $d_{n-1}g=0$.

It follows that
$$
\End_{\Kb{\proj{A}}}(\cpx{T})\simeq \End_{\Kb{\al S}}(\cpx{P})
$$
as rings, since $\Hom_{\al S}(V,-):\add(V)\ra \proj{\End_{\mathcal{S}}(V)}$
is fully faithful. To show the claim, it suffices to prove that there is a ring isomorphism
$$
\Theta: \End_{\Kb{\al S}}(\cpx{P})\lra \End_{\al S}(M\oplus Y).
$$

Now let $(\alpha,\alpha_1,\cdots, \alpha_n):\cpx{P}\lra\cpx{P}$ be a chain map between $\cpx{P}$ with $\alpha\in\End(X), \alpha_i\in\End(M_i)$ for $i=1,\cdots, n-1$ and $\alpha_n\in\End(M\oplus M_n)$. By the exact sequence  $(\ddagger)$ in the definition,
$$
 0\lra \Hom_{\al C}(Y\oplus M, Y\oplus M)\xra {\overline{g}_*} \Hom_{\al C}(M_n\oplus M, Y\oplus M)\xra{\overline{d_{n-1}}_*} \Hom_{\al C}(M_{n-1},Y\oplus M)\cdots\xra {f_*} \Hom_{\al C}(X, Y\oplus M)\quad
$$
Since $\overline{d_{n-1}}_*(\alpha_n\overline{g})=\overline{d_{n-1}}\alpha_n\overline{g}=\alpha_{n-1}\overline{d_{n-1}}\overline{g}=0$, there is a unique morphism $\beta\in\End(Y\oplus M)$  such that $\overline{g}_*(\beta)=\overline{g}\beta=\alpha_n\overline{g}$.
Therefore, there exists a unique morphism $\beta\in\End_{\al C}(Y\oplus M)$ such that the following diagram is commutative:

$$\xymatrix@M=1mm{
X\ar[r]^{f}\ar[d]^{\alpha} &M_1\ar[r]^{d_1}\ar[d]^{\alpha_1} & M_2\ar[r]^{d_2}\ar[d]^{\alpha_2} &\cdots\ar[r]^{d_{n-2}}&M_{n-1}\ar[d]^{\alpha_{n-1}}\ar[r]^{\overline{d_{n-1}}} &M_n\oplus M\ar[d]^{\alpha_n}\ar[r]^{\overline{g}} &Y\oplus M\ar@{..>}[d]^{\beta}\\
X\ar[r]^{f} &M_1\ar[r]^{d_1} &M_2\ar[r]^{d_2} &\cdots\ar[r]^{d_{n-2}}& M_{n-1}\ar[r]^{\overline{d_{n-1}}} \ar[r]^{\overline{d_{n-1}}} &M_n\oplus M\ar[r]^{\overline{g}} &Y\oplus M\\
}$$
where $\overline{g}=\begin{pmatrix}
1 &0 \\
0 &g
\end{pmatrix}: M_n\oplus M\ra Y\oplus M$.
To check $\beta\in\End_{\al S}(Y\oplus M)$, we set
$
\alpha_n=\begin{pmatrix}
x_1&x_2 \\
x_3 &x_4
\end{pmatrix}\in
\begin{pmatrix}
\End_{\al C}(M) &\Hom_{\al C}(M,M_n) \\
\Hom_{\al C}(M_n,M) &\End_{\al C}(M_n)
\end{pmatrix},
$
and
$
\beta=\begin{pmatrix}
\beta_1&\beta_2 \\
\beta_3 &\beta_4
\end{pmatrix}\in \begin{pmatrix}
\End_{\al C}(M) &\Hom_{\al C}(M,Y) \\
\Hom_{\al C}(Y,M) &\End_{\al C}(Y)
\end{pmatrix}.
$
It follows from $\overline{g}\beta=\alpha_n\overline{g}$ that
$
\begin{pmatrix}
1&0 \\
0 &g
\end{pmatrix}
\begin{pmatrix}
\beta_1&\beta_2 \\
\beta_3 &\beta_4
\end{pmatrix}=\begin{pmatrix}
x_1&x_2 \\
x_3 &x_4
\end{pmatrix}
\begin{pmatrix}
1&0 \\
0 &g
\end{pmatrix}.
$
Consequently,  we get $\beta_1=x_1, \beta_2=x_2g, g\beta_3=x_3, g\beta_4=x_4g$. It follows that $\beta\in\End_{\al S}(M\oplus Y)$.

To show that $\Theta$ is well-defined, it suffices to prove that the chain map $(\alpha,\alpha_1,\cdots, \alpha_n)$ is homotopic to the zero if and only if $\beta=0$.
If $(\alpha,\alpha_1,\cdots, \alpha_n)$ is null-homotopic, then there exists $h_n: M_n\oplus M\ra M_{n-1}$ such that $\alpha_n=h_n\overline{d_{n-1}}$. In this case, we have
$$
\alpha_n\overline{g}=h_n\overline{d_{n-1}}\overline{g}=0,
$$
and therefore $\overline{g}_*(\beta)=0$,  we thus get $\beta=0$.

Now suppose that $\beta=0$. Then $\overline{g}\beta=\alpha_n\overline{g}=0$. By the exact sequence ($\dagger$) of the Definition \ref{higher}, the following sequence
$$
 0\ra \Hom_{\al C}(M\oplus M_n, X)\xra {f^*} \cdots\ra\Hom_{\al C}(M\oplus M_n, M_{n-1})
\xra {\overline{d_{n-1}}^*}\Hom_{\al C}(M\oplus M_n, M_n\oplus M)\xra{\overline{g}^*}\Hom_{\al C}(M\oplus M_n, M\oplus Y)
$$
is exact. It follows that there exists a morphism $h_n: M\oplus M_n\ra M_{n-1}\oplus M$ such that $\alpha_n=h_n\overline{d_n}$. Since $(\alpha_{n-1}-\overline{d_{n-1}}h_n)\overline{d_{n-1}}=\alpha_{n-1}\overline{d_{n-1}}-\overline{d_{n-1}}h_n\overline{d_{n-1}}=\alpha_{n-1}\overline{d_{n-1}}-\overline{d_{n-1}}\alpha_n=0$, by the exact sequence ($\dagger$) of the Definition \ref{higher} again, the following sequence
$$\begin{array}{rl}
 0\ra \Hom_{\al C}(M_{n-1}, X)\xra {f^*} \cdots\ra\Hom_{\al C}(M_{n-1}, M_{n-2})\xra{d_{n-2}^*}\Hom_{\al C}(M_{n-1}, M_{n-1})
\xra {\overline{d_{n-1}}^*}\Hom_{\al C}(M_{n-1}, M_n\oplus M)\\\xra{\overline{g}^*}\Hom_{\al C}(M_{n-1}, M\oplus Y)
 \end{array}$$
is exact, there exists a morphism $h_{n-1}: M_{n-1}\ra M_{n-2}$ such that $\alpha_{n-1}=h_{n-1}d_{n-2}+\overline{d_{n-1}}h_n$.
By the induction, we have the following diagram:
$$
\xymatrix@M=1mm{
X\ar[r]^{f}\ar[d]^{\alpha} &M_1\ar@{..>}[dl]_{h_1}\ar[r]^{d_1}\ar[d]^{\alpha_1} & M_2\ar@{..>}[dl]_{h_2}\ar[r]^{d_2}\ar[d]^{\alpha_2} &\cdots\ar[r]^{d_{n-2}}&M_{n-1}\ar@{..>}[dl]_{h_{n-1}}\ar[d]^{\alpha_{n-1}}\ar[r]^{\overline{d_{n-1}}} &M_n\oplus M\ar[d]^{\alpha_n}\ar[r]^{\overline{g}} \ar@{..>}[dl]_{h_n}&Y\oplus M\ar@{..>}[d]^{\beta=0}\\
X\ar[r]^{f} &M_1\ar[r]^{d_1} &M_2\ar[r]^{d_2} &\cdots\ar[r]^{d_{n-2}}& M_{n-1}\ar[r]^{\overline{d_{n-1}}} \ar[r]^{\overline{d_{n-1}}} &M_n\oplus M\ar[r]^{\overline{g}} &Y\oplus M,\\
}
$$
and we get $\alpha_n=h_n\overline{d_n}, \alpha_{n-1}=h_{n-1}d_{n-2}+\overline{d_{n-1}}h_n$, $\alpha_i=h_id_{i-1}+d_ih_{i+1}$ for $i=2,\cdots, n-2$, $\alpha_1=h_1f+d_1h_2$ and
$\alpha=f h_1$.
Therefore, $(\alpha,\alpha_1,\cdots, \alpha_n)$ is null-homotopic in $\Kb{\al S}$.

So, by the above argument,  the following map
$$
 \begin{alignedat}{3}
\Theta: \End_{\Kb{\al S}}(\cpx{P})\lra \End_{\al S}(M\oplus Y)\\
\overline{(\alpha,\alpha_1,\cdots, \alpha_n)} \mapsto \beta
 \end{alignedat}
$$
is well-defined.  This map is injective, it remains to show that  $\Img(\Theta)=\End_{\al S}(M\oplus Y)$.
For any
$$
\beta=\begin{pmatrix}
\beta_1&\beta_2 \\
\beta_3 &\beta_4
\end{pmatrix}\in \begin{pmatrix}
\End_{\al S}(M) &\Hom_{\al S}(M,Y) \\
\Hom_{\al S}(Y,M) &\End_{\al S}(Y)
\end{pmatrix},
$$
it follows that there exist $\beta'_2\in\Hom(M,M_n)$, $\beta'_4\in\End(M_n)$ such that $\beta_2=\beta'_2g ,g\beta_4=\beta'_4g$.
Consequently,
$$
\begin{pmatrix}
1&0 \\
0 &g
\end{pmatrix}
\begin{pmatrix}
\beta_1&\beta_2 \\
\beta_3 &\beta_4
\end{pmatrix}=\begin{pmatrix}
\beta_1&\beta'_2g \\
g\beta_3 &g\beta_4
\end{pmatrix}=\begin{pmatrix}
\beta_1&\beta'_2 \\
gx_3 &\beta'_4
\end{pmatrix}
\begin{pmatrix}
1&0 \\
0 &g
\end{pmatrix}.
$$
Set $\alpha_n=\begin{pmatrix}
\beta_1&\beta'_2 \\
gx_3 &\beta'_4
\end{pmatrix}$.
Since $\overline{d_{n-1}}\alpha_n\overline{g}=0$, by the exact sequence $(\dagger)$ of the Definition \ref{higher} and induction, there exist maps
$\alpha_{n-1}: M_{n-1}\ra M_{n-1}$ such that $\alpha_{n-1}\overline{d_{n-1}}=\overline{d_{n-1}}\alpha_n$ for $2\leq i\leq n-2$, and there exists a unique map $\alpha\in \End_{\al S}(X)$ such that we have the following diagram
$$\xymatrix@M=1mm{
X\ar[r]^{f}\ar@{..>}[d]^{\alpha} &M_1\ar[r]^{d_1}\ar@{..>}[d]^{\alpha_1} & M_2\ar[r]^{d_2}\ar@{..>}[d]^{\alpha_2} &\cdots\ar[r]^{d_{n-2}}&M_{n-1}\ar@{..>}[d]^{\alpha_{n-1}}\ar[r]^{\overline{d_{n-1}}} &M_n\oplus M\ar@{..>}[d]^{\alpha_n}\ar[r]^{\overline{g}} &Y\oplus M\ar[d]^{\beta}\\
X\ar[r]^{f} &M_1\ar[r]^{d_1} &M_2\ar[r]^{d_2} &\cdots\ar[r]^{d_{n-2}}& M_{n-1}\ar[r]^{\overline{d_{n-1}}} \ar[r]^{\overline{d_{n-1}}} &M_n\oplus M\ar[r]^{\overline{g}} &Y\oplus M
.}$$
This implies that $\overline{(\alpha,\alpha_1,\cdots, \alpha_n)}\in \End_{\Kb{\al S}}(\cpx{P})$ and $\beta=\Theta(\overline{(\alpha,\alpha_1,\cdots, \alpha_n)})$, so the map $\Theta$ is surjective. Note that
$$
\End_{\mathcal {S}}(M\oplus Y)\simeq\begin{pmatrix}
\End_{\mathcal {S}}(M)&       \Hom_{\mathcal {S}}(M,Y)\\
\Hom_{\mathcal {S}}(Y,M)     & \End_{\mathcal {S}}(Y)
\end{pmatrix}=B,
$$
 This completes the proof. $\square$

\section{Locally finite algebras and locally tilting sets of complexes}\label{locallyfinite}
In this section, we  we shall introduce locally finite $\Phi$-algebras, where $\Phi$ is an
admissible subset of $\mathbb{Z}$ and study derived equivalences between them by locally tilting families of complexes.
Note that $\Phi$ is an infinite subset of  $\mathbb{Z}$.

\subsection{Locally finite algebras with enough idempotents}

In this subsection, we introduce locally finite algebras and show that there are locally finite tilting families for such algebras.

Let $R$ be a commutative ring and let $\mathcal {I}$ be an index set. For each $i, j\in \mathcal {I}$, let $\Lambda_i$ be an $R$-algebra and
$\Lambda_{ij}$ be a $\Lambda_i$-$\Lambda_j$-bimodule, with $\Lambda_{ii}=\Lambda_i$. For each $i,j,k\in \mathcal {I}$, we suppose that there is a $\Lambda_i$-$\Lambda_k$-bimodule homomorphism
$\mu_{ijk}:\Lambda_{ij}\otimes_{\Lambda_j}\Lambda_{jk}\lra \Lambda_{ik}$. We further also assume that $(\mu_{ijk}\otimes 1_{\Lambda_{kl}}) \circ\mu_{ikl}=(1_{\Lambda_{ij}} \otimes\mu_{jkl})\circ \mu_{ijl}$.
We impose the following finiteness conditions.

$(1)$ For each $i\in\mathcal{I}, \Lambda_{ij}=0$ for all but finitely many $j\in \mathcal {I}$;

$(1)$ For each $j\in\mathcal{I}, \Lambda_{ij}=0$ for all but finitely many $i\in \mathcal {I}$.

Let $\Lambda$ be the set of all $\mathcal {I}\times \mathcal {I}$ matrixes, $(\lambda_{ij})$, such that, for $i,j\in \mathcal {I}$, $a_{ij}\in \Lambda_{ij}$, and all but finite number of $\lambda_{ij}$ are zero.
The matrix addition and matrix multiplication defined above provide $\Lambda$ with $R$-algebra structure. Note that if $\mathcal {I}$ is an infinite set, then $\Lambda$ has no identity element.
For $i\in \mathcal {I}$, let $e_i$ be the matrix in $\Lambda$ having $1\in \Lambda_i$ in $(i,i)$-entry and $0$ in all other entries and denote $\mathcal{E}=\{e_i\}_{i\in \mathcal {I}}$. The algebra $\Lambda$ has enough orthogonal idempotents, namely, the elements of $\mathcal{E}$, since $\Lambda=\bigoplus_{i\in \mathcal {I}}\Lambda e_i=\bigoplus_{i\in \mathcal {I}}e_i\Lambda$. We say that the pair
$(\Lambda,\mathcal{E})$ is a locally finite $R$-algebra with respect to $\mathcal {I}$.
Let $\Lambda^u\Modc$ denote the category of left locally unital $\Lambda$-modules $X$ such that if $x\in X$ then there is a finite set $\mathcal {J}\subset \mathcal {I}$ such that $x(\bigoplus_{i\in \mathcal {J}}e_i)=x$.
Note that $\Lambda^u\Modc$ is a Grothendieck category. A Grothendieck category is a cocomplete abelian category with a generator
where filtered colimits are exact. We say that a set $\{Y_i\}_{i\in \mathcal {I}}$ is a locally finite set of $\Lambda$-modules if for each $i\in \mathcal {I}$, $Y_i$ is in $\Lambda^u\Modc$. A set $\{\cpx{Y_i}\}_{i\in \mathcal {I}}$ is a locally finite set of $\Lambda$-complexes if for each $i\in \mathcal {I}$, $\cpx{Y_i}$ is in $\C{\Lambda^u\Modc}$. Let $\D{\Lambda^u\Modc}$ denote the derived categories of locally unital $\Lambda$-modules.

\begin{rem}
$(1)$
The data of a locally finite algebra $A$ is the same as the data of a
small category $\mathcal {A}$ with object set $\mathcal {I}$ and morphisms $\Hom(i,j)=e_i Ae_j$. In this
incarnation, locally finite algebra homomorphisms correspond to functors. A left $A$-module becomes a functor from $\mathcal{A}$ to $\Ve$,  and then a module homomorphism is a natural
transformation of functors. For example, the projective module $Ae_i$  corresponds to the functor $\Hom(i,-): \mathcal{A} \ra \Ve$. Then the Yenoda Lemma asserts that there is a fully faithful
functor from $\mathcal {A}$ to $\proj{A}$ sending $i\in\mathcal{A}$ to $Ae_i$ and a morphism $a\in\Hom_{\mathcal{A}}(i,j)$ to the homomorphism $Ae_i\ra Ae_j$ defined by left multiplication by $a \in e_iAe_j$.
This extends canonically
to an equivalence of categories
$$
\widehat{\mathcal {A}}\ra \proj{A}
$$
where $\widehat{\mathcal {A}}$ denotes the additive Karoubi envelope of $\mathcal {A}$, that is, the idempotent completion
of the additive envelope of $\mathcal {A}$. Note that $\{Ae_i, i\in I\}$ is a projective generating family for $A$. By a projective generating family for an Abelian category $\mathcal {A}$, we mean a small family
$\{P_x\}_{x\in X}$ of finitely generated projective objects such that for each $V \in\mathcal {A}$, there is some
$x \in X$ with $\Hom(P_x, V)\neq 0$. Then we get the category $\mathcal{A}\Modc$ is then equivalent to $A^u\Modc$. We refer to \cite{BD,Wis} for more information about locally finite algebras.

$(2)$ The locally finite algebras have powerful application in noncommutative algebras (for example, see \cite{Mori})

\end{rem}

We introduce a locally finite tilting set of complexes called a locally finite tilting family which is a generalization of \cite[Definiition 3.3]{GH}.

\begin{defn} Let $(\Lambda,\mathcal{E})$ be a locally finite $R$-algebra. We say that $\{\cpx{T_i}\}$ is a locally finite tilting family if the following conditions are satisfied:

$(1)$ $\cpx{T_i}\in\Kb{\proj{\Lambda}}$ for each $i\in \mathcal{I}$;

$(2)$ $\Hom(\cpx{T_i}, \cpx{T_j}[n])=0$ for $n\neq 0$ and $i,j\in\mathcal{I}$;

$(3)$ $\thick(\cpx{T_i}, i\in \mathcal{I})=\Kb{\proj{\Lambda}}$.
\end{defn}

\begin{rem}  If $\mathcal{I}$ is finite, then $\bigoplus_{i\in\mathcal{I}}\cpx{T}_i$ is a tilting complex by \ref{R}.
If $\mathcal{I}$ is infinite, then $\bigoplus_{i\in\mathcal{I}}\cpx{T}_i$ is not a tilting complex.
\end{rem}

Let $\{\cpx{T_i}\}$ be a locally finite tilting complex set. For $i,j\in \mathcal{I}$,
define $\Sigma_{ij}=\Hom_{\K{\Lambda}}(\cpx{T_i}, \cpx{T_j})$ and let $\Sigma$ be the set of  all $\mathcal{I}\times \mathcal{I}$ matrixes, $(\sigma_{ij})$, such that, for $i,j\in \mathcal{I}$, $a_{ij}\in \Sigma_{ij}$, and all but finite number of $\sigma_{ij}$ are zero. $\Sigma$ is an $R$-algebra via matrix addition and matrix multiplication, where the structure maps $\mu_{ijk}:\Sigma_{ij}\otimes_{\Sigma_j}\Sigma_{jk}\lra \Sigma_{ik}$ are given by compositions.
For $i\in \mathcal{I}$, let $d_i=(\sigma_{lm})\in\Sigma$, where $\sigma_{Lm}=1_{\cpx{T_i}}$ for $i=l=m$  and $0$ otherwise and denote $\mathcal{D}=\{d_i\}_{i\in\mathcal{I}}$.

\begin{thm}\label{ll} Let $(\Lambda,\mathcal{E})$ be a locally finite $k$-algebra with respect to $\mathcal{I}$ and let $\{\cpx{T_i}\}_{i\in \mathcal{I}}$ be a locally finite tilting family of $\Lambda$-complexes. Let $(\Hom_{\K{\Lambda}}(\cpx{T_i}, \cpx{T_j}), \mathcal{D})$ be the locally finite endomorphism ring of $\{\cpx{T_i}\}$ with respect to $\mathcal{I}$ and denote
$\Gamma=(\Hom_{\K{\Lambda}}(\cpx{T_i}, \cpx{T_j}))_{\mathcal{I}\times \mathcal{I}}$. Then there is a triangle equivalence
$$
\D{\Lambda^u\Modc}\lra\D{\Gamma^u\Modc}$$
\end{thm}

\begin{proof} If  $(\Lambda,\mathcal{E})$ is a locally finite $k$-algebra with respect to $\mathcal{I}$, we can modify the structure of $k$-linear category structure $\mathcal{A}$ in the following way: As the set of objects we take the set $\mathcal{I}$, The space of morphisms between objects $\{i\},\{j\}$ is given by
$$
\Hom(\{i\},\{j\})=\Lambda_{ij}.$$
The category $\mathcal{A}\Modc$ is then equivalent to $\Lambda^u\Modc$. There is a tilting subcategory $\al U=\{U_i, i\in \mathcal{I}\}$ for $\mathcal{A}\Modc$, such that $U_i=\cdots\ra 0\ra (T^b_i,-)\cdots\ra (T^a_i,-)\ra 0\ra\cdots$ if $\cpx{T_i}$ of the form $0\ra T^a_i\ra \cdots\ra T^b_i\ra 0$ with integers $a\leq b$.
Let $$
\Gamma=\bigoplus_{i,j\in \mathcal{I}}\Hom_{\K{\Lambda}}(\cpx{T_i}, \cpx{T_j})
$$
viewed as a locally finite algebra with distinguished idempotents $(e_i:=1_{\cpx{T_i}})_{i\in \mathcal{I}}$.
Let $\mathcal {B}$ be the $R$-linear category with object set $\mathcal{I}$ and morphisms $\Hom(i,j)=e_i\Gamma e_j$, where the projective $\Gamma$-module $\Gamma e_i$ is correspondent to $U_i$.
Therefore, $\mathcal {B}$ is equivalent to the tilting subcategory $\al U$. The result follows from Theorem \ref{k}.
\end{proof}

\begin{rem} $(i).$ If $\mathcal{I}$ is a finite set, then this theorem is Rickard's construction in \cite{Ri1}.

$(ii).$ If $\D{\Lambda^u\Modc}\simeq \D{\Gamma^u\Modc}$ is an equivalence, then we say that the locally finite algebras $\Lambda$ and $\Gamma$ are derived equivalent.

\end{rem}

In the following, we will give an example to explicit our theorem.

\begin{ex}\cite[Example 4.3]{Hu2017} Let $k$ be a field, and let $Q$ be the infinite quiver
$$
\xymatrix{
\bullet &\ar[l]_(1){0}_{\alpha_1}_(0){1}\bullet &\bullet\ar[l]_(0){2}_{\alpha_2} &\bullet\ar[l]_(0){3}_{\alpha_3}&\ar[l]\cdots
}.
$$
A representation of $Q$ over $k$ is a collection of vector spaces $V_i$ for each vertex $i$ together with linear maps $f_{\alpha_i}: V_i\ra V_{i-1}$ for all $i$. Let $\mathcal{A}$ be the category of all finite dimensional representations $(V_i, f_{\alpha_{i+1}})_{i\geq 0}$ of $Q$
satisfying $f_{\alpha_i}f_{\alpha_{i-1}}=0$ for all $i>0$.
Let $P_0$ be the representation $k\lla 0\lla 0\lla\cdots$, and, for each $i>0$,  let $P_i$ be the representation $0\lla\cdots\lla k\llaf{1}k\lla 0\lla\cdots$, where the two $k$'s correspond to the vertices $i-1, i$. Then $\mathcal{A}$ is an abelian category with enough projective objects and  $P_i, i\geq 0$ are precisely those indecomposable projective objects in $\mathcal{A}$. Consider the following complexes over $\mathcal{A}$:
$$
\cpx{T}_i: \quad 0\lra P_0\lra \cdots\lra P_{i-1}\lra P_i\lra 0, \quad i\geq 0.$$
It is easy to check that $\{\cpx{T}_i|i\geq 0\}$ is a locally finite tilting family of $\Db{\mathcal{A}}$, that is, the following two conditions are satisfied.

a) $\Hom_{\Db{\mathcal{A}}}(\cpx{T}_i, \cpx{T}_j[l])=0$  for all $i, j\in\mathbb{N}$ and $ l\neq 0$;

b) $\thick\{\cpx{T}_i|i\geq 0\}=\Db{\mathcal{A}}$.

{\parindent=0pt The } locally finite tilting family $\{\cpx{T}_i|i\geq 0\}$  is equivalent as a category to the quiver $Q_T$:
$$\xymatrix{
\bullet\ar[r]^(0){0}^(1){1}^{\beta_1} &\bullet\ar[r]^(1){2}^{\beta_2} &\bullet\ar[r]^(1){3}^{\beta_3} &\bullet\ar[r] & \cdots
}$$
For each $i\geq 0$, let $P^*_i$ be the representation
$0\lra\cdots\lra 0\lra k\lraf{1}k\lraf{1}k\lra\cdots$, where the first $k$ corresponds to the vertex $i$.  Let $\mathcal{B}$ be the category of finitely generated representations of $Q_T$ over $k$. Then $\mathcal{B}$ is an abelian category with enough projective objects, and the indecomposable projective objects are $P_i^*, i\in\mathbb{N}$.  Note that $\gldim\mathcal{B}=1$ and $\Db{\mathcal{B}}=\Kb{\proj{B}}$.  By Theorem \ref{ll} or \cite[Theorem 3.6]{Keller2006},  there is a triangle equivalence $F: \Db{\mathcal{B}}\lra\Db{\mathcal{A}}$ sending $P_i^*$ to $\cpx{T}_i$ for all $i\in\mathbb{N}$.
\end{ex}

Let $A$ be a finite dimensional $k$-algebra. Denote by $D=\Hom_k(-,k)$ the standard duality on $A\modc$.
The repetitive algebra $\widehat{A}$ proposed by Hughes and Waschb\"{u}sch \cite{Huwa}, is a Frobenius algebra and always infinite-dimensional except in the trivial case $A=0$. We consider $\widehat{A}$ as the infinite matrix algebra, without identity

$$
\widehat{A}=
\begin{pmatrix}
\ddots &&& 0\\
\ddots &A_{i-1}\\
&  D(A)_{i-1} & A_i\\
& & D(A)_i & A_{i+1} \\
& & & D(A)_{i+1}&\ddots\\
0& & & &\ddots\\
\end{pmatrix},
$$
in which matrices have only finitely many non-zero entries, $A_i=A$ is placed on the main diagonal, $D(A)_i=D(A)$ for all $i\in\mathbb{Z}$, all the remaining entries are zero, and the multiplication is induced from the canonical maps $A\otimes_AD(A)\ra D(A)$, $D(A)\otimes_AA\ra D(A)$ and the zero map $D(A)\otimes_AD(A)\ra 0$.

As is known, repetitive algebras of finite dimensional algebras are locally bounded algebras.
Recall that an algebra $A$ is called locally bounded if there exists a complete set of pairwise orthogonal idempotents $\{e_x|x\in I\}$ such that $Ae_x$ and $e_xA$ are finite dimensional over a field $k$ for all $x\in I$. For a locally bounded algebra $A$, any finite generated $A$-module has finite length. In particular, $A\modc$ is an abelian category.  The following example modified the theorem of Chen \cite{CHen}.

\begin{ex} Let $A$ and $B$ be finite dimensional algebras.
Suppose that $A$ and $B$ are derived equivalent and that $\cpx{T}$ is a tilting complex over $A$ such that $\End(\cpx{T})\simeq B$.
Let $\widehat{A}$ and $\widehat{B}$ be repetitive algebras of $A$ and $B$, respectively.
Then $\widehat{A}e_i\otimes_A\cpx{T}$ is a locally finite tilting family of $\widehat{A}$, where $e_i$ is the matrix with $1\in A$ in $(i,i)$-entry, and $0$ in other entries such that $e_i\widehat{A}e_i\simeq A$ and $i\in\mathbb{Z}$. Therefore, we have the following

(1) $\widehat{A}e_i\otimes_A\cpx{T}$ is self-orthogonal.
$$\begin{aligned}
\Hom_{\Db{\widehat{A}}}(\widehat{A}e_i\otimes_A\cpx{T},\widehat{A}e_j\otimes_A\cpx{T}[n])\simeq
\Hom_{\Db{A}}(\cpx{T},(\widehat{A}e_i,\widehat{A}e_j\otimes_A\cpx{T}[n]))\\\simeq \Hom_{\Db{A}}(\cpx{T},e_i\widehat{A}e_j\otimes_A\cpx{T}[n])
\simeq \begin{cases}
B, & if~i=j,n=0.
\\D(B), & if~i+1=j,n=0.
\\0, & others.
\end{cases}
\end{aligned}
$$

(2)We have the following algebra isomorphism by the matrix multiplication,
$$
(\Hom_{\Db{\widehat{A}}}(\widehat{A}e_i\otimes_A\cpx{T},\widehat{A}e_j\otimes_A\cpx{T}))_{i,j\in\mathbb{Z}}\simeq\widehat{B}
$$

(3) Since $A\in\thick(\cpx{T})$, we have $\widehat{A}e_i\in\thick(\widehat{A}e_i\otimes_A\cpx{T})$. It follows that $\widehat{A}\in\thick(\widehat{A}e_i\otimes_A\cpx{T})$.

Then by Theorem \ref{k}, $\widehat{A}$ and $\widehat{B}$ are derived equivalent.
\end{ex}

\subsection{Locally finite $\Phi$-Beilinson-Green algebras}

In this subsection, we shall introduce the locally
$\Phi$-Green algebras, where $\Phi$ is an
admissible set of $\mathbb{Z}$.

Let  $\mathbb Z$ be the
set of all integers. Recall that a subset of $\mathbb Z$ containing $0$ is called an
admissible subset of $\mathbb Z$ is defined in \cite{Hu2013}.
Let $\Phi$ be a subset of $\mathbb Z$.
For an additive $R$-category $\al T$ and an endo-functor $F$ from $\al T$ to $\al T$, we recall the
definition of $\Phi$-Auslander-Yoneda $R$-algebras from \cite{Hu2013} in the following.
Let $\E^{i,F,\Phi}_{\al T}$ be the bi-functor
\begin{center}
$\Hom_{\al T}(-,F^i-): {\al T} \times {\al T}\lra \mathbb Z \Modc$
\[
(X,Y)\mapsto \E^{i,F,\Phi}_{\al T}(X,Y):=
\left\{
\begin{array}{cc}
\Hom_{\al T}(X,F^iY),& if\ \ i\in\Phi.\\
 0, & if\ \ i\notin \Phi.
\end{array} \right.
\]

$X\xra f X'\mapsto \Hom_{\al T}(f,F^i Y)$,   $Y\xra g Y'\mapsto \Hom_{\al T}(X,F^i g),$
\end{center}
and let
\begin{center}
$\E^{\Phi,i,F}_{\al T}(X,Y):=\bigoplus_{j\in \mathbb Z}\E^{j-i,F,\Phi}_{\al T}(X,Y)$.
\end{center}
Suppose that $X,Y$ and $Z$ are objects in $\al T$. Let
$f_i\in \E^{i, F, \Phi}_{\al T}(X,Y)$ and $g_j\in \E^{j, F, \Phi}_{\al
T}(Y,Z)$. The composition of $f_i$ and $g_j$ is defined as follows:

$$\E^{i, F, \Phi}_{\al T}(X,Y)\times \E^{j, F, \Phi}_{\al T}(Y,Z)\lra
\E^{i+j, F, \Phi}_{\al T}(X,Z)$$

\begin{equation}(f_i,g_j)\mapsto f_i\circ g_j=
\begin{cases}
f_i(F^i g_j), & if~i\in\Phi,~j\in\Phi~and~i+j\in\Phi.
\\0, & others.
\end{cases}
\end{equation}

%For $(f_i)_{i\in\Phi}\in\E^{\Phi,i,F}_{\al T}(X,Y)$ and $(g_i)_{i\in\Phi}\in \E^{\Phi,j,F}_{\al T}(Y,Z)$, the composition of $f,g$ is induced by above.

If $X=Y$, we write $\E^{F,\Phi}_{\al T}(X)$ for $\E^{F,\Phi}_{\al T}(X,X)$. Set $\E^{F,\Phi}_{\al T}(X)=\bigoplus_{i\in\mathbb Z}
\E^{i,F,\Phi}_{\al T}(X)$. In case $\al T$ is a triangulated category and $F=[1]$, we denote $\E^{F,\Phi}_{\al T}(X,Y)$ and  $\E^{F,\Phi}_{\al
T}(X)$ by $\E^{\Phi}_{\al T}(X,Y)$ and $\E^\Phi_{\al T}(X)$, respectively.
 In \cite{Hu2013}, Hu and Xi proved that $\Phi$ is an
admissible subset in $\mathbb Z$ if
and only if $\E^{\Phi}_{\al T}(X)$ is an associated algebra.
It is called the $\Phi$-Auslander--Yoneda algebra of $X$ \cite{Hu2013}.
Recall that in \cite{PZ}, the $\Phi$-Beilinson-Green algebra is defined provided that $\Phi$ be a finite admissible subset of $\mathbb{Z}$. In this paper, we will consider that $\Phi$ is an infinite set of $\mathbb{Z}$.
Let $R$ be a commutative ring and $\Phi$ be an admissible subset of $\mathbb{Z}$.
Then we define the locally $\Phi$-Beilinson-Green algebras $\scr G^{\Phi,F}
(X)$ for an object $X$ in $\al T$, and mention some basic properties of these algebras. Let $s<0<m$.
Firstly, let us define an $R$-module $\scr G^{\Phi,F}
(X)$ as follows:
$$
\scr G^{\Phi,F}(X)=$$
$$
\begin{pmatrix}
\vdots & \vdots & \vdots & \vdots & \vdots & \vdots& \vdots& \vdots\\
\cdots  & \E^{F}(X)_{s+1,s+1}  &   \E^{F}(X)_{s+1,s+2}       & \cdots &  \cdots & \cdots & \E^{F}(X)_{s+1,m-1}& \E^{F}(X)_{s+1,m}&  \cdots\\
\cdots  & \E^{F}(X)_{s+1,s+2}  &   \E^{F}(X)_{s+2,s+2}       & \cdots &  \cdots & \cdots & \E^{F}(X)_{s+2,m-1}& \E^{F}(X)_{s+2,m}&  \cdots\\
\vdots & \vdots & \vdots & \vdots & \vdots & \vdots& \vdots& \vdots\\
\cdots   & \E^{F}(X)_{0,s+1}& \cdots & \E^{F}(X)_{0,0}   & \E^{F}(X)_{0,1}  & \cdots &\E^{F}(X)_{0,m-1} & \E^{F}(X)_{0,m}&  \cdots\\
\cdots  & \E^{F}(X)_{1,s+1} &\cdots & \E^{F}(X)_{1,0}   &  \E^{F}(X)_{1,1}&\cdots & \E^{F}(X)_{1,m-1} & \E^{F}(X)_{1,m}&  \cdots\\
\vdots & \vdots & \vdots & \vdots & \vdots & \vdots& \vdots& \vdots\\
\cdots  & \E^{F}(X)_{m-1,s+1} &\cdots & \E^{F}(X)_{m-1,0}& \E^{F}(X)_{m-1,1} &
\cdots & \E^{F}(X)_{m-1,m-1} &\E^{F}(X)_{m-1,m}&  \cdots\\
\cdots  & \E^{F}(X)_{m,s+1} &
\cdots & \E^{F}(X)_{m,0} &\E^{F}(X)_{m,1}&\cdots &\E^{F}(X)_{m,m-1} &\E^{F}(X)_{m,m}&  \cdots\\
\vdots & \vdots & \vdots & \vdots & \vdots & \vdots& \vdots& \vdots\\
\end{pmatrix},
$$
where $\E^{F}(X)_{i,j}=\E^{i-j,F,\Phi}_{\al T}(X)$ is defined as above, $i,j\in \Phi$.
That is, $\scr G^{\Phi,F}(X)=(\E^{i-j,F,\Phi}_{\al T}(X))_{i,j\in \Phi}$. Note that if
$i-j\notin\Phi$, then $\E^{i-j,F,\Phi}_{\al T}(X)=0$.
Secondly, for any $x=(x_{ij})_{i,j\in \Phi},y=(y_{ij})_{i,j\in \Phi} \in\scr G^{\Phi,F} (X)$,
the multiplication of $x$ and $y$ is defined as follows: $xy=z=(z_{lt})_{l,t\in \Phi}$, where
$$
z_{lt}=\sum_{k=t}^l x_{lk}F^{l-k}y_{kt}
=\sum_{\substack{k-t\in\Phi\\ l-k\in\Phi}}x_{lk}F^{l-k}y_{kt}.
$$

\begin{rem} $(i).$ Recall that the triangular matrix algebra of a graded algebra of above form seems first to appear in the paper \cite{G} by Edward L. Green in 1975.
A special case of this kind of algebras appeared in \cite{Bei} by A. A. Beilinson in 1978. Perhaps it is more appropriate to name this triangular matrix algebra as the
$\Phi$-Beilinson-Green algebra of $X$.

$(ii).$ If $\Phi\neq 0$, then the locally $\Phi$-Beilinson-Green algebra has no unit, but has a locally unit,
let $\scr G^{\Phi,F}(X)^u\Modc$ denote the abelian category of all left locally unital $\scr G^{\Phi,F}(X)$-modules, recall that, a $\scr G^{\Phi,F}(X)$-module $X$ is locally unital if $x\in X$ then there is a finite set $\mathcal{J}\subset \Phi$ such that $x(\bigoplus_{i\in\mathcal{J}}e_i)=x$.  Denote by $\D{\scr G^{\Phi,F}(X)^u\Modc}$ the derived category of locally unital $\scr G^{\Phi,F}(X)$-modules.

\end{rem}

The following fact of the locally finite $\Phi$-Beilinson-Green algebra $\scr G^{\Phi,F}(X)$ is useful, which can be easily checked.
Let $e_i$ be the following matrix with $1_{\End_{\al T}(X)}$ in $(i,i)$-entry, and $0$ in other entries, that is,
$$e_i=
\begin{pmatrix}
 \vdots & \vdots & \vdots &   \vdots& \vdots& \vdots&\vdots&\vdots\\
 \cdots &0   & 0   &  0  & \cdots & 0     & 0 &\cdots \\
 \vdots & \vdots & \vdots &   \vdots& \vdots& \vdots&\vdots&\vdots\\
 \cdots &0 & 0 & 0& \cdots & 0 & 0 &\cdots &\\
 \cdots &0 & 0 & 0& \cdots & 1_{\End_{\al T}(X)} & 0 &\cdots &\\
 \cdots &0 & 0 & 0& \cdots & 0 & 0 &\cdots &\\
\vdots & \vdots & \vdots & \vdots & \vdots & \vdots& \vdots& \vdots\\
 \cdots &0   & 0   &  0  & \cdots & 0     & 0 &\cdots \\
 \vdots & \vdots & \vdots &   \vdots& \vdots & \vdots& \vdots&\vdots\\\
 \end{pmatrix}.$$ Therefore,  as a left $\scr G^{\Phi,F}(X)$-module, $\scr G^{\Phi,F}(X)e_i\cong \E^{\Phi,i,F}(X),$ where $\E^{\Phi,i,F}(X)=\oplus_{j\in \Phi} \E^{F}(X)_{j,i}$.
Then $\scr G^{\Phi,F}(X)\cong \oplus_{i\in \Phi} \E^{\Phi,i,F}
(X)$ as left $\scr G^{\Phi,F}(X)$-modules.

The following lemma is essentially taken from \cite[Lemma\;3.5]{Hu2013} by its variation, the proof given there carries over to the present situation.

\begin{lem}\label{llx}
Let $\Phi$ be an admissible subset of $\mathbb{Z}$ and let $X$ be an object in $\al T$. Assume that
$X_1$, $X_2$, $X_3\in \add X$. Then we have the following:

\begin{itemize}

\item[{\rm (1)}] The $\scr G^{\Phi,F}(X)$-module $\E^{\Phi,i,F}(X,X_k)$ is finitely generated projective,
for any $0\leq i\leq m$ and $k=1,2,3$.

\item [{\rm (2)}]There is a natural isomorphism $$\mu: \E^{i-j,F%\Phi}_{\al T}
}(X_1,X_2)\lra \Hom_{\scr G^{\Phi,F}(X)}
(\E^{\Phi,i,F}(X,X_1),\E^{\Phi,j,F}(X,X_2)),$$
which sends $x\in \E^{i-j,F}(X_1,X_2)$ to the morphism $(x)\mu:
\E^{\Phi,i,F}(X,X_1)\lra \E^{\Phi,j,F}(X,X_2)$, which maps
$(f_k)$ to $(f_kF^{k-i}(x))$.

\item [{\rm (3)}] If $x\in \E^{i-j,F}
(X_1,X_2)$ and $y\in \E^{j-k,F}(X_2,X_3)$, then $(xF^{i-j}(y))\mu=
(x)\mu (F^{i-j}(y))\mu$.
\end{itemize}
\end{lem}

\section{Derived equivalences between the quotient algebras of locally $\Phi$-Beilinson-Green algebras}\label{Derivedequiva}

Let $\Phi$ be an admissible subset of $\mathbb{Z}$. With the notations in hands,
we can give the following theorem which is the one of the main results of this paper.

\begin{thm}\,\label{Theorem-ghost}
Let $(\al C,\mathbb{E},\mathfrak{s})$ be an $n$-exangulated $k$-category ($n\geq 1$) with an $n$-exangle endo-functor $F$,
and let $M$ be an object in $\mathcal {C}$. Suppose that $\mathcal {C}(M, F^iX)=0=\mathcal {C}(Y,F^iM))$ for all $0\neq i\in\Phi$. Let
$$
X\xra f M_1\ra M_2\ra\cdots\ra M_n\xra g Y\draf{\delta}
$$
be an $n$-$\mathbb{E}$-exangle in $\mathcal {C}$ with $M_j\in\add(M)$ for all $j=1,\cdots,n$, such that $f$ is a left $\add_{\Oc{\al C}{F}{\Phi}}(M)$-approximation
of $X$ and that $g$ is a right $\add_{\Oc{\al C}{F}{\Phi}}(M)$-approximation of $Y$ in $\Phi$-orbit category $\Oc{\al C}{F}{\Phi}$.
Then the quotient rings of locally $\Phi$-Beilinson-Green algebras
$$
\frac{\scr G^{\Phi,F}(X\oplus M)}{I}
\quad\text{and} \quad\frac{\scr G^{\Phi,F}(M\oplus Y)}{J}
$$ are derived equivalent, where $I=\diag(\cdots,\Fcogh_M (X\oplus M),\cdots)$ and $J=\diag(\cdots,\Fgh_M (Y\oplus M),\cdots)$ are ideals of $\scr G^{\Phi,F}(X\oplus M)$ and $\scr G^{\Phi,F}(M\oplus Y)$, respectively.
\end{thm}

\begin{rem}
In particular, if $\Phi=0$, then the quotient algebras
$\frac{\End(X\oplus M)}{\Fcogh_M (X\oplus M)}$ and $\frac{\End(M\oplus Y)}{\Fgh_M (Y\oplus M)}$
are derived equivalent. As is known, an $n$-angulated category is an $n$-exangulated category,
then this result generalizes Chen and Hu's result for $n$-angulated categories.
\end{rem}

\subsection{The locally $\Phi$-Beilinson-Green algebras of orbit categories and the quotient algebras of locally $\Phi$-Beilinson-Green algebras}

\medskip

The following lemma is very useful to understand the ideals of factorizable ghosts and of factorizable coghosts.

\begin{lem}\label{fg}
Let $(\mathcal{C},\mathbb{E},\mathfrak{s})$ be an $n$-exangulated $k$-category ($n\geq 1$) with an $n$-exangle endo-functor $F$,
and let $M$ be an object in $\mathcal {C}$. Suppose that $\Phi$ is an admissible subset of $\mathbb{Z}$.  Let
$$
 X\xra f M_1\ra M_2\ra\cdots\ra M_{n}\xra g Y\draf{\delta}
$$
 be an $n$-$\mathbb{E}$-exangle in $\mathcal {C}$ with $M_j\in\add(M)$ for all $j=1,\cdots,n$. Suppose that $\mathcal{D}=\add_{{\mathcal{C}}^{F,\Phi}}(M)$. Then we have

(1). If $g$ is a right $\mathcal{D}$-approximation in $\Oc{\mathcal{C}}{F}{\Phi}$, and $\mathcal{C}(Y, F^iM)=0$ for all $0\neq i \in \Phi$,  then $\Fgh_{\mathcal{D}}(Y\oplus M)=\Fgh_M(Y\oplus M)$;

(2).  If $f$ is a left $\mathcal{D}$-approximation in $\Oc{\mathcal{C}}{F}{\Phi}$, and $\mathcal{C}(M, F^iX)=0$ for all $0\neq i\in\Phi$, then $\Fcogh_{\mathcal{D}}(X\oplus M)=\Fcogh_M(X\oplus M)$.
\end{lem}

\begin{proof} We revise the proof of \cite[Lemma 4.5]{CH} to our case.

(1).  Set $\tilde{g}=\left[\begin{smallmatrix}g&0\\0&1\end{smallmatrix}\right]: M_{n-2}\oplus M\lra Y\oplus M$.  By Lemma \ref{2.2}(3), we have $\Fgh_{\al D}(M, Y\oplus M)=0$.  Now we consider $\Fgh_{\al D}(Y, Y\oplus M)$. Since ${g}$ is a right ${\al D}$-approximation, by Lemma \ref{2.2} (1),  $\Fgh_{\al D}(Y, Y\oplus M)$ consists of morphisms $x:=(x_i)\in{\al C}^{F, \Phi}(Y, Y\oplus M)$ such that the composite $gx$ in $\Oc{\al C}{F}{\Phi}$ vanishes, or equivalently $g*x_i=gx_i=0$ in ${\al C}$ for all $i\in \Phi$. Since $\tilde{g}$ is also a right ${\al D}$-approximation, a morphism $x:=(x_i)\in{\al C}^{F, \Phi}(Y, Y\oplus M)$  factorizes through an object in $\mathcal{D}$ if and only if it factorizes through $\tilde{g}$. Thus $\Fgh_{\al D}(Y, Y\oplus M)$ consists of precisely those morphisms $(x_i)\in{\al C}^{F, \Phi}(Y, Y\oplus M)$ satisfying the conditions:

\smallskip
 (a). $gx_i=0$ for all $i\in\Phi$;

 (b). There is some $(y_i)\in \Oc{\al C}{F}{\Phi}(Y, M_{n-2}\oplus M)$ such that $y_i*\tilde{g}=x_i$ for all $i\in\Phi$.

\smallskip
{\parindent=0pt Note that, by} our assumption that $\mathcal{C}(Y, F^iM)=0$ for all $0\neq i\in\Phi$, the morphism $y_i$ in condition (b) above is zero for all $0\neq i\in\Phi$, and correspondingly $x_i=0$ for all $0\neq i\in\Phi$. Then $\Fgh_{\al D}(Y, Y\oplus M)$ actually consists of morphisms $(x_i)\in{\al C}^{F, \Phi}(Y, Y\oplus M)$ with $x_i=0$ for all $0\neq i\in\Phi$ such that $gx_0=0$ and $x_0=y_0\tilde{g}$ for some $y_0: Y\ra M_n\oplus M$ in $\mathcal{C}$. This is equivalent to saying that $x_0$ factorizes through $w$ and $\add(M)$ in ${\al C}$. Hence $\Fgh_{\al D}(Y\oplus M, Y\oplus M)=\Fgh_M(Y\oplus M)$ and the statement (1) is proved.

$(2)$. The proof of (2) is dual.
\end{proof}

Set $U=M\oplus X$ and $V=M\oplus Y$. Now consider the locally $\Phi$-Beilinson-Green algebras of
$U=M\oplus X$ and $V=M\oplus Y$ in $\scr G^{F,\Phi}/\Fgh_{\al D''}$, where $\al D''=\add_{\scr G^{F,\Phi}}(M)$.
Therefore, by Lemma \ref{fg}, the locally $\Phi$-Beilinson-Green algebras of $U$ and $V$ in $\scr G^{F,\Phi}/\Fgh_{\al D''}$ are the following
$$
\End_{\scr G^{F,\Phi}/\Fgh_{\al D''}}(U)=
\begin{pmatrix}
 \vdots & \vdots & \vdots & \vdots & \vdots& \vdots& \vdots\\
 \cdots&\frac{\End_{\al C}(U)}{\Fcogh_M (U)}     & \E^{F}(U)_{s,s+1}                  &  \cdots& \E^{F}(U)_{s,m-1}& \E^{F}(U)_{s,m}&\cdots\\
\vdots &  \vdots & \vdots & \vdots & \vdots&\vdots&\vdots\\
\cdots&\E^{F}(U)_{0,s} & \cdots & \frac{\End_{\al C}(U)}{\Fcogh_M (U)}  &  \cdots & \E^{F}(U)_{0,m}&\cdots\\
\vdots &  \vdots & \vdots & \vdots & \vdots&  \vdots&\vdots \\
\cdots&\E^{F}(U)_{m,s}  & \cdots &\cdots &\E^{F}(U)_{m,m-1} &\frac{\End_{\al C}(U)}{\Fcogh_M (U)}&\cdots\\
\vdots &  \vdots & \vdots & \vdots & \vdots&  \vdots&\vdots \\
\end{pmatrix},
$$
where $\E^{F}(U)_{i,j}=\E^{i-j,F,\Phi}_{\al C}(U)$, and

$$
\End_{\scr G^{F,\Phi}/\Fgh_{\al D''}}(V)=\begin{pmatrix}
\vdots &  \vdots & \vdots & \vdots & \vdots & \vdots\\
\cdots  &\frac{\E_{\al C}(V)}{\Fgh_M (V)}     & \E^{F}(V)_{s,s+1}     &  \cdots& \E^{F}(V)_{s,m-1}&\cdots  \\
\vdots &  \vdots & \vdots & \vdots & \vdots & \vdots\\
\cdots  &\E^{F}(V)_{0,s} & \cdots & \frac{\E_{\al C}(V)}{\Fgh_M (V)}  &  \cdots &\cdots\\
\vdots &\vdots &  \vdots & \vdots & \vdots & \vdots&   \\
\cdots  &\E^{F}(V)_{m,s}  & \cdots &\cdots &\frac{\E_{\al C}(V)}{\Fgh_M (V)} &\cdots\\
\vdots &  \vdots & \vdots & \vdots & \vdots & \vdots
\end{pmatrix},
$$
where $\E^{F}(V)_{i,j}=\E^{i-j,F,\Phi}_{\al C}(V)$.

In the following, we will show that locally $\Phi$-Beilinson-Green algebras in orbit categories are the quotient algebras by some ideals.

Let
$$
I=
\begin{pmatrix}
\vdots &\vdots & \vdots & \vdots & \vdots & \vdots & \vdots& \vdots& \vdots\\
\cdots &\Fcogh_M (U) & 0    &   0               & \cdots & \cdots  &  \cdots& 0& \cdots\\
\cdots &0& \Fgh_M (U)  &   0       & \cdots &  \cdots & \cdots & 0& \cdots\\
\vdots & \vdots & \vdots & \vdots & \vdots & \vdots& \vdots& \vdots\\
\cdots &0 & 0& \cdots & \Fcogh_M (U)   & \cdots & \cdots &0 & \cdots\\
\cdots &0 & 0 &\cdots & 0   &  \Fcogh_M (U) &\cdots & 0& \cdots \\
\vdots & \vdots & \vdots & \vdots & \vdots & \vdots& \vdots& \vdots\\
\cdots &0 &0 &\cdots & 0& 0 &
\cdots & \Fcogh_M (U)& \cdots \\
\vdots &\vdots & \vdots & \vdots & \vdots & \vdots & \vdots& \vdots& \vdots
\end{pmatrix},
$$ where $\Fcogh_M (U) $ are in $(i,i)$-entries for all $i\in\Phi$,
and
$$
\scr G^{F,\Phi}(U)=
\begin{pmatrix}
 \vdots & \vdots & \vdots & \vdots & \vdots& \vdots& \vdots\\
 \cdots&\End_{\al C}(U)     & \E^{F}(U)_{s,s+1}                  &  \cdots& \E^{F}(U)_{s,m-1}& \E^{F}(U)_{s,m}&\cdots\\
\vdots &  \vdots & \vdots & \vdots & \vdots&\vdots&\vdots\\
\cdots&\E^{F}(U)_{0,s} & \cdots & \End_{\al C}(U)  &  \cdots & \E^{F}(U)_{0,m}&\cdots\\
\vdots &  \vdots & \vdots & \vdots & \vdots&  \vdots&\vdots \\
\cdots&\E^{F}(U)_{m,s}  & \cdots &\cdots &\E^{F}(U)_{m,m-1} &\End_{\al C}(U)&\cdots\\
\vdots &  \vdots & \vdots & \vdots & \vdots&  \vdots&\vdots \\
\end{pmatrix},
$$
where $\End_{\al C}(U) $ are in $(i,i)$-entries for all $i\in\Phi$.
And let
$$
J=
\begin{pmatrix}
\vdots &\vdots & \vdots & \vdots & \vdots & \vdots & \vdots& \vdots& \vdots\\
\cdots &\Fgh_M (V) & 0    &   0               & \cdots & \cdots  &  \cdots& 0& \cdots\\
\cdots &0& \Fgh_M (V)  &   0       & \cdots &  \cdots & \cdots & 0& \cdots\\
\vdots & \vdots & \vdots & \vdots & \vdots & \vdots& \vdots& \vdots\\
\cdots &0 & 0& \cdots & \Fgh_M (V)   & \cdots & \cdots &0 & \cdots\\
\cdots &0 & 0 &\cdots & 0   &  \Fgh_M (V) &\cdots & 0& \cdots \\
\vdots & \vdots & \vdots & \vdots & \vdots & \vdots& \vdots& \vdots\\
\cdots &0 &0 &\cdots & 0& 0 &
\cdots & \Fgh_M (V)& \cdots \\
\vdots &\vdots & \vdots & \vdots & \vdots & \vdots & \vdots& \vdots& \vdots
\end{pmatrix},
$$
where $\Fgh_M (M\oplus Y) $ are in $(i,i)$-entries for all $i\in\Phi$,
and
$$
\scr G^{F,\Phi}(V)=\begin{pmatrix}
 \vdots & \vdots & \vdots & \vdots & \vdots& \vdots& \vdots\\
 \cdots&\End_{\al C}(V)     & \E^{F}(V)_{s,s+1}                  &  \cdots& \E^{F}(V)_{s,m-1}& \E^{F}(V)_{s,m}&\cdots\\
\vdots &  \vdots & \vdots & \vdots & \vdots&\vdots&\vdots\\
\cdots&\E^{F}(V)_{0,s} & \cdots & \End_{\al C}(V)  &  \cdots & \E^{F}(V)_{0,m}&\cdots\\
\vdots &  \vdots & \vdots & \vdots & \vdots&  \vdots&\vdots \\
\cdots&\E^{F}(V)_{m,s}  & \cdots &\cdots &\E^{F}(V)_{m,m-1} &\End_{\al C}(V)&\cdots\\
\vdots &  \vdots & \vdots & \vdots & \vdots&  \vdots&\vdots \\
\end{pmatrix},
$$
where $\End_{\al C}(V) $ are in $(i,i)$-entries for all $i\in\Phi$.

\begin{lem} \label{3.3}
The $I$ and $J$ are ideals of $\scr G^{F,\Phi}(U)$ and $\scr G^{F,\Phi}(V)$, respectively.
\end{lem}

\begin{proof}
Let
$$\begin{pmatrix}
 \vdots & \vdots & \vdots & \vdots & \vdots& \vdots& \vdots& \vdots\\
 \cdots & 0 & x_{ss}    &   0  & \cdots & \cdots  &  \cdots& \cdots\\
\vdots & \vdots & \vdots & \vdots & \vdots& \vdots& \vdots& \vdots\\
 \cdots& 0 & \cdots & x_{00}   & \cdots & \cdots &0 &\cdots \\
 \vdots & \vdots & \vdots & \vdots & \vdots& \vdots& \vdots& \vdots\\
 \cdots & 0 &\cdots & 0& 0 &
\cdots &x_{mm} &\cdots\\
 \vdots & \vdots & \vdots & \vdots & \vdots& & \vdots & \vdots
   \end{pmatrix},
$$
be a element of $I$, and $$
\begin{pmatrix}
\vdots &  \vdots & \vdots & \vdots & \vdots& \vdots& \vdots\\
\cdots &u_{ss}    & u_{s,s+1}                  &  \cdots& u_{s,m-1}& u_{sm}&\cdots\\
\vdots &  \vdots & \vdots &\vdots &  \vdots & \vdots & \vdots \\
\cdots &u_{0s} & \cdots & u_{00} &  \cdots & u_{0m}& \cdots\\
\vdots &  \vdots & \vdots & \vdots & \vdots&  \vdots & \vdots \\
\cdots&u_{ms}  & \cdots &\cdots &u_{m,m-1} &u_{mm}&\cdots\\
\vdots &  \vdots & \vdots & \vdots & \vdots& \vdots& \vdots\\
\end{pmatrix},
$$
be an element of $\scr G^{F,\Phi}(U)$.
Then

$$\begin{pmatrix}
 \vdots & \vdots & \vdots & \vdots & \vdots& \vdots& \vdots& \vdots\\
 \cdots & 0 & x_{ss}    &   0  & \cdots & \cdots  &  \cdots& \cdots\\
\vdots & \vdots & \vdots & \vdots & \vdots& \vdots& \vdots& \vdots\\
 \cdots& 0 & \cdots & x_{00}   & \cdots & \cdots &0 &\cdots \\
 \vdots & \vdots & \vdots & \vdots & \vdots& \vdots& \vdots& \vdots\\
 \cdots & 0 &\cdots & 0& 0 &
\cdots &x_{mm} &\cdots\\
 \vdots & \vdots & \vdots & \vdots & \vdots& & \vdots & \vdots
   \end{pmatrix}
\begin{pmatrix}
\vdots &  \vdots & \vdots & \vdots & \vdots& \vdots& \vdots\\
\cdots &u_{ss}    & u_{s,s+1}                  &  \cdots& u_{s,m-1}& u_{sm}&\cdots\\
\vdots &  \vdots & \vdots &\vdots &  \vdots & \vdots & \vdots \\
\cdots &u_{0s} & \cdots & u_{00} &  \cdots & u_{0m}& \cdots\\
\vdots &  \vdots & \vdots & \vdots & \vdots&  \vdots & \vdots \\
\cdots&u_{ms}  & \cdots &\cdots &u_{m,m-1} &u_{mm}&\cdots\\
\vdots &  \vdots & \vdots & \vdots & \vdots& \vdots& \vdots\\
\end{pmatrix}
$$
$$
=\begin{pmatrix}
\vdots &  \vdots & \vdots & \vdots & \vdots& \vdots& \vdots\\
\cdots&x_{ss}u_{ss}    & x_{ss}u_{s,s+1}    &  \cdots& x_{ss}u_{s,m-1}& x_{ss}u_{sm}&\cdots\\
\vdots &  \vdots & \vdots & \vdots & \vdots& \vdots& \vdots\\
\cdots&x_{00}u_{0s} & \cdots &x_{00} u_{00} &  \cdots & x_{00}u_{0m}&\cdots\\
\vdots &  \vdots & \vdots & \vdots & \vdots& \vdots& \vdots\\
\cdots&x_{mm}u_{m,s}  & \cdots &\cdots &x_{mm}u_{m,m-1} &x_{mm}u_{mm}&\cdots\\
\vdots &  \vdots & \vdots & \vdots & \vdots& \vdots& \vdots\\
\end{pmatrix}.$$

Then we get $x_{ii}u_{ij}=0$ for $i\neq j$, and $x_{ii}u_{ii}\in \Fcogh_M (U)$, since $\Fcogh_M (U)$ is an ideal of $\End_{\al C}(U)$.
Similarly,
$$
\begin{pmatrix}
\vdots &  \vdots & \vdots & \vdots & \vdots& \vdots& \vdots\\
\cdots &u_{ss}    & u_{s,s+1}                  &  \cdots& u_{s,m-1}& u_{sm}&\cdots\\
\vdots &  \vdots & \vdots &\vdots &  \vdots & \vdots & \vdots \\
\cdots &u_{0s} & \cdots & u_{00} &  \cdots & u_{0m}& \cdots\\
\vdots &  \vdots & \vdots & \vdots & \vdots&  \vdots & \vdots \\
\cdots&u_{ms}  & \cdots &\cdots &u_{m,m-1} &u_{mm}&\cdots\\
\vdots &  \vdots & \vdots & \vdots & \vdots& \vdots& \vdots\\
\end{pmatrix}
\begin{pmatrix}
 \vdots & \vdots & \vdots & \vdots & \vdots& \vdots& \vdots& \vdots\\
 \cdots & 0 & x_{ss}    &   0  & \cdots & \cdots  &  \cdots& \cdots\\
\vdots & \vdots & \vdots & \vdots & \vdots& \vdots& \vdots& \vdots\\
 \cdots& 0 & \cdots & x_{00}   & \cdots & \cdots &0 &\cdots \\
 \vdots & \vdots & \vdots & \vdots & \vdots& \vdots& \vdots& \vdots\\
 \cdots & 0 &\cdots & 0& 0 &
\cdots &x_{mm} &\cdots\\
 \vdots & \vdots & \vdots & \vdots & \vdots& & \vdots & \vdots
   \end{pmatrix},
$$
$$
=\begin{pmatrix}
 \vdots & \vdots & \vdots & \vdots & \vdots& \vdots& \vdots\\
 \cdots&u_{ss}x_{ss}    & u_{s,s+1} F^1x_{s+1,s+1}                 &  \cdots& u_{s,m-1}F^{m-1-s}x_{m-1,m-1}& u_{sm}F^{m-s}x_{mm}&\cdots\\
 \vdots & \vdots & \vdots & \vdots & \vdots& \vdots& \vdots\\
\cdots&u_{0s}F^sx_{0s} & \cdots & u_{00}x_{00} &  \cdots & u_{0m}F^mx_{mm}&\cdots\\
\vdots & \vdots & \vdots & \vdots & \vdots& \vdots& \vdots\\
\cdots&u_{m,s}F^{s-m}x_{ss}  & \cdots &\cdots &u_{m,m-1}F^{-1}x_{m-1,m-1} &u_{mm}x_{mm}&\cdots\\
 \vdots & \vdots & \vdots & \vdots & \vdots& \vdots& \vdots\\
 \end{pmatrix}.$$

Then we get $u_{ij}F^{i-j}x_{jj}=0$ for $i\neq j$, and $u_{ii}x_{ii}\in \Fcogh_M (U)$, since $\Fcogh_M (U)$ is an ideal of $\End_{\al C}(U)$.
\end{proof}

Thus we have a quotient algebra of $\scr G^{F,\Phi}(U)$ by $I$,
$$\frac{\scr G^\Phi(U)}{I}=\begin{pmatrix}
 \vdots & \vdots & \vdots & \vdots & \vdots& \vdots& \vdots\\
 \cdots&\frac{\End_{\al C}(U)}{\Fcogh_M (U)}      & \E^{F}(U)_{s,s+1}                  &  \cdots& \E^{F}(U)_{s,m-1}& \E^{F}(U)_{s,m}&\cdots\\
\vdots &  \vdots & \vdots & \vdots & \vdots&\vdots&\vdots\\
\cdots&\E^{F}(U)_{0,s} & \cdots & \frac{\End_{\al C}(U)}{\Fcogh_M (U)}  &  \cdots & \E^{F}(U)_{0,m}&\cdots\\
\vdots &  \vdots & \vdots & \vdots & \vdots&  \vdots&\vdots \\
\cdots&\E^{F}(U)_{m,s}  & \cdots &\cdots &\E^{F}(U)_{m,m-1} &\frac{\End_{\al C}(U)}{\Fcogh_M (U)}&\cdots\\
\vdots &  \vdots & \vdots & \vdots & \vdots&  \vdots&\vdots \\
\end{pmatrix}.
$$
Consequently, $\End_{\scr G^{F,\Phi}/\Fgh_{\al D''}}(U)=\frac{\scr G^\Phi(U)}{I}$.
Similarly, we have a quotient algebra
$\scr G^\Phi(V)$ by $J$, where
$$
\frac{\scr G^\Phi(V)}{J}=\begin{pmatrix}
\vdots &  \vdots & \vdots & \vdots & \vdots & \vdots\\
\cdots  &\frac{\E_{\al C}(V)}{\Fgh_M (V)}     & \E^{F}(V)_{s,s+1}     &  \cdots& \E^{F}(V)_{s,m-1}&\cdots  \\
\vdots &  \vdots & \vdots & \vdots & \vdots & \vdots\\
\cdots  &\E^{F}(V)_{0,s} & \cdots & \frac{\E_{\al C}(V)}{\Fgh_M (V)}  &  \cdots &\cdots\\
\vdots &\vdots &  \vdots & \vdots & \vdots & \vdots&   \\
\cdots  &\E^{F}(V)_{m,s}  & \cdots &\cdots &\frac{\E_{\al C}(V)}{\Fgh_M (V)} &\cdots\\
\vdots &  \vdots & \vdots & \vdots & \vdots & \vdots
\end{pmatrix}.$$
and $\End_{\scr G^{F,\Phi}/\Fgh_{\al D''}}(V)=\frac{\scr G^\Phi(V)}{J}$.

\subsection{Locally finite tilting family and locally endomorphism algebras }

For $i\in\Phi$, we denote by ${\al C}^{\Phi, i, F}$ the category with the same objects of $\Phi$-orbit $\Oc{\al C}{F}{\Phi}$, and the morphism space
${\al C}^{\Phi, i, F}(X,Y)=\bigoplus_{j\in\Phi}{\al C}(X,F^{j-i}Y)$, for all $X, Y\in \al C$. If $i=0$, then ${\al C}^{\Phi, i, F}=\Oc{\al C}{F}{\Phi}$. Let
$$
 X\xra f M_1\ra M_2\ra\cdots\ra M_n\xra g Y\draf{\delta}
$$
 be an $n$-$\mathbb{E}$-exangle in $\mathcal {C}$ such that $M_j\in\add(M)$ for all $j=1,\cdots,n$, and  that $f$ is a left $\add_{\Oc{\al C}{F}{\Phi}}(M)$-approximation
of $X$, $g$ is a right $\add_{\Oc{\al C}{F}{\Phi}}(M)$-approximation of $Y$ in $\Phi$-orbit category $\Oc{\al C}{F}{\Phi}$. Let $\cpx{T}$ be the complex
$$
0\lra X\xra f M_1\ra M_2\ra\cdots\lra M_n\oplus M\lra 0$$
with $X$ in degree zero.  Then it follows from Lemma \ref{4.2} that $\cpx{T}$ is self-orthogonal in $\Kb{{\al C}^{\Phi,i,F}/\Fcogh_{\al D'}}$. Set $U=M\oplus X$. By Lemma \ref{llx},
$$
\E^{\Phi, i, F}(U, -)={\al C}^{\Phi, i, F}(U, -): \add_{{\al C}^{\Phi, i, F}} U\lra \proj{\E^{\Phi, i, F}(U,U)}=\proj{\End_{{\al C}^{\Phi, i, F}}(U,U)}
$$
is fully faithful, which induces a fully faithful triangle functor
$$
{\al C}^{\Phi, i, F}(U, -): \Kb{\add_{{\al C}^{\Phi, i, F}} U}\lra \Kb{ \proj{\End_{{\al C}^{\Phi, i, F}}(U,U)}}.
$$
Thus, we get a full triangle embedding
$$
({\al C}^{\Phi, i, F}/\Fcogh_{\al D'})(U, -): \Kb{\add_{{\al C}^{\Phi, i, F}/\Fcogh_{\al D'}} U}\lra \Kb{ \proj{\End_{{\al C}^{\Phi, i, F}/\Fcogh_{\al D'}}(U,U)}}.
$$
By Lemma \ref{llx}, we get a full triangle embedding
$$
(\scr G^{\Phi,F}/\Fcogh_{\al D''})(U, -): \Kb{\add_{\scr G^{\Phi,F}/\Fcogh_{\al D''}} U}\lra \Kb{ \proj{\End_{\scr G^{\Phi,F}/\Fcogh_{\al D''}}(U,U)}},
$$
where $\al D''=\add_{\scr G^{\Phi,F}}(M)$.

Let $\tilde{\cpx{T_i}}:=({\al C}^{\Phi, i, F}/\Fcogh_{\al D'})(U, \cpx{T})$. Then $\tilde{\cpx{T_i}}$ is self-orthogonal in $\Kb{ \proj{\End_{{\al C}^{\Phi, i, F}/\cogh_{\al D'}}(U,U)}}$ by Lemma \ref{4.2}.
Moreover,
$$
\End_{\scr G^{\Phi,F}/\Fcogh_{\al D''}}(U,U)=\oplus_{i\in\Phi}({\al C}^{\Phi, i, F}/\Fcogh_{\al D'})(U, U)
$$
and
$$
({\al C}^{\Phi, i, F}/\Fcogh_{\al D'})(U, U)\in \thick(\{\tilde{\cpx{T_i}},i\in I\}),
$$
therefore, $\thick(\{\tilde{\cpx{T_i}},i\in I\})$ generates $\Kb{ \proj{\End_{\scr G^{\Phi,F}/\Fcogh_{\al D''}}}(U,U)}$ as a triangulated category. So we have the following lemma.

\begin{lem} \label{lem3.3.3}
 $\{\tilde{\cpx{T_i}},i\in I\}$ is
a locally finite tilting family over $\End_{\scr G^{\Phi,F}/\Fcogh_{\al D''}}(U,U)$.
\end{lem}

\noindent{\bf Remark.}  (1) The complex $\oplus_{i\in\Phi}
\tilde{\cpx T_i}$ is a tilting complex in $\Cb{\proj{\End_{\scr G^{\Phi,F}/\Fcogh_{\al D''}}(U,U)}}$, if $\Phi$ is a finite admissible subset of $\mathbb{Z}$.

(2)
In fact, by the definition, we get
$$(\scr G^{F,\Phi}/\Fcogh_{\al D''})(U,X)=\begin{pmatrix}
\vdots &  \vdots & \vdots & \vdots & \vdots & \vdots\\
\cdots&\frac{\E_{\al T}(U,X)}{\Fcogh_M (U,X)}     & \E^{F}(U,X)_{s,s+1}                  &  \cdots& \E^{F}(U,X)_{s,m}&\cdots\\
\vdots &  \vdots & \vdots & \vdots & \vdots& \vdots\\
\cdots&\E^{F}(U,X)_{0,s} & \cdots & \frac{\E_{\al T}(U,X)}{\Fcogh_M (U,X)}   &  \cdots &\cdots\\
\vdots &  \vdots & \vdots & \vdots & \vdots&   \vdots \\
\cdots&\E^{F}(U,X)_{m,s}  & \cdots &\cdots  &\frac{\E_{\al T}(U,X)}{\Fcogh_M (U,X)} &\cdots\\
\vdots &  \vdots & \vdots & \vdots & \vdots & \vdots\\
\end{pmatrix},$$
and
 $$(\scr G^{F,\Phi}/\Fcogh_{\al D''})(U,M)=
\begin{pmatrix}
\vdots &  \vdots & \vdots & \vdots & \vdots & \vdots\\
\cdots&\E_{\al T}(U,M)    & \E^{F}(U,M)_{s,s+1}                  &  \cdots& \E^{F}(U,M)_{s,m}&\cdots\\
\vdots &  \vdots & \vdots & \vdots & \vdots& \vdots\\
\cdots&\E^{F}(U,M)_{0,s} & \cdots & \E_{\al T}(U,M)  &  \cdots &\cdots\\
\vdots &  \vdots & \vdots & \vdots & \vdots&   \vdots \\
\cdots&\E^{F}(U,M)_{m,s}  & \cdots &\cdots  &\E_{\al T}(U,M) &\cdots\\
\vdots &  \vdots & \vdots & \vdots & \vdots & \vdots\\

\end{pmatrix}.$$

To prove Theorem \ref{Theorem-ghost}, it suffices to show that the locally endomorphism ring of $\{\tilde{\cpx T_i},i\in\Phi\}$ is isomorphic to $\End_{(\scr G^{F,\Phi}/\Fcogh_{\al D''})}(V)$.

\begin{lem}\label{lem3.3.4}
There is a ring isomorphism between endomorphism ring $\End_{(\scr G^{F,\Phi}/\Fcogh_{\al D''})}(V)$ and locally endomorphism ring $(\Hom_{\scr K^b( \proj{\End_{\scr G^{F,\Phi}/\Fcogh_{\al D''}}(U,U)})}(\tilde{\cpx T_i},\tilde{\cpx{T_j}}))_{i,j\in\Phi}$.
\end{lem}

{\bf Proof.}
To show the lemma, for any $i,j\in\Phi$, we will construct an isomorphism from
locally endomorphism ring $(\Hom_{\scr K^b( \proj{\End_{\scr G^{F,\Phi}/\Fcogh_{\al D''}}(U,U)})}(\tilde{\cpx T_i},\tilde{\cpx{T_j}}))_{i,j\in\Phi}$
 to $\scr G^{F,\Phi}(V)/\Fgh_{\al D''}(V)$. Since

$$(\Hom_{\scr K^b( \proj{\End_{\End_{\scr G^{F,\Phi}/\Fcogh_{\al D''}})}}}(\tilde{\cpx T_i},\tilde{\cpx{T_j}}))_{i,j\in\Phi}$$
$$\cong
\begin{pmatrix}
\cdots &\cdots &\cdots\\
\cdots& \Hom_{\scr K^b(\proj{\End_{\scr G^{\Phi,F}/\Fcogh_{\al D''}}(U,U)})}(\tilde{\cpx T_i},\tilde{\cpx T_j})&\cdots\\
\cdots&\cdots&\cdots
\end{pmatrix}_{\Phi\times \Phi},
$$
we will construct the isomorphism by calculating
$\Hom_{\scr K^b(\proj{\End_{\scr G^{\Phi,F}/\Fcogh_{\al D''}}(U,U)})}(\tilde{\cpx T_i},\tilde{\cpx T_j}).$
By Lemma \ref{llx},
$$
\Hom_{\scr K^b(\proj{\End_{\scr G^{\Phi,F}/\Fcogh_{\al D''}}(U,U)})}(\tilde{\cpx T_i},\tilde{\cpx T_j})\simeq \Hom_{\Kb{{\al C}^{i-j, F,\Phi}/\Fcogh_{\al D'}}}(\cpx{T},\cpx{T})\simeq\Hom_{\Kb{{\al C}/\Fcogh_{\al D'}}}(\cpx{T},F^{i-j}(\cpx{T})),
$$
where $\al D'=\add_{\al C^{i-j,F,\Phi}}(M)$.
 \medskip
For each chain map $\cpx{u}$ in $\Hom_{\Cb{\al C}}(\cpx{T},F^{i-j}(\cpx{T}))$, by Lemma \ref{Lemma-n-exangle}(3),  there is a morphism $u\in\Hom_{\al C}(Y\oplus M,F^{i-j}(Y\oplus M)$ such that the diagram
$$
\quad\xymatrix{
T^0\ar[r]^{d_T^0}\ar[d]^{u^0}   &T^1\ar[r]\ar[r]^{d_T^1}\ar[d]^{u^1} & \cdots\ar[r]^{d_T^{n-1}} &T^{n}\ar[d]^{u^n}\ar[r]^{\tilde{g}}& Y\oplus M\ar@{-->}[r]^{\tilde{\delta}}\ar@{-->}[d]^{u{\quad\quad\quad\quad\quad\quad(\bigstar)}} & \\
F^{i-j}(T^0)\ar[r]^{F^{i-j}d_T^0} &F^{i-j}(T^1)\ar[r]^{F^{i-j}d_T^1} & \cdots\ar[r]^{F^{i-j}d_T^{n-1}} &F^{i-j}(T^n)\ar[r]^{F^{i-j}\tilde{g}}& F^{i-j}(Y\oplus M)\ar@{-->}[r]^(0.7){\Psi(\tilde{\delta})} & \\
}
$$
is commutative, where $\tilde{g}=\left[\begin{smallmatrix}g&0\\0&1\end{smallmatrix}\right]: T^n=M_n\oplus M\lra Y\oplus M$,  $\tilde{\delta}=\left[\begin{smallmatrix} \delta\\
0\end{smallmatrix}\right]\in \mathbb{E}(Y\oplus M,T^0)$ and $\Psi(\tilde{\delta})=\Upsilon_{(F^{i-j}(T^{n+1}),F^i(T^0))}(\Upsilon_{(F^{i-j-1}(T^{n+1}),F^{i-j-1}(T^0))}\cdots(\Upsilon_{(T^{n+1},T^0)}(\tilde{\delta})))
\in\mathbb{E}(F^{i-j}(T^{n+1}),F^{i-j}(T^0))$, and $T^{n+1}=Y\oplus M$.
If $u'$ is another morphism in $\End_{\al C}(Y\oplus M)$ making the above diagram commutative, then $\tilde{g}(u-u')=0=(u-u')^{\sharp}\Psi(\tilde{\delta})$.  Since $\tilde{g}$ is a right $\add_{\al D}(M)$-approximation by our assumption, the morphism $(u-u')$ belongs to $\gh_M(Y\oplus M)$ by Lemma \ref{2.2} (1). It follows from $(u-u')^{\sharp}\Psi(\tilde{\delta})=0$ that $u-u'$ factorizes through $F^{i-j}(T^n)$, which is in $\add_{\al D'}(M)$. Hence $u-u'$ is in $\Fgh_{\al D'}(Y\oplus M)$. Denote by $\bar{u}$ the morphism in ${\al C}^{\Phi, i, F}/\Fgh_{\al D'}$ corresponding to $u$. Thus, we get a map
$$
\theta_{i,i}: \End_{\Cb{\al C}}(\cpx{T}, \cpx{T})\lra \End_{\al C/\Fgh_M}(Y\oplus M)\simeq\frac{\E_{\al C}(V)}{\Fgh_M (V)}
$$
sending $\cpx{u}$ to $\bar{u}$, which is clearly a ring homomorphism.
If $i\neq j$, by the proof of Lemma \ref{fg}, $\Fgh_{\al D'}(M,Y\oplus M)=0$ and $\Fgh_{\al D'}(Y,Y\oplus M)=0$. Therefore, $u-u'=0$. Then we get a map
$$\theta_{i,j}: \Hom_{\Cb{ \al C^{i-j, F,\Phi}}}(\cpx {T},\cpx {T})\lra \E^{F}(M\oplus Y)_{i,j}  $$
sending $\cpx{u}$ to $u$.
Then we get a map
$$
\theta=\begin{pmatrix}
\cdots &\cdots &\cdots\\
\cdots& \theta_{i,j}&\cdots\\
\cdots&\cdots&\cdots
\end{pmatrix}_{\Phi\times \Phi}: (\Hom_{\Cb{{\al C}^{i-j, F,\Phi}}}(\cpx{T},\cpx{T}))_{i,j\in\Phi}\lra \End_{\scr G^{F,\Phi}/\Fgh_{\al D''}}(V)
$$
where
$$\End_{\scr G^{F,\Phi}/\Fgh_{\al D''}}(V)=\begin{pmatrix}
\vdots &  \vdots & \vdots & \vdots & \vdots & \vdots\\
\cdots  &\frac{\E_{\al T}(V)}{\Fgh_M (V)}      & \E^{F}(V)_{s,s+1}     &  \cdots& \E^{F}(V)_{s,m-1}&\cdots  \\
\vdots &  \vdots & \vdots & \vdots & \vdots & \vdots\\
\cdots  &\E^{F}(V)_{0,s} & \cdots & \frac{\E_{\al T}(V)}{\Fgh_M (V)}   &  \cdots &\cdots\\
\vdots &\vdots &  \vdots & \vdots & \vdots & \vdots&   \\
\cdots  &\E^{F}(V)_{m,s}  & \cdots &\cdots &\frac{\E_{\al T}(V)}{\Fgh_M (V)}  &\cdots\\
\vdots &  \vdots & \vdots & \vdots & \vdots & \vdots
\end{pmatrix}.$$

{\bf Claim 1.} The map $\theta=\begin{pmatrix}
\cdots &\cdots &\cdots\\
\cdots& \theta_{i,j}&\cdots\\
\cdots&\cdots&\cdots
\end{pmatrix}$ is surjective.

Indeed, for $\theta_{i,j}$, and for each $u$ in $\Hom_{\al C}(V,F^{i-j}(V))$, since $\tilde{g}$ is a right $\add_{\al D}(M)$-approximation, there is $u^{n}: T^{n}\lra T^{n}$ such that $\tilde{g}u=u^{n}F^{i-j}\tilde{g}$. Thus, by Lemma \ref{Lemma-n-exangle}, we get morphisms $u^l: T^l\lra T^l, l=0, \cdots, n-1$, making the above diagram $(\bigstar)$ commutative. This shows that $\theta$ is a surjective ring homomorphism.

{\bf Claim 2.} We shall prove that there is a surjective ring homomorphism
$$
\varphi=\begin{pmatrix}
\cdots &\cdots &\cdots\\
\cdots& \varphi_{i,j}&\cdots\\
\cdots&\cdots&\cdots
\end{pmatrix}_{\Phi\times \Phi}: (\End_{\Cb{\al C^{i-j, F,\Phi}}}(\cpx{T}))_{i,j\in\Phi}\lra (\End_{\Kb{{\al C^{i-j, F,\Phi}}/\Fcogh{\al D'}}}(\cpx{T}))_{i,j\in\Phi}
$$
It suffices to show that for any $i,j\in\Phi$, there is a surjective homomorphism
$$\varphi_{i,j}: \End_{\Cb{\al C^{i-j, F,\Phi}}}(\cpx{T})\lra \End_{\Kb{{\al C^{i-j, F,\Phi}}/\Fcogh{\al D'}}}(\cpx{T}), $$
which is  the composite of the ring homomorphism
$$\End_{\Cb{\al C^{i-j, F,\Phi}}}(\cpx{T})\ra\End_{\Cb{{\al C^{i-j, F,\Phi}}/\Fcogh{\al D'}}}(\cpx{T})$$
 induced by the canonical functor ${\al C^{i-j, F,\Phi}}\ra \al C^{i-j, F,\Phi}/\Fcogh_M$ and the canonical surjective ring homomorphism $$\End_{\Cb{\al C^{i-j, F,\Phi}}/\Fcogh{\al D'}}(\cpx{T})\lra\End_{\Kb{{\al C^{i-j, F,\Phi}}/\Fcogh{\al D'}}}(\cpx{T}).$$

{\bf Claim 3.} The maps $\theta_{i,j}$ and $\varphi_{i,j}$ have the same kernel.

In the case $i=j$,
$$\varphi_{i,i}: \End_{\Cb{\al C}}(\cpx{T})\lra \End_{\Kb{{\al C}/\Fcogh_M}}(\cpx{T}), $$
we shall show that $\theta_{i,i}$ and $\varphi_{i,i}$ have the same kernel.
A chain map $\cpx{u}$ is in $\Ker\varphi_{i,i}$ if and only if there exist $h^l: T^l\ra T^{l-1}, l=1,\cdots, n-1$ in $\mathcal{C}$ such that  $u^0-d_T^0h^1$, $u^l-h^lF^{i-j}d_T^{l-1}-d_T^lh^{l+1}, i=1,\cdots, n-1$, and $u^{n-2}-h^{n-2}d_T^{n-3}$ are all in $\Fcogh_{\al D'}$.  Using the fact that $T^i\in\add(M)$ for all $i>0$,  one can see, by Lemma \ref{2.2}, that this is equivalent to saying that $u^{n-2}-h^{n-2}d_T^{n-3}=0$,  $u^l=h^ld_T^{l-1}+d_T^lh^{l+1}$ for $l=1,\cdots, n-1$, and $u^0-d_T^0h^1\in\Fcogh_M(T^0)$.

\smallskip

 Let $\cpx{u}$ be in $\Ker\varphi_{i,i}$, and suppose that  $u\in\End_{\al C}(Y\oplus M)$ fits the commutative diagram $(\bigstar)$ above. Then $\theta(\cpx{u})=\bar{u}$.
We have $u^{n-2}=h^{n-2}d_T^{n-3}$, and consequently $\tilde{g}u=u^{n-2}\tilde{g}=h^{n-2}d_T^{n-3}\tilde{g}$, which is zero by Lemma \ref{Lemma-n-exangle} (1).  It follows from Lemma \ref{2.2} (1) that $u\in\gh_M(Y\oplus M)$. The fact $\cpx{u}\in\Ker\varphi_{i,i}$ also implies that $u^0-d_T^0h^1\in\Fcogh_M(T^0)$.  In particular, the morphism $u^0-d_T^0h^1$ factorizes through an object in $\add(M)$.  Assume that $u^0-d_T^0h^1=ab$ for some $a\in{\al T}(T^0, M')$ and $b\in{\al T}(M', T^0)$ with $M'\in\add(M)$. Since $d_T^0$ is a left $\add(M)$-approximation, we see that $a$ factorizes through $d_T^0$, and hence $u^0-d_T^0h^1$ factorizes through $d_T^0$. Consequently, the morphism $u^0$ also factorizes through $d_T^0$, say, $u^0=d_T^0\alpha$. By Lemma \ref{5.13}, $u$ factorizes through $T^{n-2}\in\add(M)$.  Altogether, we have shown that  $u$ belongs to $\Fgh_M(Y\oplus M)$. It follows that $\bar{u}=0$ and $\cpx{u}\in\Ker\theta$.  Hence $\Ker\varphi\subseteq\Ker\theta$.

  Conversely, suppose that $\cpx{u}\in\Ker\theta_{i,i}$ and $u\in\End_{\al T}(Y\oplus M)$ fits the commutative diagram $(\bigstar)$.  Then $\theta(\cpx{u})=\bar{u}=0$, that is, $u\in\Fgh_M(Y\oplus M)$.  Since $\tilde{g}$ is a right $\add(M)$-approximation, by Lemma \ref{2.2} (1), we have $\tilde{g}u=0$. Thus $u^{n-2}\tilde{g}=0$. By Lemma \ref{Lemma-n-exangle} (2), there is a morphism $h^{n-2}: T^{n-2}\ra T^{n-3}$ such that $u^{n-2}=h^{n-2}d_T^{n-3}$. Now $(u^{n-3}-d_T^{n-3}h^{n-2})d_T^{n-3}=u^{n-3}d_T^{n-3}-d_T^{n-3}u^{n-2}=0$. If $n\geq 4$, then, by Lemma \ref{Lemma-n-exangle} (2), there is a morphism $h^{n-3}: T^{n-3}\ra T^{n-4}$ such that $u^{n-3}-d_T^{n-3}h^{n-2}=h^{n-3}d_T^{n-4}$. Moreover, $(u^{n-4}-d_T^{n-4}h^{n-3})d_T^{n-4}=d_T^{n-4}u^{n-3}-d_T^{n-4}h^{n-3}d_T^{n-4}=d_T^{n-4}d_T^{n-3}h^{n-2}=0$.  Repeating this process, we get $h^l: T^l\ra T^{l-1}, l=1,\cdots, n-2$ such that $u^{n-2}=h^{n-2}d_T^{n-3}$, $u^l=h^ld_T^{l-1}+d_T^lh^{l+1}$ for $l=1,\cdots, n-3$, and $(u^0-d_T^0h^1)d_T^0=0$. Since $d_T^0$ is a left $\add(M)$-approximation, we deduce from Lemma \ref{2.2} (2) that $u^0-d_T^0h^1\in\cogh_M(T^0)$.  Since $u$ factorizes through an object in $\add(M)$ and $\tilde{g}$ is a right $\add(M)$-approximation, it is easy to see that $u$ factorizes through $\tilde{g}$, consequently, $u^0$ factorizes through $d_T^0$ by Lemma \ref{5.13}. Hence $u^0-d_T^0h^1$ factorizes through an object in $\add(M)$, and consequently belongs to $\Fcogh_M(T^0)$. Thus we have shown that $\cpx{u}\in\Ker\varphi$, and  $\Ker\theta\subseteq\Ker\varphi$.

In the case $i\neq j$,
$$\varphi_{i,j}: \End_{\Cb{\al C^{i-j, F,\Phi}}}(\cpx{T})\lra \End_{\Kb{{\al C^{i-j, F,\Phi}}/\Fcogh{\al D'}}}(\cpx{T}), $$
we shall show that $\theta_{i,j}$ and $\varphi_{i,j}$ have the same kernel.
A chain map $\cpx{u}$ is in $\Ker\varphi_{i,j}$ if and only if there exist $h^l: T^l\ra F^{i-j}T^{l-1}, l=1,\cdots, n-2$ in $\mathcal{C}$ such that  $u^0-d_T^0h^1$, $u^l-h^lF^{i-j}d_T^{l-1}-d_T^lh^{l+1}, l=1,\cdots, n-3$, and $u^{n-2}-h^{n-2}d_T^{n-3}$ are all in $\Fcogh_{\al D'}$.  Using the fact that $F^{i-j}T^l\in\add_{T^{i-j, F,\Phi}}(M)$ for all $l>0$,  one can see, by Lemma \ref{2.2}, that this is equivalent to saying that $u^{n-2}-h^{n-2}d_T^{n-3}=0$,  $u^l=h^id_T^{l-1}+d_T^ih^{i+1}$ for $i=1,\cdots, n-3$, and $u^0-d_T^0h^1\in\Fcogh_{\al D'}(T^0)$.
Since $d^0=f$ is a left $\add_{\Oc{\al C}{F}{\Phi}}(M)$-approximation, then $\Fcogh_{\al D'}(T^0)$ consisting $g\in\al C(T^0, F^{i-j}T^0)$ such that $gd^0=0$ and $g$ factorizes through $d^0$, then $g=0$ by $\al C(M, F^iX)=0$. Therefore, $\Fcogh_{\al D'}(T^0)=0$. Then $u^0-d_T^0h^1=0$.

\smallskip

 Let $\cpx{u}$ be in $\Ker\varphi_{i,j}$, and suppose that  $u\in\End_{\al C}(Y\oplus M)$ fits the commutative diagram $(\bigstar)$ above. Then $\theta(\cpx{u})=u$.
We have $u^{n-2}=h^{n-2}d_T^{n-3}$, and consequently $\tilde{g}u=u^{n-2}\tilde{g}=h^{n-2}d_T^{n-3}\tilde{g}$, which is zero by Lemma \ref{Lemma-n-exangle} (1).  It follows from Lemma \ref{2.2} (1) that $u\in\gh_{\al D'}(Y\oplus M)$. The fact $\cpx{u}\in\Ker\varphi_{i,j}$ also implies that $u^0-d_T^0h^1=0$.   Thus it follows from Lemma \ref{5.13} that $u$ factorizes through $F^{i-j}T^{n-2}$.  Altogether, we have shown that  $u$ belongs to $\Fgh_{\al D'}(Y\oplus M)$ and $\Fgh_{\al D'}(Y\oplus M)=0$ by Lemma \ref{fg}. It follows that $u=0$ and $\cpx{u}\in\Ker\theta_{i,j}$.  Hence $\Ker\varphi_{i,j}\subseteq\Ker\theta_{i,j}$.

  Conversely, suppose that $\cpx{u}\in\Ker\theta_{i,j}$ and $u\in\End_{\al C}(Y\oplus M)$ fits the commutative diagram $(\bigstar)$.  Then $\theta(\cpx{u})=u=0$.  Since $\tilde{g}$ is a right $\add_{\Oc{\al C}{F}{\Phi}}(M)$-approximation, by Lemma \ref{2.2} (1), we have $\tilde{g}u=0$. Thus $u^{n-2}F^{i-j}\tilde{g}=0$. By Lemma \ref{Lemma-n-exangle} (2), there is a morphism $h^{n-2}: T^{n-2}\ra T^{n-3}$ such that $u^{n-2}=h^{n-2}F^{i-j}d_T^{n-3}$. Now $(u^{n-3}-d_T^{n-3}h^{n-2})d_T^{n-3}=u^{n-3}d_T^{n-3}-d_T^{n-3}u^{n-2}=0$. If $n\geq 4$, then, by Lemma \ref{Lemma-n-exangle} (2), there is a morphism $h^{n-3}: T^{n-3}\ra T^{n-4}$ such that $u^{n-3}-d_T^{n-3}h^{n-2}=h^{n-3}F^{i-j}d_T^{n-4}$. Moreover, $(u^{n-4}-d_T^{n-4}h^{n-3})d_T^{n-4}=d_T^{n-4}u^{n-3}-d_T^{n-4}h^{n-3}d_T^{n-4}=d_T^{n-4}d_T^{n-3}h^{n-2}=0$.  Repeating this process, we get $h^l: T^l\ra T^{l-1}, l=1,\cdots, n-2$ such that $u^{n-2}=h^{n-2}F^{i-j}d_T^{n-3}$, $u^l=h^lF^{i-j}d_T^{l-1}+d_T^lh^{l+1}$ for $l=1,\cdots, n-3$, and $(u^0-d_T^0h^1)F^{i-j}d_T^0=0$. Since $d_T^0$ is a left $\add_{\Oc{\al C}{F}{\Phi}}(M)$-approximation, we deduce from Lemma \ref{2.2} (2) that $u^0-d_T^0h^1\in\cogh_{\al D'}(T^0)$.  Since $u=0$, by Lemma \ref{5.13}, $(u^0)_{\sharp}\tilde{\delta}=(u)^{\sharp}\Theta(\tilde{\delta})=0$, $u^0$ factorizes through $d_T^0$. Hence  $u^0-d_T^0h^1$ factorizes through an object in $\al D'$, and consequently belongs to $\Fcogh_{\al D'}(T^0)=0$. Consequently, $u^0-d_T^0h^1=0$. Thus we have shown that $\cpx{u}\in\Ker\varphi_{i,j}$, and  $\Ker\theta_{i,j}\subseteq\Ker\varphi_{i,j}$.

From the above argument, it follows that $\theta$ and $\varphi$ have the same kernel, so we have the following diagram
$$
\xymatrix{
0&\Ker(\theta)&(\Hom_{\Cb{{\al C}^{i-j, F,\Phi}}}(\cpx{T},\cpx{T}))_{i,j\in\Phi} & \End_{\scr G^{F,\Phi}/\Fgh_{\al D''}}(V)&0\\
0&\Ker(\varphi)&(\Hom_{\Cb{{\al C}^{i-j, F,\Phi}}}(\cpx{T},\cpx{T}))_{i,j\in\Phi} & (\Hom_{\Kb{{\al C}^{i-j, F,\Phi}/\Fcogh_{\al D'}}}(\cpx{T},\cpx{T}))_{i,j\in\Phi}&0.
\ar"1,1";"1,2"^{}
\ar"2,1";"2,2"^{}
\ar"1,2";"1,3"^{}
\ar"2,2";"2,3"^{}
\ar"1,4";"1,5"^{}
\ar"2,4";"2,5"^{}
\ar"1,3";"1,4"^{\theta}
\ar"2,3";"2,4"^{\varphi}
\ar@{=}"1,2";"2,2"
\ar@{=}"1,3";"2,3"
\ar@{-->}"1,4";"2,4"^{\simeq}
}
$$
Hence the locally endomorphism ring
$$(\Hom_{\scr K^b( \proj{\End_{\scr G^{\Phi,F}/\Fcogh_{\al D''}}(U,U)})}(\tilde{\cpx T_i},\tilde{\cpx{T_j}}))_{i,j\in\Phi}$$  and endomorphism ring
$$\End_{\scr G^{F,\Phi}/\Fgh_{\al D''}}(V)$$ are isomorphic, and the result then follows.

{\bf Proof of Theorem \ref{Theorem-ghost}} By Lemmas \ref{lem3.3.3} and \ref{lem3.3.4}, the theorem is straightforward from Theorem \ref{ll}. $\square$\\

\section{Derived equivalences of locally finite $\Phi$-Beilinson-Green algebras}\label{derivedequivalence}

\subsection{Derived equivalences of locally finite $\Phi$-Beilinson-Green algebras from a given derived equivalence}

In this subsection, we will prove that derived equivalent locally $\Phi$-Beilinson-Green algebras can be constructed from a given derived equivalence.

First, we recall the definition of stable functor of non-negative triangle functor between derived categories of abelian categories.  Suppose that $\mathcal{A}$ and $\mathcal{B}$ are abelian categories with enough projective objects.  The full subcategories of projective objects are denoted by $\proja$ and $\projb$, respectively.  The corresponding stable categories are denoted by $\sta$ and $\stb$, respectively.
We also write $\Db{\mathcal{A}}$, $\Df{\mathcal{A}}$ and $\Dz{\mathcal{A}}$ for the full subcategories of $\D{\mathcal{A}}$ consisting of complexes isomorphic to bounded complexes, complexes bounded above, and complexes bounded below, respectively. Moreover, for integers $m\leq n$ and for a collection of objects $\mathcal{X}$, we write $\D[{[m, n]}]{\mathcal{X}}$ for the full subcategory of $\D{\mathcal{A}}$ consisting of complexes $\cpx{X}$ isomorphic in $\D{\mathcal{A}}$ to complexes  with terms in $\mathcal{X}$ of the form
$$0\lra X^m\lra\cdots\lra X^n\lra 0.$$

\begin{defn}\cite[Definition 4.1]{Hu2017}
A triangle functor $G: \Db{\mathcal{A}}\lra \Db{\mathcal{B}}$  is called  uniformly bounded if there are integers $r<s$ such that $G(X)\in \D[{[r, s]}]{\mathcal{B}}$ for all $X\in\mathcal{A}$, and is called non-negative if $G$ satisfies the following conditions:

\smallskip
(1) $G(X)$ is isomorphic to a complex with zero homology in all negative degrees for all $X\in\mathcal{A}$.

(2) $G(P)$ is isomorphic to a complex in $\Kb{\projb}$ with zero terms in all negative degrees for all $P\in\proja$.
\label{def-uni-bounded-non-negative}
\end{defn}

Now suppose that
$$G: \Db{\mathcal{A}}\lra\Db{\mathcal{B}}$$
is a non-negative triangle functor.  For any object $X\in \mathcal{A}$, by definition, $G(X)$ has no homology in negative degrees. Take a projective resolution of $G(X)$ and then do good truncation at degree zero. Then there is a triangle
 $$\cpx{U}_X\lraf{i_X}G(X)\lraf{\pi_X} M_X\lraf{\mu_X}\cpx{U}_X[1]\quad\quad (\clubsuit)$$
in $\Db{\mathcal{B}}$ with $M_X\in\mathcal{B}$ and $\cpx{U}_X\in\D[{[1,n_X]}]{\projb}$ for some $n_X>0$.
We can define a functor $\bar{G}: \sta\lra \stb$
  as follows. For each $X\in\mathcal{A}$, we fix a triangle  $$\xi_{X}: \quad \cpx{U}_X\lraf{i_X}G(X)\lraf{\pi_X} M_X\lraf{\mu_X}\cpx{U}_X[1]$$
in $\Db{\mathcal{B}}$ with $M_X\in\mathcal{B}$, and $\cpx{U}_X$ a complex in $\D[{[1, n_X]}]{\projb}$ for some $n_X>0$.  The existence is guaranteed by $(\clubsuit)$. For each morphism $f: X\ra Y$ in $\mathcal{A}$, we can form a commutative diagram in $\Db{\mathcal{B}}$:
 $$\xymatrix@M=1.5mm{
 \cpx{U}_X\ar[r]^{i_X}\ar[d]^{a_f} & G(X)\ar[r]^{\pi_X}\ar[d]^{F(f)} & M_X\ar[r]^{\mu_X}\ar[d]_{b_f} &\cpx{U}_X[1]\ar[d]^{a_f[1]}\\
 \cpx{U_Y}\ar[r]^{i_Y} & G(Y)\ar[r]^{\pi_Y} & M_Y\ar[r]^{\mu_Y} &\cpx{U_Y}[1]\\
 }$$
If $b'_f$ is another morphism such that $\pi_Xb'_f=G(f)\pi_Y$, then $\pi_X(b_f-b'_f)=0$, and $b_f-b'_f$ factorizes through $\cpx{U}_X[1]$. By \cite[Corollary 3.4]{Hu2017}, the map $b_f-b'_f$ factorizes through $U_X^1$ which is projective. Hence the morphism $\underline{b_f}\in\stb(M_X, M_Y)$ is uniquely determined by $f$. Moreover, suppose that $f$ factorizes through a projective object $P$ in $\mathcal{A}$, say $f=gh$ for $g: X\ra P$ and $h: P\ra Y$.  Then $\pi_X(b_f-b_gb_h)=G(f)\pi_Y-G(g)\pi_Pb_h=G(f)\pi_Y-G(g)G(h)\pi_Y=0$.  Hence $b_f-b_gb_h$ factorizes through $\cpx{U}_X[1]$, and factorizes through $U_X^1$ by \cite[Corollary 3.4]{Hu2017}.  Thus $b_f$ factorizes through $P\oplus U_X^1$ which is projective.  Hence $\underline{b_f}=0$.  Then we get a well-defined map
$$\phi: \Hom_{\sta}(X, Y)\lra \Hom_{\stb}(M_X, M_Y), \quad \underline{f}\mapsto \underline{b_f}. $$
It is easy to say that $\phi$ is functorial in $X$ and $Y$.  Defining $\bar{G}(X):=M_X$ for each $X\in\mathcal{A}$ and $\bar{F}(\underline{f}):=\phi(\underline{f})$ for each morphism $f$ in $\mathcal{A}$, we get a functor
$$\bar{G}: \sta\lra\stb $$
which is called {\em the stable functor} of $G$.

\medskip
For derived equivalences between module categories of rings, we have the following lemma.
\begin{lem} \cite[Lemma 4.2]{Hu2017}
Let $G: \Db{A\Modc}\lra\Db{B\Modc}$ be a derived equivalence between two rings $A$ and $B$. Then

$(1)$  $G$ is uniformly bounded.

$(2)$ $G$ is non-negative  if and only if the tilting complex associated to $F$ is isomorphic in $\Kb{\proj{B}}$ to a complex with zero terms in all positive degrees. In particular, $F[i]$ is non-negative for sufficiently small $i$.
\end{lem}

So, we can get a stable functor associated for each derived equivalence between module categories of rings.

Let $S$ be a commutative artin ring, and let $A$, $B$ be Artin $S$-algebras. Suppose that  $G: \Db{A\Modc}\lra\Db{B\Modc}$ is a derived equivalent,  and let $\cpx{P}$ be the tilting
complex associated to $G$. Without loss of generality, as in \cite{Hu2011,P1}, up to shift, we assume
that $\cpx{P}$ is a radical complex of the following form
$$
0\ra P^{-n}\ra P^{-n+1} \ra \cdots\ra P^{-1}\ra P^{0}\ra 0,
$$
and that $P^{0}\neq 0 \neq P^{-n}$. Then there is a tilting complex
$\bar{P}^{\bullet}$ for $B$ associated to the quasi-inverse $G'$ of $G$
of the form
$$
0\ra \bar{P}^{0}\ra \bar{P}^{1} \ra \cdots\ra \bar{P}^{n-1}\ra
\bar{P}^{n}\ra 0,
$$ with the differentials being radical maps.
Suppose that $\cpx{P}$ and $\bar{P}^{\bullet}$ are the tilting
complexes associated to $G$ and the quasi-inverse of $G$, respectively.
Set $P=\oplus^{n}_{i=1}P^{-i}$ and $\bar{P}=\oplus^{n}_{i=1}\bar{P}^{i}$.
Then $G$ is called an almost $\nu$-stable derived equivalence \cite{Hu2011}
provided that $\add(P)=\add(\nu_{A}P)$
and $\add(\bar{P})=\add(\nu_{B}\bar{P})$, where $\nu$ is the Nakayama functor.

We recall some basic facts on almost $\nu$-stable derived
equivalences, which will be used in our proofs.

\begin{lem}\cite[Lemma 3.3]{Hu2013}\label{hu}
Let $G: \Db{A\Modc}\ra \Db{B\Modc}$ be an almost $\nu$-stable
derived equivalence. Suppose that  $\cpx{P}$ and $\bar{P^{\bullet}}$ are tilting complexes associated to $G$ and to the quasi-inverse $G'$, respectively.
Then we have the following:

$(1)$ For any $A$-module $X$, the complex $G(X)$ is isomorphic
in $\Db{B\Modc}$ to a radical complex $\bar{P}^{\bullet}_{X}$ of the form
$$
0\ra \bar{P}_{X}^{0}\ra \bar{P}_{X}^{1} \ra \cdots\ra
\bar{P}_{X}^{n-1}\ra \bar{P}_{X}^{n}\ra 0
$$
with $\bar{P}_{X}^{i}\in\add(P)$ for all $i>0$.

$(2)$ For any $B$-module $Y$, the complex $G'(Y)$ is isomorphic
in $\Db{A}$ to a radical complex $P^{\bullet}_{Y}$ of the form

$$
0\ra P_{Y}^{-n}\ra P_{Y}^{-n+1} \ra \cdots\ra P_{Y}^{-1}\ra
P_{Y}^{0}\ra 0
$$
with $P_{Y}^{i}\in\add(\bar{P})$ for all $i<0$.

$(3)$ There is a stable equivalence $\bar{G}:\stmodc{A}\lra\stmodc{B}$ with $\bar{G}(X)=\bar{P}_{X}^{0}$ for each $A$-module $X$.

$(4)$ There is a stable equivalence $\overline{G'}:\stmodc{B}\lra\stmodc{A}$ with $\overline{G'}(Y)=P_{Y}^{0}$ for each $B$-module $Y$.
 Moreover, the functor $\overline{G'}$ is a quasi-inverse of $\bar{G}$ defined in $(3)$.

\end{lem}

\begin{thm}\label{PPP}
Let $A$ and $B$ be Artin algebras and let $\Phi$ be an admissible subset of $\mathbb{Z}$. Suppose that  $G: \Db{A\Modc}\lra\Db{B\Modc}$ is an almost $\nu$-derived equivalent.
Denote by $\bar{G}$ the stable functor induced by $G$.  If $X$ is an $A$-module,
then locally $\Phi$-Beilinson-Green algebras $\scr G^{\Phi}(A\oplus X)$
and $\scr G^{\Phi}(B\oplus \bar{G}(X))$ are derived equivalent.
\end{thm}
\begin{proof}
By Lemma \ref{hu}(3), we have $\bar{G}(X)=P^0_X$.
Let $\bar{\cpx{T}}=\bar{P}^{\bullet}\oplus\bar{P}^{\bullet}_{X}$. Set $\cpx{T_i}:=\E^{\Phi, i}(V, \cpx{\bar{T}})$.
The complex $\cpx{T_i}$ has the following form
$$
0\ra \E^{\Phi,i}_{B}(V,\bar{T}^{0})\ra \E^{\Phi,i}_{B}(V,\bar{T}^{1})\ra\E^{\Phi,i}_{B}(V,\bar{T}^{2})\ra \cdots\ra \E^{\Phi,i}_{B}(V,\bar{T}^n)\ra 0,
$$
where $\bar{T}^i=\bar{P}^i\oplus{\bar{P}_{X}}^i$ for all $0\leq i\leq n$. In the following, we shall prove that $\{\cpx{T_i},i\in\Phi\}$ is
a locally finite tilting family over $\scr G^{\Phi}(B\oplus \bar{G}(X))$. Set $V:=B\oplus \bar{G}(X)$.

$(1)$
$\Hom_{\Kb{\scr G^{\Phi}(V)}}(\E_{B}^{\Phi,i}(V,\cpx{\bar{T}}),\E_{B}^{\Phi,j}(V,\cpx{\bar{T}})[m])=0$
for any $i,j\in \Phi$ and $m\neq 0$.

Assume that $f=(f^s)_{s\geq0}$ is in
$\Hom_{\Kb{\scr G^{\Phi}(V)}}(\E_{B}^{\Phi,i}(V,\cpx{\bar{T}}),\E_{B}^{\Phi,j}(V,\cpx{\bar{T}})[m])$.
Then we have the following commutative diagram

$$
\xymatrix{ 0\ar[r]\ar[d]
&\E^{\Phi,i}_{B}(V,\bar{T}^{0})\ar[r]\ar^{f^0}[d]
  & \E^{\Phi,i}_{B}(V,\bar{T}^{1})\ar[r]\ar^{f^1}[d]&
  \E^{\Phi,i}_{B}(V,\bar{T}^{2}) \ar^{f^2}[d]\ar[r]&\cdots\\
 \E^{\Phi,j}_{B}(V,\bar{T}^{m-1})\ar[r] &
\E^{\Phi,j}_{B}(V,\bar{T}^{m})\ar[r]
  & \E^{\Phi,j}_{B}(V,\bar{T}^{m+1})\ar[r]&
  \E^{\Phi,j}_{B}(V,\bar{T}^{m+2})\ar[r]&\cdots. }
$$
By the construction as above, we see that
$\E_{B}^{\Phi}(V,\bar{T}^{t})=0$ for $t<0$ since
$\cpx{\bar{T}}=\bar{P}^{\bullet}\oplus\bar{P}^{\bullet}_{X}$ has the
form as follows:
$$
0\ra P^0\oplus \bar{P}_{X}^{0}\ra \bar{P}^1\oplus P_{X}^{1} \ra \cdots\ra
\bar{P}^{n-1}\oplus \bar{P}_{X}^{n-1}\ra \bar{P}^n\oplus \bar{P}_{X}^{n}\ra 0,
$$
where $\bar{T}^0=\bar{P}^0\oplus \bar{P}_{X}^{0}$ is non-projective and
$\bar{T}^i=P^i\oplus P_{X}^{i}$ are projective for $1\leq i\leq n$. It
follows from Lemma \ref{llx} that $f^k=\mu(g^k)$, where
$g^k:\bar{T}^k\ra\bar{T}^{k+m}[i-j]$, for all $1\leq k\leq n$. From the
commutativity of the chain map with differentials, we have
$$
\E_{B}^{\Phi,i}(V,d)f^{k+1}=f^k\E_{B}^{\Phi,j}(V,d),
$$
that is,
$$
\mu(d^k)\mu(g^{k+1})=\mu(g^k)\mu(d^{k+m}).
$$
Therefore, $\mu(d^kg^{k+1}-g^kd^{k+m})=0$. Hence, we get $d^kg^{k+1}-g^kd^{k+m}=0$. Consequently,
$\cpx{g}=(g^k)\in\Hom_{\Kb{\add_{B}V}}(\cpx{\bar{T}},\cpx{\bar{T}}[i-j+m])$
and $\cpx{f}=(f^k)=(\mu(g^k))$. By \cite[Lemma 3.6(1)]{Hu2013},
$\Hom_{\Kb{\add_{B}V}}(\cpx{\bar{T}},\cpx{\bar{T}}[m])=0$ for all $m\neq 0$. Therefore,
$\cpx{g}$ is null-homotopic, and consequently, $\cpx{f}$ is
null-homotopic. This shows the exceptionality of the complex
$\E_{B}^{\Phi,i}(V,\cpx{\bar{T}})$.

$(2)$
 $\thick(\{\E_{B}^{\Phi,i}(V,\cpx{\bar{T}})\}_{i\in\Phi})=\Kb{\proj{\scr G^{\Phi}(V)}}$.

Indeed, by \cite[Lemma 3.6(2)]{Hu2013}, we see that $\add \cpx{\bar{T}}$ generates
$\Kb{\add_{B}V}$ as a triangulated category. Since $\scr G^{\Phi}_{B}(V)=\oplus_{i\in\Phi}\E_{B}^{\Phi,i}(V)$, it is easy to see that
$\scr G^{\Phi}_{B}(V)$ is in the smallest subcategory of $\Kb{\proj{\scr G^{\Phi}_{B}(V)}}$ which is generated
by $\thick(\{\E_{B}^{\Phi,i}(V,\cpx{\bar{T}})\}_{i\in\Phi})$, where $V$ is in
$\thick{\add_{B}\cpx{\bar{T}}}$. Therefore, $\thick(\{\E_{B}^{\Phi,i}(V,\cpx{\bar{T}})\}_{i\in\Phi})=\Kb{\proj{\scr G^{\Phi}(V)}}$.

(3) The locally endomorphism ring of $\cpx{T}_i$ is isomorphic to $\scr G_A^{\Phi}(U)$.
By Lemma \ref{llx}, we have the following
$$
\Hom_{\Kb{\scr G^{\Phi}(V)}}(\E_{B}^{\Phi,i}(V,\cpx{\bar{T}}),\E_{B}^{\Phi,j}(V,\cpx{\bar{T}}))
\cong
\left\{
\begin{array}{cc}
\End(U), & if\ \  i\neq j;\\
\E(U)_{i,j},  & if\ i=j.
\end{array}
\right.
$$
It follows from Theorem \ref{ll} that, the locally $\Phi$-Beilinsion--Green algebras $\scr G^{\Phi}(A\oplus X)$
and $\scr G^{\Phi}(B\oplus \bar{G}(X))$ are derived equivalent.

\end{proof}
\begin{rem}
(1) Hu and Xi \cite[Theorem 3.4(2)]{Hu2013} proved that, if $A$ and $B$ are
almost $\nu$-stable derived equivalent, then the $\Phi$-Auslander-Yoneda algebras of $A$ and $B$ are stably
equivalent of Morita type for a finite admissible set $\Phi$ of $\mathbb{N}$. After that, Peng \cite{Pe} proved that
if $A$ and $B$ are
almost $\nu$-stable derived equivalent, then the $\Phi$-Beilinsion-Greeen algebras of $A$ and $B$ are stably
equivalent of Morita type for a finite admissible set $\Phi$ of $\mathbb{N}$. In \cite{P3},we have proved that, if $A$ and $B$ are
stably equivalent of Morita type, then $\Phi$-Beilinson-Green algebras of $A$ and $B$ are stably equivalences of Morita type for a finite admissible set $\Phi$ of $\mathbb{N}$.

(2) Here an admissible set $\Phi$ of $\mathbb{N}$ can be a infinite set of $\mathbb{N}$, so we generalize Peng's result to infinite case.
\end{rem}

In another direction, we consider the derived equivalences instead of almost $\nu$-stable derived equivalences, and we give some restrictions on the modules.
Let $_{A}\mathcal {X}=\{X\in A\modc\, |\,\Ext^{i\geq 1}(X, A)=0\}$ be a full subcategory of $A\modc$.
Ringel and Zhang \cite{RZ} call these modules semi-Gorenstein projective modules. In the following, we will
study the locally $\Phi$-Beilinson-Green algebras of semi-Gorenstein projective modules.

\begin{lem}\cite[Lemma 3.3]{P1}\label{PP}  Let $G: \Db{A}\lra \Db{B}$ be a derived equivalence between
Artin algebras $A$ and $B$, and let $G'$ be the quasi-inverse of $G$.
Suppose that $\cpx{P}$ and $\bar{P}^{\bullet}$ are the tilting
complexes associated to $G$ and $G'$, respectively. Then

$(i)$ For $X\in _{A}\mathcal {X}$, the complex $G(X)$ is isomorphic
in $\Db{B}$ to a radical complex $\bar{P}^{\bullet}_{X}$ of the form
$$
0\ra \bar{P}_{X}^{0}\ra \bar{P}_{X}^{1} \ra \cdots\ra
\bar{P}_{X}^{n-1}\ra \bar{P}_{X}^{n}\ra 0
$$
with $\bar{P}_{X}^{0}\in _{B}\mathcal {X}$ and $\bar{P}_{X}^{i}$
projective $B$-modules for $1\leq i\leq n$.

$(ii)$ For $Y\in_{B}\mathcal {X}$, the complex $G'(Y)$ is isomorphic
in $\Db{A}$ to a radical complex $P^{\bullet}_{Y}$ of the form

$$
0\ra P_{Y}^{-n}\ra P_{Y}^{-n+1} \ra \cdots\ra P_{Y}^{-1}\ra
P_{Y}^{0}\ra 0
$$
with $P_{Y}^{-n}\in_{A}\mathcal {X}$ and $P_{Y}^{i}$ projective
$A$-modules for $-n+1\leq i\leq 0$.
\end{lem}

Suppose that  $G: \Db{A\Modc}\lra\Db{B\Modc}$ is a derived equivalent between left coherent rings $A$ and $B$. Then we can have a stable functor $\bar{G}$ between finitely presented modules $A\modc$ and $B\modc$. Recall that $A$ is a left coherent ring if and only if $A\modc$ the category consisting of finitely presented $A$-modules is an abelian category.

\begin{thm}\label{ps1}
Let $\Phi$ be an admissible subset of $\mathbb{Z}$. Suppose that  $G: \Db{A\Modc}\lra\Db{B\Modc}$ is a derived equivalent between left coherent rings $A$ and $B$.
Denote by $\bar{G}$ the stable functor induced by $G$.
 If $X$ is a finitely presented semi-Gorenstein projective $A$-module, then locally $\Phi$-Beilinson-Green algebras $\scr G^{\Phi}(A\oplus X)$
and $\scr G^{\Phi}(B\oplus \bar{G}(X))$ are derived equivalent.
\end{thm}

\begin{proof} The proof is similar to that of Theorem \ref{PPP}. For the convenience we give the details here.
By Lemma \ref{PP}, if $X$ is an $A$-module with $\Ext^i_{A}(X, A)=0$ for $i\geq 1$, the complex $G(X)$ is isomorphic
in $\Db{B}$ to a radical complex $\bar{P}^{\bullet}_{X}$ of the form
$$
0\ra \bar{P}_{X}^{0}\ra \bar{P}_{X}^{1} \ra \cdots\ra
\bar{P}_{X}^{n-1}\ra \bar{P}_{X}^{n}\ra 0
$$
with $\bar{P}_{X}^{i}$
projective $B$-modules for $1\leq i\leq n$ and $\bar{G}(X)=P^0_X$ is a $B$-module with $\Ext^i_{B}(\bar{P}_{X}^0, B)=0$ for $i\geq 1$.

Let $\bar{\cpx{T}}=\bar{P}^{\bullet}\oplus\bar{P}^{\bullet}_{X}$. Set $\cpx{T_i}:=\E^{\Phi, i}(V, \cpx{\bar{T}})$.
In the following, we will show that $\{\cpx{T_i},i\in\Phi\}$ is
a locally tilting family over $\scr G^{\Phi}(B\oplus \bar{G}(X))$. Set $V:=B\oplus \bar{G}(X)$.

$(1)$
$\Hom_{\Kb{\scr G^{\Phi}(V)}}(\E_{B}^{\Phi,i}(V,\cpx{\bar{T}}),\E_{B}^{\Phi,j}(V,\cpx{\bar{T}})[m])=0$
for any $i,j\in \Phi$ and $m\neq 0$.

Assume that $f=(f^s)_{s\geq0}$ is in
$\Hom_{\Kb{\scr G^{\Phi}(V)}}(\E_{B}^{\Phi,i}(V,\cpx{\bar{T}}),\E_{B}^{\Phi,j}(V,\cpx{\bar{T}})[m])$.
Then we have the following commutative diagram

$$
\xymatrix{ 0\ar[r]\ar[d]
&\E^{\Phi,i}_{B}(V,\bar{T}^{0})\ar^{\E_{B}^{\Phi,i}(V,d^0)}[r]\ar^{f^0}[d]
  & \E^{\Phi,i}_{B}(V,\bar{T}^{1})\ar^{\E_{B}^{\Phi,i}(V,d^1)}[r]\ar^{f^1}[d]&
  \E^{\Phi,i}_{B}(V,\bar{T}^{2}) \ar^{f^2}[d]\ar[r]&\cdots\\
 \E^{\Phi,j}_{B}(V,\bar{T}^{m-1})\ar^{\E_{B}^{\Phi,i}(V,d^{m-1})}[r] &
\E^{\Phi,j}_{B}(V,\bar{T}^{m})\ar^{\E_{B}^{\Phi,i}(V,d^m)}[r]
  & \E^{\Phi,j}_{B}(V,\bar{T}^{m+1})\ar^{\E_{B}^{\Phi,i}(V,d^{m+1})}[r]&
  \E^{\Phi,j}_{B}(V,\bar{T}^{m+2})\ar[r]&\cdots. }
$$
By the construction as above, we see that
$\E_{B}^{\Phi}(V,\bar{T}^{t})=0$ for $t<0$ since
$\cpx{\bar{T}}=\bar{P}^{\bullet}\oplus\bar{P}^{\bullet}_{X}$ has the
form as follows:
$$
0\ra P^0\oplus \bar{P}_{X}^{0}\ra \bar{P}^1\oplus P_{X}^{1} \ra \cdots\ra
\bar{P}^{n-1}\oplus \bar{P}_{X}^{n-1}\ra \bar{P}^n\oplus \bar{P}_{X}^{n}\ra 0,
$$
where $\bar{T}^0=\bar{P}^0\oplus \bar{P}_{X}^{0}$ is non-projective and
$\bar{T}^i=P^i\oplus P_{X}^{i}$ are projective for $1\leq i\leq n$. It
follows from Lemma \ref{llx} that $f^k=\mu(g^k)$, where
$g^k:\bar{T}^k\ra\bar{T}^{k+m}[i-j]$, for all $1\leq k\leq n$. From the
commutativity of the chain map with differentials, we have
$$
\E_{B}^{\Phi,i}(V,\cpx{d})f^{k+1}=f^k\E_{B}^{\Phi,j}(V,\cpx{d}),
$$
that is,
$$
\mu(d^k)\mu(g^{k+1})=\mu(g^k)\mu(d^{k+m}).
$$
Therefore, $\mu(d^kg^{k+1}-g^kd^{k+m})=0$. Hence, we get $d^kg^{k+1}-g^kd^{k+m}=0$. Consequently,
$\cpx{g}=(g^k)\in\Hom_{\Kb{\add_{B}V}}(\cpx{\bar{T}},\cpx{\bar{T}}[i-j+m])$
and $\cpx{f}=(f^k)=(\mu(g^k))$. By \cite[Lemma 3.7(1)]{P1},
$\Hom_{\Kb{\add_{B}V}}(\cpx{\bar{T}},\cpx{\bar{T}}[m])=0$ for all $m\neq 0$.
Therefore, $\cpx{g}$ is null-homotopic, and consequently, $\cpx{f}$ is
null-homotopic. This shows the exceptionality of the complex
$\E_{B}^{\Phi,i}(V,\cpx{\bar{T}})$.

$(2)$
 $
\thick(\{\E_{B}^{\Phi,i}(V,\cpx{\bar{T}})\}_{i\in\Phi})=\Kb{\proj{\scr G^{\Phi}(V)}}$.

Indeed, by \cite[Lemma 3.7(2)]{P1}, we see that $\add\cpx{\bar{T}}$ generates
$\thick(\add_{B}V)$ as a triangulated category. Since $\scr G^{\Phi}_{B}(V)=\oplus_{i\in\Phi}\E_{B}^{\Phi,i}(V)$, it is easy to see that
$\scr G^{\Phi}_{B}(V)$ is in the smallest subcategory of $\Kb{\proj{\scr G^{\Phi}_{B}(V)}}$ which is generated
by $\thick(\{\E_{B}^{\Phi,i}(V,\cpx{\bar{T}})\}_{i\in\Phi})$, where $V$ is in
$\add_{B}\cpx{\bar{T}}$. Therefore, $
\thick(\{\E_{B}^{\Phi,i}(V,\cpx{\bar{T}})\}_{i\in\Phi})=\Kb{\proj{\scr G^{\Phi}_{B}(V)}}$.

(3) The locally endomorphism ring of $\cpx{T}_i$ is isomorphic to $\scr G_A^{\Phi}(U)$.
By Lemma \ref{llx}, we have the following
$$
\Hom_{\Kb{\scr G^{\Phi}(V)}}(\E_{B}^{\Phi,i}(V,\cpx{\bar{T}}),\E_{B}^{\Phi,j}(V,\cpx{\bar{T}}))
\cong
\left\{
\begin{array}{cc}
\End(U), & if\ \  i\neq j;\\
\E(U)_{i,j},  & if\ i=j.
\end{array}
\right.
$$
Then by Theorem \ref{ll}, the locally $\Phi$-Beilinsion-Green algebras $\scr G^{\Phi}(A\oplus X)$
and $\scr G^{\Phi}(B\oplus \bar{G}(X))$ are derived equivalent.

\end{proof}

\subsection{Derived equivalences for locally finite Beilinson-Green algebras from graded derived equivalences of group graded algebras}

Construction of tilting complexes for group graded algebras was primarily motivated by the problem of finding reduction methods for Brou\'e's Abelian Defect Group Conjecture.

Let $k$ be a field and $G$ be a group (not necessarily finite). Suppose that $R=\bigoplus_{g\in G}R_g$ and $S=\bigoplus_{g\in G}S_g$ are $G$-graded $k$-algebras such that $R$ is $k$-flat. We denote by $R\Gr$ the category of $G$-graded $R$-modules, and by $R\gr$ the category of finitely generated $G$-graded $R$-modules.

The group $G$ acts on $G$-graded $R$-modules $M\in R\Gr$ by letting $M(g)=\bigoplus_{h\in G}M(g)_h$ be the $g$-suspension of $M$, where $M(g)_h=M_{hg}$ for all $g,h\in G$. If $R$ is strongly graded, then $G$ acts on $A$-modules $X\in A\Modc$ by conjugation $X\mapsto R_g\otimes_A X$. Note that $(R\otimes_A X)(g)$ is naturally isomorphic to  $R\otimes_A (R_g\otimes_A X)$ in $R\Gr$.

 A $G$-graded $(R,S)$-bimodule $M$ can be regarded as an  $R\otimes S^{\mathrm{op}}$-module graded by the $G\times G$-set $G\times G/\delta(G)$, where $\delta(G)$ is the diagonal subgroup of $G\times G$, with $1$-component $M_1$ a module over the diagonal subalgebra \[\Delta(R\otimes S^{\mathrm{op}}):=(R\otimes S^{\mathrm{op}})_{\delta(G)}=\bigoplus_{g\in G}R_g\otimes S_{g^{-1}}.\] If $R$ and $S$ are strongly graded, then $M$ and $(R\otimes S^{\mathrm{op}})\otimes_{\Delta(R\otimes S^{\mathrm{op}})}M_1$ are naturally isomorphic $G$-graded $(R,S)$-bimodules.

Recall that an object $\tilde T$ of $\D{\Gr{R}}$ is called a $G$-{\it graded tilting complex} if it satisfies the following conditions:
\begin{enumerate}
\item[(i)] $\tilde T\in R\Grperf$; this means that, regarded as a complex of $R$-modules,  $\tilde T\in R\perf$, that is $\tilde T$ is bounded, and its terms are finitely generated projective $R$-modules.
\item[(ii)] $\bigoplus_{g\in G}\Hom_{\D{\Gr{R}}}(\tilde T,\tilde T(g)[n])=0$ for $n\neq0$.
\item[(iii)] $\mathrm{add}\{\tilde T(g)\mid g\in G\}$ generates $\Grperf{R}$ as a triangulated category.
\end{enumerate}

The following result was proved in \cite[Theorem 2.4]{Mar} and was reproved \cite[Theorem 2.4]{MarP}, based on Keller's approach \cite{Keller1994}, but note that the assumption that $G$ is finite is not needed.

\begin{thm}\label{gd} The following statements are equivalent:

$(1)$ There is a $G$-graded tilting complex $\tilde T\in\D{\Gr{R}}$ and an isomorphism $S\simeq\End_{\D{R}}(\tilde T)^{\mathrm{op}}$ of $G$-graded algebras.

$(2)$ There is a complex $\tilde U$ of $G$-graded $(R,S)$-bimodules such that the functor
\[\tilde U\otimesL_S-: \D{S}\lra \D{R}\]
is an equivalence.

$(3)$ There are equivalences
\[F: \D{R}\lra \D{S} \quad and \quad F^{\mathrm{gr}}: \D{\Gr{R}}\lra \D{\Gr{S}}\]
of triangulated categories such that $F^{gr}$ is $G$-graded functor and the diagram
\[\xymatrix@M=2mm{ \D{\Gr{R}}\ar[d]^{\mathscr{U}} \ar[r]^{F^{\mathrm{gr}}} & \D{\Gr{S}}\ar[d]^{\mathscr{U}}  \\  \D{R}\ar[r]^{F} & \D{S},}\]
is commutative.

$(4)$ There are equivalences
\[
F_{\mathrm{perf}}: \D{\perf{R}}\lra \D{\perf{S}} \quad and \quad F_{\mathrm{perf}}^{\mathrm{gr}}: \D{\Grperf{R}}\lra \D{\Grperf{S}}
\]
of triangulated categories such that $F_{\mathrm{perf}}^{\mathrm{gr}}$ is $G$-graded functor and
$\mathscr{U}\circ F_{\mathrm{perf}}^{\mathrm{gr}}=F_{\mathrm{perf}}\circ\mathscr{U}$.

{\rm(5) (provided that $R$ and $S$ are strongly graded)}  There are (bounded) complexes $U$ of $\Delta(R\otimes S^{\mathrm{op}})$ modules and  $V$ of $\Delta(S\otimes R^{\mathrm{op}})$-modules, and isomorphisms $U\otimesL_{B} V\simeq A$ in $\mathscr{D}^b(\Delta(R\otimes R^{\mathrm{op}}))$ and $V\otimesL_{A}U\simeq B$ in $\mathscr{D}^b(\Delta(S\otimes S^{\mathrm{op}}))$.
\end{thm}\label{mac}

Associated a $G$-graded $k$-algebra $A$, there is a Beilinson-Green algebra $\overline{A}$ defined as a $G\times G$-matrix algebra with $(\overline{A}_{gh})_{g,h\in G}$, where
$\overline{A}_{gh}=A_{gh^{-1}}$. If $M=\oplus_{g\in G}M_g$ is a $G$-graded $A$-$B$-bimodule, then $\bar{M}=(M_{gh^{-1}})_{g,h\in G}$ is a $\overline{A}$-$\overline{B}$-bimodule.

We can get derived equivalences between the Beilinson-Green algebras from $G$-graded derived equivalences between $G$-graded algebras $A$ and $B$ for a group $G$.

Hence we have the following theorem.

\begin{thm}\label{GBG} Suppose that there is a $G$-graded derived equivalence between $G$-graded algebras $A$ and $B$, then there is a derived equivalence between the Beilinson-Green algebras $\overline{A}$ and $\overline{B}$.
\end{thm}

\begin{proof} Suppose that there is a $G$-graded derived equivalence between $G$-graded algebras $A$ and $B$, by Lemma \ref{gd}, there is a $G$-graded tilting complex $\cpx{T}$ over $A$, and
$B\simeq\End_{\D{R}}(\cpx{ T})^{\mathrm{op}}$ of $G$-graded algebras. That is,
$$
B=\oplus_{g\in G}B_g\simeq\End_{\D{R}}(\cpx{ T})^{\mathrm{op}}=\oplus_{g\in G}\Hom_{\D{R}}(\cpx{ T},\cpx{T}(g))
$$
We claim that $\cpx{T}(g)$ is a locally tilting family over $\overline{A}$.

Indeed, we have the following facts

(1) $\Hom(\cpx{T}(h),\cpx{T}(g)[i])_{g,h\in G}=\Hom(\cpx{T},\cpx{T}(gh^{-1})[i])_{g,h\in G}\cong
\left\{
\begin{array}{cc}
B_{gh^{-1}}=\overline{B}_{gh}, & if\ \  i\neq 0;\\
0,  & if\ i=0.
\end{array}
\right. $

(2) $\thick\{\cpx{T}(g)\}_{g\in G}$ generates $\Grperf{\overline{A}}$ as a triangulated category.
Then by Theorem \ref{ll}, the Beilinsion-Green algebras $\overline{A}$
and $\overline{B}$ are derived equivalent.
\end{proof}

\begin{rem}  In \cite{CO}, the associated Beilinsion-Green algebras are smash product of $G$-graded algebras.
Asashiba \cite{Asa} obtains derived equivalence between smash products from derived equivalences for $G$-graded categories (not necessary graded derived equivalences) under some condition.
\end{rem}

\section{Examples}\label{example}

In this section, we  give an example to illustrate our Theorem \ref{Theorem-ghost}.

Throughout this section, we assume that $A$ is a self-injective Artin algebra, and write, for $n\in \mathbb{N}$, $$\mathcal{D}_{n}:=\add\big(\bigoplus_{i=0}^n\Sigma^{-i}A\big)$$
in $\K{\modc{A}}$. For simplicity, we will write $\mathscr{K}$ for $\K{A\modc}$.  The $i$-th cohomology of a complex $\cpx{C}$ in $\mathscr{K}$ is denoted by $H^i(\cpx{C})$. For an $A$-module $X$, we denote by $X^*$ the right $A$-module $\Hom_A(X, A)$. Let $D$ be the usual duality, and let $\nu_A:=D\Hom_A(-, A)$ be the Nakayama functor. When $P$ is a finitely generated projective $A$-module, there is a natural isomorphism $\Hom_A(P, -)\cong D\Hom_A(-, \nu_AP)$  which can be obtained by applying $D$ to the isomorphism in \cite[p.41, Proposition 4.4(b)]{Auslander1995}. This further induces an isomorphism $\mathscr{K}(\cpx{P}, -)\cong D\mathscr{K}(-,\nu_A\cpx{P})$ {for all bounded complexes} $\cpx{P}$ of finitely generated projective $A$-modules. This will be frequently used in this section.

\begin{lem} \cite[Lemma 5.1]{CH} With notation as above, {the ideals
$\cogh_{\mathcal{D}_n}$ and $\gh_{\mathcal{D}_n}$ in $\mathscr{K}$ are equal, and both of them} consist of morphisms  $\cpx{\alpha}$ such that $H^i(\cpx{\alpha})=0$ for all $0\leq i\leq n$.
\label{lemma-ann-ghost}
\end{lem}

 {\em Remark. } This lemma implies $\Fcogh_{\mathcal{D}_n}$ and $\Fgh_{\mathcal{D}_n}$ also coincide. Denote by $\mathcal{G}$ the ideal of $\mathscr{K}$ consisting of ghost maps.
 Let $\cpx{X}$ be a complex of $A$-modules. Then $\cogh_{\mathcal{D}_n}(\cpx{X})=\gh_{\mathcal{D}_n}(\cpx{X})=\mathcal{G}(\cpx{X})$. Let $\mathcal{G}_{\mathcal{D}_n}:=\mathcal{G}\cap \mathcal{F}_{\mathcal{D}_n}$. {Then $\mathcal{G}_{\mathcal{D}_n}(\cpx{X})=\Fcogh_{\mathcal{D}_n}(\cpx{X})=\Fgh_{\mathcal{D}_n}(\cpx{X})$.}

In the following, we give a concrete example.

\medskip
\begin{ex} Let $k$ be a field, and let $A=k[x,y]/(x^n-y^s,xy)$. Suppose  that  $\cpx{X}$ is the complex
$$0\lra A\xra{\cdot x} A\lra 0$$
with the left $A$ in degree zero.  The endomorphism algebra $\End_{\mathscr{K}/\mathcal{G}}(\cpx{X}\oplus A\oplus \Sigma^{-1}A)$ is denoted by $\Lambda_x$.
The construction above gives a $\mathcal{D}_1$-split triangle
$$\cpx{Y}\lra A\oplus \Sigma^{-1}A\lra \cpx{X}\lra \Sigma\cpx{Y}$$
in $\mathscr{K}$.  An easy calculation shows that $\cpx{Y}$ is isomorphic in $\K{A\modc}$ to the complex
$$0\lra A\xra{\cdot y}A\lra 0.$$
Then the algebras $\Lambda_x=\End_{\mathscr{K}/\mathcal{G}}(\cpx{X}\oplus A\oplus \Sigma^{-1}A)$ and $\End_{\mathscr{K}/\mathcal{G}}(\cpx{Y}\oplus A\oplus \Sigma^{-1}A)$ are derived equivalent.  Note that $\End_{\mathscr{K}/\mathcal{G}}(\cpx{Y}\oplus A\oplus \Sigma^{-1}A)$ is just $\Lambda_y$. That is, the algebra $\Lambda_x$ is derived equivalent to $\Lambda_y$.

 To describe $\Lambda_x$ in terms of  quivers with relations,
we  give some morphisms in $\Lambda_x$.
$$\alpha_1: \xymatrix{0\ar[d]\ar[r] & A\ar[d]^{\cdot x}\\
0\ar[r] & A}\quad  \alpha_2:\xymatrix{A\ar[d]_{\cdot x}\ar[r] & 0\ar[d]\\
A\ar[r] & 0},  $$

$$\beta_1:{}\xymatrix{0\ar[r]\ar[d]&A\ar[d]^{id}\\A\ar[r]^{\cdot x}&A
}, \quad\beta_2:{} \xymatrix{A\ar[r]^{\cdot x}\ar[d]^{id}&A\ar[d]\\A\ar[r]&0},\quad \beta_3:{}
\xymatrix{A\ar[r]^{\cdot x}\ar[d]&A\ar[d]^{\cdot y}\\
0\ar[r]&A},
 \quad \beta_4:{}\xymatrix{A\ar[r]\ar[d]^{\cdot x^y}&0\ar[d]\\A\ar[r]^{\cdot x}&A}.$$

It is easy to see that the above morphisms generate the Jacobson radical of $\Lambda_x$.

\medskip
In this case, the above morphisms are irreducible in $\add(\cpx{X}\oplus A\oplus \Sigma^{-1}A)$  and the algebra $\Lambda_x$ is given by the following quiver with relations.
$$\begin{array}{c}
\xymatrix{{\bullet}\ar@(lu,ld)_(.4){\alpha_1}\ar@<2pt>[r]^{\beta_1}&\bullet\ar@<2pt>[l]^{\beta_3}^(1){1}\ar@<2pt>[r]^{\beta_2}&\bullet\ar@<2pt>[l]^{\beta_4}^(1){2}^(0){3}\ar@(ru,rd)^{\alpha_2}\\
}\\
\alpha_1\beta_1=\beta_3\alpha_1=\alpha_2\beta_4=\beta_1\beta_2=\beta_4\beta_3=\beta_2\alpha_2=0  \\
 \alpha_1^{n}=(\beta_1\beta_3)^s, \,
  \alpha_2^{n}=(\beta_4\beta_2)^s, \\
  (\beta_3\beta_1)^s+(\beta_2\beta_4)^s=0.
\end{array}$$
\end{ex}

\begin{ex}
Let $A=k[x]/(x^n)$. Then $A$ is a representation-finite self-injective algebra.
Denote the indecomposable $A$-module by
$$X_r:=k[x]/(x^r)$$ for $r=1,2,\cdots,n$.
 We thus shows that $\scr G^{\Phi}_{A}(A\oplus X_r)$ and $\scr G^{\Phi}_{A}(A\oplus X_{n-r})$ are derived equivalent with $\Omega(X_r)=X_{n-r}$.

In the following, let $\Phi=\{0,1,2\}$ and $r=1$.
Then we describe $\Phi$-Beilinson-Green algebras $\scr G^{\Phi}_{A}(A\oplus X_1)$ and $\scr G^{\Phi}_{A}(A\oplus X_{n-1})$ in terms of quivers with relations as follows
$$
\begin{array}{c}
 \xymatrix{
 1 \ar@(dl,ul)[]^{\alpha} \ar@/^/[rr]| \beta & & 2 \ar@/^/[ll]|\gamma\ar[d]_{\delta}\ar@/^/[dd]^{\delta''}\\
  3 \ar@(dl,ul)[]^{\alpha'} \ar@/^/[rr]| {\beta'} & & 4 \ar@/^/[ll]|{\gamma'} \ar[d]_{\delta'} \\
  5 \ar@(dl,ul)[]^{\alpha''} \ar@/^/[rr]| {\beta''} & & 6 \ar@/^/[ll]|{\gamma''},  }\\
  \gamma\beta=\gamma\alpha=\alpha\beta=\beta\gamma-\alpha^{n-1}=\beta\delta=\delta\gamma'=0,\\
  \gamma'\beta'=\gamma'\alpha'=\alpha'\beta'=\beta'\gamma'-\alpha'^{n-1}=\beta'\delta'=\delta'\gamma''=0,\\
  \gamma''\beta''=\gamma''\alpha''=\alpha''\beta''=\beta''\gamma''-\alpha''^{n-1}=\delta\delta'-\delta''=0.
  \end{array}
$$
and
$$
\begin{array}{c}
  \xymatrix{
   1\ar@/^/[rr]| x & & 2 \ar@/^/[ll]| y\ar[d]_{\eta}\ar@/^/[dd]^{\eta''}\\
  3 \ar@/^/[rr]|{x'} & & 4 \ar@/^/[ll]| {y'} \ar[d]_{\eta'}  \\
  5 \ar@/^/[rr]|{x''} & & 6 \ar@/^/[ll]| {y''}  } \\
 (xy)^{n-1}= (x'y')^{n-1}=(x''y'')^{n-1}=x\eta=\eta y'=x'\eta'=\eta' y''=\eta\eta'-\eta''=0.
 \end{array}
$$
We can calculate that $\dk(\scr G^{\Phi}_{A}(A\oplus X_1))=3n+12, \dk(\scr G^{\Phi}_{A}(A\oplus X_{n-1}))=12n-6$.
\end{ex}

\noindent{\bf Acknowledgements.} Shengyong Pan is funded by China Scholarship Council. He thanks his mother for her 70th birthday with thankfulness, love and encouragements. The revised version of this paper was done during a visit of Shengyong Pan to the University of Edinburgh, he would like to thank Professor Susan J. Sierra for her hospitality and useful discussions. We thank the anonymous referee for valualbe comments and suggestions, especially for the proof of Theorem \ref{4.1}.

{\footnotesize


\begin{thebibliography}{99}
{\small


\bibitem{Asa}
Asashiba,~H.,
\emph{Derived equivalences and smash products},
Proc. of the 49th Symposium on Ring Theory and Representation Theory, 12--17, Symp. Ring Theory Represent. Theory Organ. Comm., Shimane, 2017.

\bibitem{Aihara2012}
Aihara,~T. and Iyama,~O.,
   \emph{Silting mutation in triangulated categories}, J. Lond. Math. Soc. (2)
  85 (2012), 633--668.

\bibitem{Assem2006}
Assem, I., Simson, D. and Skowro{\'n}ski, A., \emph{Elements of the
  representation theory of associative algebras. {V}ol. 1}, London
  Mathematical Society Student Texts 65, Cambridge University Press,
  Cambridge, 2006.


\bibitem{Auslander1995}
Auslander,~M., Reiten,I. and Smal{\o},S.~O., \emph{Representation theory of {A}rtin
  algebras}, Cambridge Studies in Advanced Mathematics 36,
  Cambridge University Press, Cambridge, 1995.




\bibitem{Auslander1980}
Auslander,~M. and Smal{\o},S.~O., \emph{Preprojective modules over Artin algebras},  J.
  Algebra 66 (1980), 61--122.



 \bibitem{Bei}
 Beilinson, A.~A.,
  \emph{Coherent sheaves on $\mathbb{P}^n$ and problems in linear algebra}, Funct. Anal. Appl. 12 (1978), 214--216.

\bibitem{BA21}
Bennett-Tennenhaus, R. and Shah, A.,
 \emph{ Transport of structure in higher homological algebra}, J. Algebra 574 (2021), 514--549.

\bibitem{BP}
Br\"{u}stle, T. and Pan, S. Y., \emph{Transfer of derived equivalences from subalgebras to endomorphism algebras},  J. Algebra Appl. 15(6 )(2016), 1650100 (10 pages).

\bibitem{BD}
Brundan, J. and Davidson, N.,
\emph{Categorical actions and crystals}, Contemp. Math. 684 (2017), 116--159.



  \bibitem{Buan2006}
Buan, A. B., Marsh, R., Reineke, M., Reiten, I.
  and Todorov, G., \emph{Tilting theory and cluster
  combinatorics},  Adv. Math. 204 (2006), 572--618.

\bibitem{Bu}
T. B\"{u}ler, \emph{T., Exact categories},
Expo. Math. 28 (2010), 1--69.

\bibitem{CX}
Chen, H. X. and Xi, C. C.,
 \emph{Dominant dimensions, derived equivalenences and tilting modules}, Isr. J. Math. 215 (2016), 349--395.



\bibitem{CHen}Chen, Q. H.,
\emph{Derived equivalences of repetitive algebras}, Adv. Math.(China) 37 (2008), 189--196.



 \bibitem{C}
 Chen, X. W., \emph{Graded self-injective algebras "are" trivial extensions}, J. Algebra 322 (2009), 2601--2606.


\bibitem{Chen2013}
Chen, Y. P., \emph{Derived equivalences in {$n$}-angulated categories}, Algebr. Represent Theor 16 (2013), 1661--1684.

\bibitem{Chen2014}
Chen, Y. P., \emph{Derived equivalences between subrings}, Comm. Algebra 42 (2013), 4055--4065.


\bibitem{CH}
Chen, Y. P. and Hu,W., \emph{Approximations, ghosts and derived equivalences}, Proc. Royal Soc. Edinburgh 150 (2020), 813--840.



\bibitem{D}
Dugas, A., \emph{A construction of derived equivalent pairs of symmetric algebras}, Proc. Amer. Math Soc. 143 (2015), 2281--2300.


\bibitem{E}
Erdmann, K.,
\emph{Blocks of tame representation type and related algebras,} Lecture Notes in
Mathematics, vol. 1428 (Springer, 1990).


\bibitem{G}
Green,E. L.,
\emph{A criterion for relative global dimension zero with applications to graded rings},
J. Algebra 34 (1975), 130--135.


\bibitem{GH}
Green,E. L. and Happel, D.,
 \emph{Grading and derived categories},
Algebr. Represent Theor 14 (2011), 497--513.


\bibitem{GHS}
Green,E. L., J. R.~Hunton, J. R. and Snashall, N.,
\emph{Coverings, the graded center and Hochschild cohomology},
J. Pure Algebra 212 (2008), 2691--2706.

\bibitem{Geiss2013}
Geiss, C., Keller, B. and Oppermann, S., \emph{$n$-angulated categories}, J.
  Reine Angew. Math. 675 (2013), 101--120.


\bibitem{Ha1}
Happel, D., \emph{Triangulated categories in the
  representation theory of finite-dimensional algebras}, London
  Mathematical Society Lecture Note Series 119, Cambridge University
  Press, Cambridge, 1988.

\bibitem{Happel1998b}

Happel, D. and Unger, L.,
  \emph{Complements and the generalized {N}akayama conjecture},  in \emph{Algebras and
  modules, {II} ({G}eiranger, 1996)}, CMS Conf. Proc. 24, Amer.
  Math. Soc., Providence, RI, 1998, pp.~293--310.


  \bibitem{Happel1989c}
 Happel, D. and Unger, L.,
  \emph{ Almost complete tilting modules}, Proc. Amer. Math. Soc. 107 (1989), 603--610.


\bibitem{HLN}
Herschend, M., Liu, Y. and Nakaoka, H.,
\emph{n-exangulated categories (I): Definitions and fundamental properties}, J. Algebra 570 (2021), 531--586.




\bibitem{Hoshino2003}
Hoshino, M. and Kato, Y.,
\emph{An elementary construction of tilting complexes}, J. Pure Appl. Algebra 177 (2003), 159--175.




 \bibitem{Hu2013a}
 Hu, W., Koenig, S. and Xi, C. C.,
 \emph{Derived equivalences from cohomological approximations and mutations of {$\Phi$}-{Y}oneda algebras}, Proc. Roy.
  Soc. Edinburgh Sect. A 143 (2013), 589--629.


 \bibitem{Hu2017}
 Hu, W. and Pan, S. Y.,
\emph{Stable functors of derived equivalences and Gorenstein projective modules},  Math. Narchr.
  290 (2017), 1512--1530.

\bibitem{Hu2010}
Hu, W. and Xi, C. C., \emph{Derived equivalences and stable equivalences of Morita type I},
Nagoya Math. J. 200 (2010), 107--152.

\bibitem{Hu2011}
Hu, W. and Xi, C. C.,
\emph{$\mathcal{D}$-split sequences and derived equivalences},  Adv. Math.
  227 (2011), 292--318.


\bibitem{Hu2013}
Hu, W. and Xi, C. C.,
\emph{Derived equivalences for {$\Phi$}-{A}uslander-{Y}oneda algebras},
Trans. Amer. Math. Soc. 365 (2013), 5681--5711.



\bibitem{Huwa}
Hughes, D. and Waschb\"{u}sch, J.,
\emph{Trivial extensions of tilted algebras},  Proc. London Math. Soc.
  46 (1983), 347--364.

  \bibitem{Iyama2014a}
  Iyama, O. and Wemyss, M.,
 \emph{Maximal modifications and {A}uslander-{R}eiten duality for non-isolated
  singularities},  Invent. Math. 197 (2014), 521--586.


\bibitem{Ka}
Karoubi, M., \emph{Alg\`{e}bres de Clifford et K-th\'{e}orie},  Ann. Sci. \'{E}cole Norm. Sup. 4 (1968), 161--270.



  \bibitem{Keller1994}
  Keller, B.,
\emph{Deriving {DG} categories}, Ann. Sci.
  {\'E}cole Norm. Sup. 27 (1994), 63--102.

\bibitem{Keller2005a}
  Keller, B., \emph{On triangulated orbit categories},  Doc. Math. 10
  (2005), 551--581.


\bibitem{Keller2006}% incollection
Keller, B.,
\emph{On differential graded categories},
 in: International {C}ongress of {M}athematicians. {V}ol. {II},  (Eur. Math.
  Soc., Z{\"u}rich, 2006),  pp.\,151--190.


\bibitem{Krause2015}
Krause, H., \emph{Krull-Schmidt categories and projective covers}, Expo. Math. 33
  (2015), 535--549.


\bibitem{Ladkani2010}
Ladkani, S.,
\emph{Perverse equivalences, {BB}-tilting,
  mutations and applications},  Preprint, (2010).



\bibitem{Ma}
Madsen, D.,
\emph{Ext-algebras and derived equivalences}, Colloq. Math. 104 (2006), no. 1, 113--140.


\bibitem{Mar}
Marcus, A.,
\emph{Tilting complexes for group graded algebras}, J. Group Theory 6 (2003), 175--193.




\bibitem{M3}
Marcus, A.,
\emph{Equivalences induced by graded bimodules}, Comm. Algebra 26 (1998), 713--731.

\bibitem{MarP}
Marcus, A. and Pan, S. Y.,
\emph{Tilting complexes for group graded self-injective algebras}, Tsukuba J. Math. 43 (2019), 211--222.


\bibitem{Mori}

Mori, I., \emph{B-construction and C-construction},
Comm. Algebra 41 (2013), 2071--2091.


\bibitem{Mu}
Muller, G.,
\emph{The Beilinson equivalence for differential operators},
J. Pure Applied Algebra 214 (2010), 2124--2143.


\bibitem{CO}
Nastasescu, C. and Van Oystaeyen, F., {\it Methods of Graded Rings}, Lecture Notes in Mathematics (1836).

\bibitem{N}
Neeman, A.,
{\it Triangulated Categories}. Annals of Mathematics Studies \textbf{148}, Princeton University Press, Princeton and Oxford, 2001.

\bibitem{P1}
Pan, S. Y.,
\emph{Derived equivalences for Cohen--Macaulay Auslander algebras}, J. Appl. and Pure Algebra, 216 (2012), 355--363.

\bibitem{P2}
Pan, S. Y.,
\emph{Derived equivalences for $\Phi$-Cohen--Macaulay Auslander--Yoneda algebras}, Algebr. Represent. Theory,  17 (2014), 885--903.

\bibitem{P3}
Pan, S. Y.,
\emph{Stable equivalences of Morita type for $\Phi$-Beilinson-Green algebras}, Math. Nachr. 294 (2021), no. 5, 977--996.

\bibitem{PZ}
Pan, S. Y. and Peng, Z.,
\emph{A note on derived equivalences for $\Phi$-Green algebras}, Algebr. Represent. Theory
  17 (2014), 1707--1720.

\bibitem{PX}
Pan, S. Y. and Xi, C. C.,
\emph{Finiteness of finitistic dimension is invariant under derived equivalences}, J. Algebra 322 (2009), 21--24.

\bibitem{Pe}
Peng, Z., \emph{Derived equivalences between $\Phi$-Green algebras and $n$-homological ring epimorphisms}, Ph.D. thesis, 2013.

\bibitem{Ri1}
Rickard, J., \emph{Morita theory for derived categories}, J. London Math. Soc. 39 (1989), 436--456.

\bibitem{Ri2}
Rickard, J.,
\emph{Derived categories and stable equivalence}, J. Pure Appl. Algebra 61 (1989), 303--317.

\bibitem{Ri3}
Rickard, J.,
\emph{Derived equivalences as derived functors}, J. London Math. Soc. 43 (1991), 37--48.

\bibitem{RZ}
Ringel, C. M. and Zhang, P.,
\emph{Gorenstein-projective and semi-Gorenstein-projective modules},
Algebra Number Theory (2019), in press, arXiv:1808.01809v3.



\bibitem{Rudakov1990}
Rudakov, A.~N., \emph{Exceptional collections, mutations and
  helices},  in Helices and vector bundles, London Math. Soc.
  Lecture Note Ser. \textbf{148}, Cambridge Univ. Press, Cambridge, 1990,
  pp.~1--6.

\bibitem{Vandenbergh2004}
Van den Bergh, M.,
\emph{Non-commutative crepant resolutions}, in The legacy of Niels Henrik Abel, Springer-Verlag, Berlin (2004), 749--770.


\bibitem{Wis}
Wisbauer, R., \emph{Foundations of Module and Ring Theory, Algebra, Logic and Applications}, Vol. 3, Gordon and Breach, New York, 1991.

\bibitem{Xi}
 Xi, C.C., \emph{On the finitistic dimension conjecture, I: related to representation-finite
algebras}, J. Pure Appl. Algebra 193(2004), 287--305.


}


\end{thebibliography}
\end{document}